\documentclass[12pt,english]{paper}
\usepackage[T1]{fontenc}
\usepackage[latin9]{inputenc}
\usepackage{babel}
\usepackage{verbatim}
\usepackage{units}
\usepackage{amsthm}
\usepackage{amsmath}
\usepackage{amssymb}
\usepackage{makeidx}
\makeindex
\usepackage{graphicx}
\usepackage{esint}
\usepackage{xargs}[2008/03/08]
\usepackage[unicode=true,pdfusetitle,
 bookmarks=true,bookmarksnumbered=false,bookmarksopen=false,
 breaklinks=false,pdfborder={0 0 1},backref=false,colorlinks=false]
 {hyperref}

\makeatletter
\numberwithin{equation}{section}
\numberwithin{figure}{section}
\theoremstyle{plain}
\newtheorem{thm}{\protect\theoremname}
  \theoremstyle{remark}
  \newtheorem*{rem*}{\protect\remarkname}
  \theoremstyle{plain}
  \newtheorem*{thm*}{\protect\theoremname}
 \theoremstyle{definition}
 \newtheorem*{defn*}{\protect\definitionname}
  \theoremstyle{definition}
  \newtheorem{defn}[thm]{\protect\definitionname}
  \theoremstyle{plain}
  \newtheorem{lem}[thm]{\protect\lemmaname}
  \theoremstyle{remark}
  \newtheorem{claim}[thm]{\protect\claimname}
  \theoremstyle{definition}
  \newtheorem{problem}[thm]{\protect\problemname}
  \theoremstyle{plain}
  \newtheorem{prop}[thm]{\protect\propositionname}
  \theoremstyle{plain}
  \newtheorem{cor}[thm]{\protect\corollaryname}
  \theoremstyle{plain}
  \newtheorem{conjecture}[thm]{\protect\conjecturename}

\usepackage{geometry}

\usepackage{ifpdf} % part of the hyperref bundle
\ifpdf % if pdflatex is used

 \IfFileExists{lmodern.sty}{\usepackage{lmodern}}{}

\fi % end if pdflatex is used

\makeatother

  \providecommand{\claimname}{Claim}
  \providecommand{\conjecturename}{Conjecture}
  \providecommand{\corollaryname}{Corollary}
  \providecommand{\definitionname}{Definition}
  \providecommand{\lemmaname}{Lemma}
  \providecommand{\problemname}{Problem}
  \providecommand{\propositionname}{Proposition}
  \providecommand{\remarkname}{Remark}
  \providecommand{\theoremname}{Theorem}
\providecommand{\theoremname}{Theorem}

\begin{document}
\global\long\def\bbc{\mathbb{C}}

\global\long\def\bbd{\mathbb{D}}

\global\long\def\bbn{\mathbb{N}}

\global\long\def\bbq{\mathbb{Q}}

\global\long\def\bbr{\mathbb{R}}

\global\long\def\bbt{\mathbb{T}}

\global\long\def\bbz{\mathbb{Z}}

\global\long\def\calD{\mathcal{D}}

\global\long\def\calE{\mathcal{E}}

\global\long\def\calH{\mathcal{H}}

\global\long\def\calO{\mathcal{O}}

\global\long\def\calP{\mathcal{P}}

\global\long\def\epsym{\varepsilon}

%%%%

\global\long\def\intsy{\int}

\global\long\def\O{O}

\global\long\def\eqdef{\overset{{\rm {def}}}{=}}

\global\long\def\indf{{1\hskip-2.5pt{\rm l}}}

\global\long\def\opspan{\operatorname{span}}

\global\long\def\d#1{\,\mathrm{d}#1}

\global\long\def\dd{\mathrm{d}}

\newcommandx\norm[2][usedefault, addprefix=\global, 1=]{\left\Vert #1\right\Vert _{#2}}

\newcommandx\normp[3][usedefault, addprefix=\global, 1=, 2=]{\left\Vert #1\right\Vert _{#2}^{#3}}

%%%%

\global\long\def\Pr{\calP}

% need to fix this

\global\long\def\pr#1{\Pr\left(#1\right)}

\global\long\def\Ex{\calE}

%%%%

\global\long\def\disk#1{#1 \bbd}

\global\long\def\cir#1{#1 \bbt}

%%%%

\global\long\def\probspace{\Omega}

\global\long\def\shiftspace{\bbt}

\global\long\def\measspace{Q}

\global\long\def\measfuncs{L^{2}\left(\measspace\right)}

\global\long\def\RFfuncs{L_{{\tt RF}}^{2}}

\global\long\def\HLMaxF{\mathfrak{M}}

\global\long\def\Exploc{\operatorname{Exp}_{{\tt loc}}}

\global\long\def\errfunc{\varepsilon}

\global\long\def\zrf{s}

%%%%

\global\long\def\excpset{E}

\global\long\def\dprime{{\prime\prime}}

\global\long\def\vol#1{\mathrm{vol}\left(#1\right)}

\global\long\def\volcn#1{\mathrm{vol_{\bbc^{n}}}\left(#1\right)}

\global\long\def\volrn#1#2{\mathrm{vol_{\bbr^{#1}}}\left(#2\right)}

\global\long\def\Var#1{\mathrm{Var}\left(#1\right)}

\global\long\def\Cov#1{\mathrm{Cov}\left(#1\right)}

\global\long\def\ind#1#2{\mathbf{1}_{#1}\left(#2\right)}

\global\long\def\Nalt#1{\widetilde{N}\left(#1\right)}

\global\long\def\Naltd#1{\widetilde{N}_{\delta}\left(#1\right)}

\global\long\def\Salt#1{\widetilde{S}\left(#1\right)}

\thispagestyle{empty}

\begin{center}
\begin{figure}
\centering{}\hspace*{-2cm}\includegraphics[width=19cm]{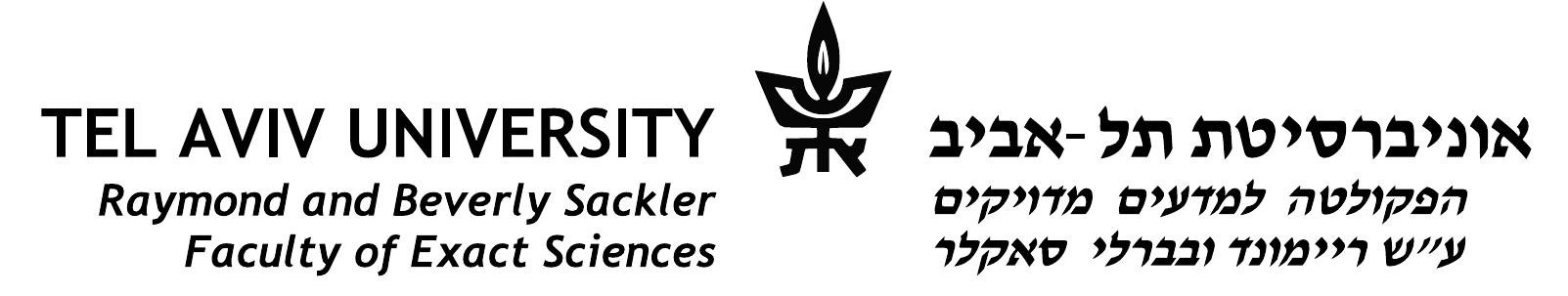}
\end{figure}
$\,$\linebreak{}
\linebreak{}
\linebreak{}
\linebreak{}
\linebreak{}

\par\end{center}

\begin{center}
\textbf{\huge{Topics in the Value Distribution}}\\
\textbf{\huge{of Random Analytic Functions}}{\LARGE{}}\linebreak{}
{\LARGE{}}\linebreak{}
{\LARGE{}}\linebreak{}
{\LARGE{}}\linebreak{}
{\LARGE{}}\linebreak{}

\par\end{center}{\LARGE \par}

\begin{center}
\textbf{Thesis submitted for the degree \textquotedblleft{}Doctor
of Philosophy\textquotedblright{}}\\

\par\end{center}

\begin{center}
by\\

\par\end{center}

\begin{center}
\textbf{\large{Alon Nishry}}\linebreak{}
\linebreak{}
\linebreak{}
\linebreak{}

\par\end{center}

\begin{center}
This Study was Supervised by
\par\end{center}

\begin{center}
\textbf{\large{Professor Mikhail Sodin}}\linebreak{}
\linebreak{}
\linebreak{}
Submitted to the senate of Tel Aviv University
\par\end{center}

\begin{center}
\textbf{September 2013}
\par\end{center}

\newpage{}

\thispagestyle{empty}
\[
\,
\]
\newpage{}

\thispagestyle{empty}\textbf{\large{Abstract}}\\
\textbf{\large{}}\\
This thesis deals with the behavior of random analytic functions.
Our approach to this subject is the classical one, that is, we study
random Taylor series with independent coefficients. In particular,
we are interested in the value distribution of these functions, this
basically means counting the number of certain values that these functions
take, in some domains. Special attention is devoted to studying the
relation between the number of zeros of these functions and the growth
of their coefficients.

The contents of this thesis can be partitioned into two parts: value
distribution of Rademacher Taylor series and asymptotics of the hole
probability for Gaussian entire functions. A certain philosophy of
this work, which is relevant to both parts, is to avoid making any
assumptions on the magnitude (variance) of the Taylor coefficients,
while studying a particular distribution of the phases. In the first
part, we study Taylor series with random signs and in the second we
study coefficients with a complex Gaussian distribution. As it turns
out, in both cases it is possible to prove general results without
making any assumptions on the `regularity' of the growth of the coefficients.

We start this thesis with a study of the properties of Rademacher
Fourier series, the main result states that the logarithm of such
series is integrable (to any power). The methods that are used are
related to those that were originally developed by F. Nazarov to tackle
similar problems for lacunary Fourier series. This result then enables
us to study in detail the behavior of Rademacher Taylor series. First
we give an answer to an old question of J.-P. Kahane, concerning the
range of Rademacher Taylor series in the unit disk. Then we prove
a theorem that relates the number of zeros of Rademacher entire functions
in large disks to the rate of growth of their coefficients. Finally,
we prove that the zeros of these entire functions are equidistributed
in sectors centered at the origin. This improves on some classical
works of Littlewood and Offord.

The main motivation of the second part is to study `exceptional' behavior
of Gaussian entire functions (given by a Taylor series). It is well
known that these functions are expected to have many zeros in large
disks around the origin. We are interested in characterizing the logarithmic
asymptotics of the probability of a `hole event', that is, when the
function has no zeros at all in such disks. We give a full solution
for this problem, by introducing a new function which measures the
relation between the rate of decay of the `hole probability' and the
magnitude of the coefficients of the Gaussian Taylor series.

\newpage{}

\thispagestyle{empty}$\,$\newpage{}

\thispagestyle{empty}\textbf{\large{Acknowledgments}}\\
\textbf{\large{}}\\
First of all I would like to thank my advisor Professor Mikhail Sodin.
I was very fortunate to have Misha as my advisor, and during my studies
I began to fully appreciate it. Our many conversations, whether they
were mathematical or not, were essential for my progress. His direction,
suggestions and constant encouragement were very important for my
success. His attention for details, especially for my writing, has
definitely contributed to this work. I would also like to mention
his willingness to help me in all kinds of problems and questions
as well as my appreciation for the very generous support that I was
given. It was by chance that I became Misha's student, and I truly
feel lucky that things turned out the way they did.

Professor Fedor Nazarov had many essential contributions to this work
and I would also like to thank him for his wonderful hospitality,
while I was visiting him in Kent State University. Doctor Andrei Iacob
carefully read this manuscript and made useful remarks.

I would like to thank (soon to be Doctor) Zemer Kosloff for being
the great friend that he is, ever since we met during our master's
degree. He definitely made these years much more interesting, fun
and memorable. The rest of the ``Room 005 gang'' (and outside members)
also deserve many thanks, we had great time together. Finally, I would
like to thank my family, for all the support, their help and care.\newpage{}

\thispagestyle{empty}
\[
\,
\]
\newpage{}

\begin{center}
\textbf{\huge{Topics in the Value Distribution of Random Analytic
Functions}}
\par\end{center}{\huge \par}

\thispagestyle{empty}

\setcounter{page}{1}

\markright{Value Distribution of Random Analytic Functions}

\tableofcontents{}

\newpage{}

\part{Introduction}

This work is concerned with `typical' and `exceptional' behavior of
analytic functions in general, and entire functions in particular.
Here the words typical and exceptional are interpreted in the probabilistic
sense. Our main objects of study are random analytic functions, more
precisely, random Taylor series. This kind of functions have a long
history, as they first appear at the end of the 19th century. It could
be argued that this notion first began with the comments of Borel
in 1896 on ``General Taylor series'', and it concerned the natural
boundary of Taylor series. The precise statement of this result is
due to Steinhaus in 1929. Further work on random series has been taken
in the 1930s, by Paley and Zygmund, Wiener, Lévy and others, see the
excellent historical survey by \cite{Ka1}, and also the book \cite{Ka2}.

In additional to random Taylor series we are also interested in this
work in random Fourier series, as the two are evidently closely related.
There are several natural reasons for studying these two types of
random series. First, one can use random Fourier series to give examples
of functions with desired and undesired properties. In addition, random
Fourier series have strong connections to Brownian motion. Furthermore,
Fourier series are a useful tool for the study of entire functions,
a fact which we use in this work. Second, random polynomials were
studied for a large variety of reasons, including some motivation
from mathematical physics, and random analytic functions are their
natural generalization. A third reason is the relation between lacunary
series and random series. It is well known since the works of Khintchin
and Kolmogorov and of Paley and Zygmund that there are many similarities
between them. Some of these similarities will be explored in this
work. As it turns out, methods that were originally developed for
lacunary series can be applied to random series as well (we will not
present any new results relating to lacunary series).

One of the main interests of this work is in the behavior of the zero
set of random analytic functions. The study of zeros of random polynomials
and random analytic functions started with the pioneering works of
Wiener, Littlewood and Offord, Kac, and Rice, and was later continued
by Erd\H{o}s, Ibragimov, Hammersley, Kahane, Maslova, Offord, and
many others. During the last two decades, the subject was revived
by several groups of researchers who came from very different areas
and established new links to mathematical physics, probability theory,
complex analysis, and complex geometry. Some of the recent works were
surveyed in various places \cite{NS3,Zel}; see also the introductory
article \cite{NS2} and the recent book \cite{BKPV}. The field of
random entire functions is still evolving, and contains many different
and interesting open problems. In addition, the problems usually involve
a unique combination of analytic and probabilistic techniques.

With only few exceptions, almost all of the interesting results proven
in the last two decades deal with Gaussian analytic functions. The
special properties of the Gaussian distribution usually play an absolutely
indispensable role in these works. At the same time, there are several
basic long-standing questions pertaining to zeros of random Taylor
series and random polynomials with independent coefficients having
non-Gaussian distribution, which remain open for about half a century,
sometimes even more. For some of these questions `the Rademacher case'
(or Bernoulli case), when the random coefficients take the values
$\pm1$ with equal probability already contains all principal difficulties%
\footnote{To quote M. Kac: ``Upon closer examination it turns out that the
proof I had in mind {[}in a previous paper{]} . . . is inapplicable
to the {[}case when the coefficients have a discrete distribution{]}
. . . This situation tends to emphasize the particular interest of
the discrete case, which surprisingly enough turns out to be the most
difficult.''%
}. Moreover, in many instances, the Rademacher case yields similar
results for more or less arbitrary random symmetric coefficients.

\section{Contents of this work}

In the first part of this work we discuss Rademacher Fourier series
and prove that the logarithm of such a series is integrable (to any
positive power). This will be in turn the main technical tool for
the second part.

In the second part of the work we study the value distribution of
random analytic functions. The first result is an answer to an old
question from J.-P. Kahane book, about the range of Rademacher Taylor
series in the unit disk. We then use the results of the first part
to prove the almost sure convergence of the number of zeros for Rademacher
entire functions. We also prove the equidistribution of the zeros
of Rademacher entire functions, under a certain growth condition on
the function. We then show that this condition is close to being necessary,
by giving an example of Rademacher entire function with `too many'
real zeros. If one takes for granted the main results of the first
part, then this part can be read indepedently from it.

The third part of this work is concerned with the exceptional behavior
of the zero set of Gaussian entire functions. We find the logarithmic
asymptotics for the probability of the event when such a function
has no zeros in a large disk around the origin. We also show that
this result depends in a crucial way on the nature of the random variables,
by giving an example of a random entire function that always has zeros
outside of some fixed disk. This part is independent from the first
two.

In each part we give some historical background as well as a list
of some open problems. We end this thesis with a list of references
(which is joint for all the parts). We now give a formal description
of the theorems which we are going to prove.

\subsection{Random Taylor series}

As was already mentioned, the main questions of interest in this work
involve the statistical properties of the zero set of random analytic
functions, given by a Taylor series of the form\index{random Taylor series}
\[
\tag{{\ensuremath{\star}}}f\left(z\right)=\sum_{n\ge0}\xi_{n}a_{n}z^{n}.
\]
Here $\left\{ \xi_{n}\right\} $ is a sequence of independent and
identically distributed random variables, and the coefficients $a_{n}$
are deterministic (i.e. non-random). The classes of functions that
are studied will be either entire functions or analytic functions
in the unit disk. It is known that the domain of convergence for functions
of the form $\left(\star\right)$ is the disk where the (non-random)
series $\sum a_{n}z^{n}$ converges%
\footnote{Some minimal assumptions on the distribution of the random variables
$\xi_{n}$ are required. For example, $\Ex\left\{ \log^{+}\left|\xi_{n}\right|\right\} <\infty$
is sufficient. %
} (\cite[Lemma 2.2.3]{BKPV}). Note that when dealing with random Taylor
series of the form $\left(\star\right)$, one usually arrives at the
following `trade off': either to impose regularity assumptions on
$a_{k}$'s trying to achieve as much generality as possible with the
$\xi_{k}$'s, or to consider general non-random coefficients $a_{k}$'s,
imposing certain assumptions on the i.i.d. factors $\xi_{k}$'s. In
general, in this work we chose the latter. For an example of papers
that deal with the first alternative, see \cite{IZ,KZ}.

In order to count the number of zeros of a random analytic function
$f\left(z\right)$, we define the zero counting measure\index{zero counting measure}
$\nu_{f}$, which has an atom for each zero of $f$ (weighted according
to the multiplicity of the zero). By Green's formula, we have the
following expression for this measure in terms of $f$:
\[
\nu_{f}\left(z\right)=\frac{1}{2\pi}\Delta\log\left|f\left(z\right)\right|,
\]
where the Laplacian on the RHS is taken in the sense of distributions.
In principle, it is difficult to analyze this measure directly, unless
one makes some further assumptions (see \cite{IZ,KZ}). We will now
explain briefly the most well-studied case, where the random variables
$\xi_{n}$ are standard complex Gaussians. Such functions are called
Gaussian analytic functions (GAFs) and a lot is known about their
zero counting measure $\nu_{f}$. Part of the motivation behind this
work is to extend this knowledge to other types of random variables.

\subsection{Gaussian analytic functions (GAFs)\index{Gaussian analytic functions}\index{GAF}}

A random variable $\xi$ is a standard complex Gaussian\index{standard complex Gaussian distribution},
if $\xi$ has density $\frac{1}{\pi}e^{-\left|z\right|^{2}}$ with
respect to Lebesgue measure $m$ in the complex plane. Therefore $\Ex\left\{ \xi\right\} =0$
and $\Ex\left\{ \left|\xi\right|^{2}\right\} =1$, where here and
in the rest of this work $\Ex$ stands for the expected value.

As a Gaussian random process, the behavior of a GAF, and thus its
zero set, is determined by its covariance kernel\index{covariance kernel}
\[
K_{f}\left(z,w\right)=\Ex\left\{ f\left(z\right)\overline{f\left(w\right)}\right\} =\sum_{n\ge0}\left|a_{n}\right|^{2}\left(z\overline{w}\right)^{n};
\]
for details, see for example \cite[p. 16]{BKPV}. A particularly nice
example is the classical Edelman-Kostlan formula (\cite[p. 25]{BKPV}),
which tells us that\index{Edelman-Kostlan formula}
\begin{equation}
\Ex\left\{ \nu_{f}\left(z\right)\right\} =\frac{1}{2\pi}\Delta\log\sqrt{K_{f}\left(z,\overline{z}\right)}\dd m\left(z\right).\label{eq:Edel-Kos_for}
\end{equation}
Now, if we write $K_{f}\left(r\right)=K_{f}\left(r,r\right)$ and
denote by $n_{f}\left(r\right)$ the number of zeros of $f$ inside
the disk $\disk r=\left\{ \left|z\right|\le r\right\} $, we conclude
that
\begin{equation}
\Ex\left\{ n_{f}\left(r\right)\right\} =\intsy_{\disk r}\Ex\left\{ \nu_{f}\left(z\right)\right\} \,\mathrm{d}m\left(z\right)=\frac{\dd\log\sqrt{K_{f}\left(r\right)}}{\dd\log r}=\frac{r}{2}\cdot\frac{K^{\prime}\left(r\right)}{K\left(r\right)}.\label{eq:Ex_n_f(r)}
\end{equation}
One of the goals of this work is to prove a similar result for other
types of random variables.

\subsection{Rademacher Fourier series - Logarithmic integrability}

The extension of (\ref{eq:Ex_n_f(r)}) for Bernoulli random variables
is reduced, using Jensen's formula, to the study of the behavior of
Rademacher Fourier series. These are series of the form\index{Rademacher Fourier series}
\[
g\left(\theta\right)=\sum_{n\in\bbz}b_{n}\xi_{n}e^{2\pi in\theta},\quad\theta\in\left[0,1\right],
\]
where the coefficients $b_{n}$ satisfy $\sum\left|b_{n}\right|^{2}<\infty$,
and the $\xi_{n}$s are i.i.d. random variables taking the values
$\pm1$ with probability $\tfrac{1}{2}$ each. We denote by $\shiftspace$
the circle (or $\bbr\backslash\bbz$) with normalized Lebesgue measure
$m$. It is not difficult to check that
\[
\normp[g][2]2\eqdef\Ex\left\{ \intsy_{\bbt}\left|g\left(\theta\right)\right|^{2}\d{\theta}\right\} =\sum\limits _{n\in\bbz}\left|b_{n}\right|^{2}=\normp[\left\{ b_{n}\right\} ][\ell^{2}\left(\bbz\right)]2.
\]
Intuitively one should expect that a random Fourier series will not
have any `local' structure, and, in particular, it should not be too
small on relatively large sets. The main theorem of the first part
of this work is the following quantitative version of this idea.\index{Rademacher Fourier series!logarithmic integrability}
\begin{thm}
\textup{\cite{NNS1}}\label{thm:log_int_for_Rade_Four_series} Let
$g$ be a Rademacher Fourier series with $\norm[g]2=1$. For every
$q\ge1$ we have
\[
\Ex\left\{ \intsy_{\bbt}\left|\log\left|g\left(\theta\right)\right|\right|^{q}\d{\theta}\right\} \le C^{q}q^{6q},
\]
where $C>0$ is some numerical constant.
\end{thm}
Notice that this theorem is general, in the sense that the constant
$C$ does not depend on the specific values of the coefficients $b_{n}$,
a fact that will be important for our applications.

\subsection{Solution to a question of Kahane}

One of the consequences of Theorem \ref{thm:log_int_for_Rade_Four_series},
is the answer to a well-known old question from Kahane's book \cite[p. xii]{Ka2}:
Suppose that $F\left(z\right)=\sum_{n\ge0}a_{n}\xi_{n}z^{n}$ is a
Rademacher Taylor series with radius of convergence one and such that
$\sum_{n\ge0}\left|a_{n}\right|^{2}=+\infty$. Is it true that, almost
surely (a.s.), the range $F\left(\left\{ \left|z\right|<1\right\} \right)$
is the whole complex plane?\index{Kahane's question}

In the second part of the work we will answer this question, and in
fact prove even more. Given $b\in\bbc$, denote by $\Lambda_{F}\left(b\right)$
the collection of all solutions to the equation $F\left(z\right)=b$,
repeated according to their multiplicity. Let $\left\{ \zeta_{n}\right\} _{n\ge0}$
be a sequence of independent complex-valued symmetric random variables,
and write $F\left(z\right)=\sum_{n\ge0}\zeta_{n}z^{n}$. We prove
\begin{thm}
\textup{\cite{NNS1}}\label{thm:Kahane_prob} If the sequence $\left\{ \zeta_{n}\right\} $
satisfies
\[
\limsup_{n\to\infty}\left|\zeta_{n}\right|^{1/n}=1\quad\mbox{and}\quad\sum_{n\ge0}\left|\zeta_{n}\right|^{2}=+\infty,\quad\mbox{a.s.},
\]
then, almost surely,
\[
\forall b\in\bbc,\,\,\sum_{w\in\Lambda_{F}\left(b\right)}\left(1-\left|w\right|\right)=+\infty.
\]

\end{thm}
The special case of Rademacher analytic functions is clearly implied
by setting $\zeta_{n}=a_{n}\xi_{n}$. The case of Gaussian coefficients
also follows (for another proof of this case see \cite[chap. 13]{Ka2}).
\begin{rem*}
It should be mentioned that the proof of Theorem \ref{thm:Kahane_prob}
requires a stronger version of Theorem \ref{thm:log_int_for_Rade_Four_series}.
We will give the statement and the proof of this version in Part \ref{part:log_int_of_Rad_Four}
Section \ref{subsect:result}.
\end{rem*}

\subsection{Equidistribution of the zeros of random entire functions}

Next, we apply the tools developed in the first part to show that
the zeros of Rademacher entire functions are equidistributed in sectors
(for GAFs this is suggested by the Edelman-Kostlan formula (\ref{eq:Edel-Kos_for})).

Let\index{Rademacher Taylor series} 
\[
F(z)=\sum_{n\ge0}\xi_{n}a_{n}z^{n}
\]
be a Rademacher Taylor series with an infinite radius of convergence
(i.e., $\sum a_{n}z^{n}$ is an entire function). We will assume that
$|F(0)|=1;$ the general case is reduced to this one by letting $F(z)=a_{\nu}z^{\nu}F_{1}(z)$,
where $\nu\in\bbn$ is the multiplicity of the zero of $F$ at the
origin. Similarly to the case of GAFs, we use the notation\index{1sigmaF@$\sigma_{F}$}
\[
\sigma_{F}^{2}(r)=\Ex\bigl\{|F(re^{{\rm i}\theta}|^{2})\bigr\}=\sum_{n\ge0}|a_{n}|^{2}r^{2n},
\]
and also put\index{1sF@$s_{F}$} 
\[
s_{F}\left(r\right)=\frac{\mathrm{d}\log\sigma_{F}(r)}{\mathrm{d}\log r}=r\cdot\frac{\dd\log\sigma_{F}\left(r\right)}{\dd r}.
\]
By $n_{F}(r,\alpha,\beta)$ \index{1nF@$n_{F}$!n_{F}(r,alpha,beta)
@$n_{F}(r,\alpha,\beta)$}\index{zero counting function!angular}we denote the number of zeros
(counted with multiplicities) of $F$ in the sector $\left\{ z\colon\alpha\le\arg z<\beta,\,\left|z\right|\le r\right\} $.
The \textit{integrated zero counting function\index{zero counting function!angular, integrated}}
is given by\index{1NF@$N_{F}$!N_{F}(r,alpha,beta)
@$N_{F}(r,\alpha,\beta)$} 
\[
N_{F}(r,\alpha,\beta)=\intsy_{0}^{r}\frac{n_{F}(r,\alpha,\beta)}{t}\d t\,.
\]
A set $E\subset\left[1,\infty\right)$ is a set of \textit{finite
logarithmic measure\index{finite logarithmic measure}} if
\[
\intsy_{E}\,\frac{\mathrm{d}t}{t}<\infty.
\]
The main result of this section shows that under certain growth condtions
on the function $\sigma_{F}\left(r\right)$ the zeros of $F$ are
equidistributed in sectors.\index{Rademacher Taylor series!equidistribution of the zeros}
\begin{thm}
\label{thm:Rade_ent_funcs_ang_dist_of_zero}Let $F$ be a Rademacher
Taylor series with an infinite radius of convergence.

(i) Suppose that $a>2$ and $\gamma\in\bigl(\tfrac{1}{2}+\tfrac{1}{a},1\bigr)$
. Then, a.s. and in mean,
\[
\sup_{0\le\alpha<\beta\le2\pi}\,\Bigl|N_{F}(r,\alpha,\beta)-\frac{\beta-\alpha}{2\pi}\,\log\sigma_{F}(r)\Bigr|=\O\left(\log^{\gamma}\left(\sigma_{F}\left(r\right)\right)\right)
\]
 when $r\to\infty$ and $\log\sigma_{F}(r)>\log^{a}r$ .

(ii) Suppose that $a>4$ and $\gamma\in\bigl(\tfrac{3}{4}+\tfrac{1}{a},1\bigr)$
. Then there exists a set $E\subset[1,\infty)$ of a finite logarithmic
measure such that, a.s. and in mean, 
\[
\sup_{0\le\alpha<\beta\le2\pi}\,\Bigl|n_{F}(r,\alpha,\beta)-\frac{\beta-\alpha}{2\pi}\, s_{F}\left(r\right)\Bigr|=\O\left(\left(s_{F}\left(r\right)\right)^{\gamma}\right)
\]
when $r\to\infty$ , $r\notin E$, and $\log\sigma_{F}(r)>\log^{a}r$.
\end{thm}
In this theorem, some lower bound on the growth of $\sigma_{F}$ is
necessary. We will show in Section \ref{sec:Rade_Func_with_many_real_zeros}
that if $\beta>\log3$, then all the zeros of the Rademacher Taylor
series $F(z)=\sum_{k\ge0}\xi_{k}e^{-\beta k^{2}}z^{k}$ are real (we
have that $\log\sigma_{F}(r)\sim C_{\beta}\cdot\log^{2}r$ as $r\to\infty$,
for this function).
\begin{rem*}
It should be mentioned that the exceptional set\index{exceptional set}
$E$ in the theorem is unavoidable, since we impose no regularity
assumptions on the sequence $\left\{ a_{n}\right\} $. On the other
hand, if the sequence $\left\{ a_{n}\right\} $ satisfies some very
mild regularity condition, the exceptional set $E$ is not needed.
This exceptional set will reappear, for similar reasons, in the third
part of this work.
\end{rem*}

\subsection{The hole probability for Gaussian entire functions}

The third part of this work is devoted to the study of a special kind
of `exceptional' behavior of the zero set. The `hole' event\index{hole event}
that we are interested in is when the function $f\left(z\right)$
has no zeros inside some domain $D$. We consider this problem for
Gaussian entire functions with general coefficients $a_{n}$. We are
interested in the logarithmic asymptotics of the probability of the
hole event, in the case where the domain $D$ is the disk around zero
with radius $r$, as $r$ tends to infinity. The following theorem
gives a rather precise answer.
\begin{thm}
\textup{\cite{Ni3}}\label{thm:hole_prob_GAFs} Let $\epsilon\in\left(0,\frac{1}{2}\right)$
and let
\[
f\left(z\right)=\sum_{n\ge0}a_{n}\xi_{n}z^{n},
\]
where $\xi_{n}$ are i.i.d. standard complex Gaussians and $a_{0}\ne0$.
Assume that $\sum_{n\ge0}a_{n}z^{n}$ is a non-constant entire function.
We have that
\[
\log^{-}\pr{f\left(z\right)\ne0\mbox{ in }\left|z\right|\le r}=S\left(r\right)+\O\left(S\left(r\right)^{\tfrac{1}{2}+\epsym}\right),\quad r\notin E,
\]
as $r\to\infty.$ Here 
\[
S\left(r\right)=2\cdot\sum_{n\ge0}\log^{+}\left(a_{n}r^{n}\right)
\]
and $E\subset\left[1,\infty\right)$ is a non-random exceptional set
of finite logarithmic measure.\end{thm}
\begin{rem*}
The set $E$ depends only on the sequence $\left\{ a_{n}\right\} $
and on $\epsym$.\index{exceptional set}
\end{rem*}

\subsection{Notation}

Throughout this work the letters $C$ and $c$ denote various positive
numerical constants; usually the first is reserved for `large' ($\ge1$)
numbers, while the second is used for `small' numbers. We use $C_{\alpha}$
for constants that depend on another quantity $\alpha$.

Several types of symbols are used to indicate the asymptotic relations
between different quantities. For positive quantities $\Phi,\Psi$,
the notation $\Phi\lesssim\Psi$ means that there exists a positive
constant $C$ so that $\Phi\le C\cdot\Psi$. Another notation for
the same relation is $\Phi=\O\left(\Psi\right)$ (the usage depends
on the convenience at each point). The notation $\Phi=\O_{\alpha}\left(\Psi\right)$
means that the implied constant depends only on $\alpha$. For positive
sequences or functions the notation $\Phi=o\left(\Psi\right)$ indicates
that the (appropriate) limit $\Phi/\Psi$ is zero.

We use $\Pr\left\{ E\right\} $or $\Pr\left(E\right)$ to denote the
probability of an event $E$. The notation $\Ex\left\{ X\right\} $
or $\Ex\left(X\right)$ is used for the expected value of the random
variable $X$.

We use $\disk r$ for the closed disk $\left\{ \left|z\right|\le r\right\} $,
while $\cir r$ is used for the circle $\left\{ \left|z\right|=r\right\} $.
Finally, the log-derivative $\frac{\dd f\left(r\right)}{\dd\log r}$\index{1ddlogr@$\frac{\dd}{\dd\log r}$}
is the same as $r\cdot\frac{\dd f\left(r\right)}{\dd r}$.

\newpage{}

\part{Logarithmic Integrability of Rademacher Fourier series\label{part:log_int_of_Rad_Four}}

We start by explaining the central role played by the logarithmic
integrability of Rademacher Fourier series. Consider a Rademacher
Taylor series\index{Rademacher Taylor series}
\[
f\left(z\right)=\sum_{n\ge0}a_{n}\xi_{n}z^{n},
\]
with radius of convergence $R$, $0<R\le\infty$, and for $r=\left|z\right|$,
put
\[
\sigma_{f}^{2}\left(r\right)=\Ex\left\{ \left|f\left(z\right)\right|^{2}\right\} =\sum_{n\ge0}\left|a_{n}\right|^{2}r^{2n}.
\]
We will assume that $\sigma_{f}\left(r\right)\to\infty$ as $r\to R$.
We are interested in the asymptotics of the random counting function
$n_{f}\left(r\right)$,\index{zero counting function} which counts
the number of zeros of $f$ in the disk $\disk r=\left\{ \left|z\right|\le r\right\} $,
as $r\to R$. To simplify the notation, assume that $a_{0}=1$, and
thus $\left|f\left(0\right)\right|=1$. Denote by
\[
N_{f}\left(r\right)\eqdef\intsy_{0}^{r}\,\frac{n_{f}\left(t\right)}{t}\, dt
\]
the integrated counting function\index{zero counting function!integrated}.
Then, by Jensen's formula\index{Jensen's formula},
\[
N_{f}\left(r\right)=\intsy_{\bbt}\log\left|f\left(re^{2\pi i\theta}\right)\right|\d{\theta}-\log\left|f\left(0\right)\right|=\log\sigma_{f}\left(r\right)+\intsy_{\bbt}\log\left|\widehat{f}\left(re^{2\pi i\theta}\right)\right|\d{\theta},
\]
where
\[
\widehat{f}\left(re^{2\pi i\theta}\right)\eqdef\frac{f\left(re^{2\pi i\theta}\right)}{\sigma_{f}\left(r\right)}.
\]
Note that for a fixed $r$ the function $\widehat{f}\left(re^{2\pi i\theta}\right)=\sum_{n\ge0}\widehat{a}_{n}\left(r\right)\xi_{n}e^{2\pi i\cdot n\theta}$
is a random Fourier series normalized by the condition $\sum_{n\ge0}\left|\widehat{a}_{n}\left(r\right)\right|^{2}=1$.

In order to explain the reduction to logarithmic integrability we
start by considering a simpler case, where $\xi_{n}$ are standard
complex Gaussian random variables. Then, for every $\theta\in\bbt$,
the random variable $\widehat{f}\left(re^{2\pi i\theta}\right)$ is
again a standard complex Gaussian. Thus $\Ex\left\{ \left|\log\left|\widehat{f}\right|\right|\right\} $
is a positive numerical constant and therefore, $\sup_{r<R}\Ex\left\{ \left|N_{f}\left(r\right)-\log\sigma_{f}\left(r\right)\right|\right\} \le C$
. We can then `differentiate' this asymptotic relation, and get
\[
n_{f}\left(r\right)=\frac{\mathrm{d}\log\sigma_{f}\left(r\right)}{\mathrm{d}\log r}+e\left(r\right),\,\mbox{with\,\,}\Ex\left|e\left(r\right)\right|=o\left(1\right)\cdot\frac{\mathrm{d}\log\sigma_{f}\left(r\right)}{\mathrm{d}\log r},\, r\to R,\, r\notin E,
\]
where $E$ is an exceptional set\index{exceptional set} of $r$'s
of finite logarithmic measure, which is unavoidable when the non-random
coefficients $a_{k}$ have an irregular behaviour. Similarly, one
gets $\sup_{r<R}\Ex\left\{ \left|N_{f}\left(r\right)-\log\sigma_{f}\left(r\right)\right|^{p}\right\} \le\left(Cp\right)^{cp}$,
for $p\ge1$. Combined with the classical Borel-Cantelli lemma, this
leads to a good almost sure estimate for the error term $e\left(r\right)$.

The same approach works, albeit with more technical difficulties,
for the Steinhaus coefficients $\xi_{k}=e^{2\pi i\gamma_{k}}$, where
$\gamma_{k}$ are independent and uniformly distributed on $\left[0,1\right]$.
In this case, one needs to estimate the $L^{p}$ norms of the logarithm
of the absolute value of a normalized linear combination of independent
Steinhaus variables. This was done by Offord in \cite{Of2}; twenty
years later, Favorov and Ullrich independently rediscovered this idea
and gave new applications (see \cite{Fa1,Fa2,Ul1,Ul2}, and also the
recent paper of Mahola and Filevich \cite{MF}). However, these techniques
fail for Rademacher random variables (in principle because linear
combinations of Rademacher random variables can vanish with positive
probability). Still, not everything is lost: note that in order to
estimate the error term in Jensen's formula we do not need a pointwise
estimate for $\Ex\left\{ \left|\log\left|\widehat{f}\left(re^{2\pi i\theta}\right)\right|\right|\right\} $
that is uniform in $\theta\in\bbt$. For our purposes, the integral
estimate $\Ex\left\{ \intsy_{\bbt}\left|\log\left|\widehat{f}\left(re^{2\pi i\theta}\right)\right|\right|^{p}\d{\theta}\right\} $
is not worse than the uniform one.

In the first section we give some background and an outline of the
proof. We also state the main tools used in the proof and some important
notations. In the second and third sections we give the proof of the
logarithmic integrability. In the last two section we give an example
that illustrates the sharpness of the result and discuss some further
related problems. In the next part we will apply this theorem to the
study of the value distribution of random analytic functions.

\section{Logarithmic Integrability - Background and the Main Result}

In this section we describe the main result of this part of the work.
It shows that an arbitrary Rademacher Fourier series cannot be too
close to a random constant. The version we gave in the introduction
corresponds to the case when the random constant is the zero function.
The extension that we prove is needed for the proof of Theorem \ref{thm:Kahane_prob}
on the range of random Taylor series in the unit disk.

\subsection{Notation\label{subsect:notation}}

We denote by $\bbt$ the interval $[0,1)\subset\bbr$, which we treat
as $\bbt=\bbr/\bbz$. $m$ will be either the Lebesgue measure on
$\bbt$ normalized by $m(\bbt)=1$, or the Lebesgue measure on $\bbr$,
depending on the context. We also write $e(\theta)=e^{2\pi{\rm i}\theta}$,
$\theta\in\bbt$.

Let $(\Omega,\Pr)$ be a standard probability space, and let $\xi_{k}\colon\Omega\to\{\pm1\}$,
$k\in\bbz$ be an independent Rademacher (or Bernoulli) sequence,
that is, each $\xi_{k}$ takes the values $\pm1$ with probability
$\tfrac{1}{2}$ each. Thus, we have a natural product measure space
$Q=\Omega\times\bbt$, with square integrable functions $L^{2}(Q)=L^{2}(Q,\mu)$.\index{1Ltwo@$L^{2}(Q)$}

Our main interest here is in the (closed) subspace $\RFfuncs\subset L^{2}(Q)$
of\\
Rademacher Fourier series\index{1Ltworf@$\RFfuncs$}\index{1Ltworf@$\RFfuncs$|see{Rademacher Fourier series}}
\[
{\displaystyle f\!=\!\sum_{k\in\bbz}a_{k}\phi_{k}},\qquad{\displaystyle \sum_{k\in\bbz}|a_{k}|^{2}\!<\!\infty},
\]
where $\phi_{k}(\omega,\theta)=\xi_{k}(\omega)e(k\theta)$, $k\in\bbz$,
$(\omega,\theta)\in Q$. We note that the system $\{\phi_{k}\}$ is
an orthonormal basis for the space $\RFfuncs$, and for $f\in\RFfuncs$
one can easily check that 
\begin{eqnarray*}
\|f\|_{2}^{2} & \eqdef & \intsy_{Q}|f|^{2}\,{\rm d}\mu=\intsy_{\Omega}|f(\omega,\cdot)|^{2}\,{\rm d}\Pr(\omega)\\
 & = & \intsy_{\bbt}|f(\cdot,\theta)|^{2}\,{\rm d}m(\theta)=\sum_{k\in\bbz}|a_{k}|^{2}=\|\{a_{k}\}\|_{\ell^{2}(\bbz)}^{2}\,.
\end{eqnarray*}
A measurable function $b$ on $Q$ that does not depend on $\theta$
will be called a \textit{random constant\index{random constant}}
(note that $b$ is not necessarily in $\RFfuncs$).

\subsection{The result\label{subsect:result}}

The goal of this part of the work is to prove that for any $f\in\RFfuncs$
with $\|f\|_{2}=1$, any $b\in L^{\infty}(\Omega)$ with $\|b\|_{\infty}<\tfrac{1}{20}$,
and every $p\ge1$, we have\index{Rademacher Fourier series!logarithmic integrability}

\begin{equation}
\intsy_{Q}\bigl|\log|f-b|\bigr|^{p}\,{\rm d}\mu\le\left(Cp\right){}^{6p}\,,\label{eq:log_int_Rad_Four_gen_case}
\end{equation}
where $C$ is some positive numerical constant. We note that the condition
on the function $b$ is a technical one. Its purpose is to avoid some
degenerate cases, for example, the case when the functions $f$ and
$b$ are both equal to $\xi_{0}$. The logarithmic integrability is
a straightforward corollary%
\footnote{It follows by applying Theorem \ref{thm:Zygmund-type-final} below
for the sub-level sets $E_{\epsym}=\left\{ \left(\omega,\theta\right)\in Q\,:\,\left|f-b\right|<\epsym\right\} $
and then integrating the LHS of (\ref{eq:log_int_Rad_Four_gen_case})
by parts.%
} of the following distribution inequality for $\RFfuncs$ functions.\index{Rademacher Fourier series!distribution inequality}
\begin{thm}
\label{thm:Zygmund-type-final} For any $f\in\RFfuncs$, any random
constant $b\in L^{\infty}(\Omega)$ with $\|b\|_{\infty}<\tfrac{1}{20}\|f\|_{2}$,
and any set $E\subset Q$ of positive measure, 
\[
\normp[f][2]2=\intsy_{Q}|f|^{2}\d{\mu}\le\exp\left(C\log^{6}\left(\frac{2}{\mu(E)}\right)\right)\intsy_{E}|f-b|^{2}\d{\mu}\,,
\]
where $C$ is some positive numerical constant.
\end{thm}
We note that the power $6$ on the RHS is (probably) not the best
possible, but in Section \ref{sect_log_int_examples} we will give
an example which shows that it cannot be replaced by any number less
than~$2$. Since the proof of Theorem \ref{thm:Zygmund-type-final}
is long and consists of several parts, we give a general description
of some of the ideas behind it in this section. The next two sections
contain the proof, and in the last section we discuss some further
problems (as well as give the aforementioned example).

\subsection{Background}

The proof of Theorem \ref{thm:Zygmund-type-final} is based on ideas
from harmonic analysis developed by Nazarov \cite{Na1,Na2} to treat
lacunary Fourier series. It uses a Turán-type lemma from \cite[Chapter~1]{Na1},
and the technique of small shifts introduced in \cite[Chapter~3]{Na1}.
It is worthwhile mentioning that the study of lacunary series has
a long history. Already in 1872, Weierstrass gave a famous example
of a continuous functions which is not differentiable at any point,
in the form of a lacunary trigonometric series. Kolmogorov in the
1920s and Zygmund in the 1930s studied the convergence and integrability
properties of lacunary series before they gave the counterpart (and
more famous) results for random series.

Let $\Lambda=\left\{ m_{k}\right\} _{k\in\bbz}\subset\bbz$ be the
`lacunary' spectrum of some Fourier series, that is, we consider series
of the form\index{lacunary Fourier series}
\begin{equation}
g\left(\theta\right)=\sum_{k\in\bbz}a_{m_{k}}e\left(m_{k}\theta\right),\quad\theta\in\bbt,\quad\sum_{k\in\bbz}\left|a_{m_{k}}\right|^{2}<\infty,\label{eq:lacunary_Four_ser}
\end{equation}
where the set $\Lambda$ is `small' in some sense. In 1948 Zygmund
proved the following uniqueness result
\begin{thm*}
Suppose that the set $\Lambda=\left\{ m_{k}\right\} _{k\in\bbz}\subset\bbz$
satisfies
\[
R\left(\Lambda\right)\eqdef\sup_{r\ne0}\#\left\{ \left(k^{\prime},k^{\dprime}\right)\,:\, m_{k^{\prime}}-m_{k^{\dprime}}=r\right\} <\infty.
\]
For every measurable set $E\subset\bbt$ of positive measure there
exists a constant $C\left(\Lambda,E\right)$ such that, for every
$g\in L^{2}\left(\bbt\right)$ of the form \textup{(\ref{eq:lacunary_Four_ser}),}
\[
\normp[g][L^{2}\left(\bbt\right)]2\le C\left(\Lambda,E\right)\intsy_{E}\left|g\right|^{2}\d m.
\]

\end{thm*}
In the 1990s (see \cite{Na1}), Nazarov significantly improved the
previous theorem by giving an effective bound for the constant $C\left(\Lambda,E\right)$.
More precisely, he proved that for any $\epsym>0$, there exists a
constant $D\left(\epsym,R\left(\Lambda\right)\right)$ such that
\[
C\left(\Lambda,E\right)\le\exp\left(\frac{D\left(\epsym,R\left(\Lambda\right)\right)}{m\left(E\right)^{2+\epsym}}\right);
\]
in particular, this implies the logarithmic integrability of such
lacunary series for every power less than $2$.\index{lacunary Fourier series!logarithmic integrability}
\begin{rem*}
Notice that in Nazarov's result the constant $C\left(\Lambda,E\right)$
depends only on the measure of $E$.

In a later work (\cite{Na2}), Nazarov considered the case of Hadamard
lacunary series\index{lacunary Fourier series!Hadamard lacunary series},
that is, series whose spectrum $\Lambda$ satisfies
\[
\liminf_{\left|k\right|\to\infty}\frac{m_{k+{\rm sign}\left(k\right)}}{m_{k}}>1.
\]
Using an improved version of the results in \cite{Na1}, he proved
that
\[
\normp[g][L^{2}\left(\bbt\right)]2\le\exp\left(C\log^{10}\left(\frac{2}{m\left(E\right)}\right)\right)\intsy_{E}\left|g\right|^{2}\d m,
\]
where $C$ is some positive numerical constant. In particular, this
result implies the logarithmic integrability of such series for every
positive power. Our proof is motivated by the ideas contained in both
of these papers. It exemplifies --- perhaps in a spirit similar to
Kolmogorov and Zygmund's results --- the close connection between
lacunary and random series. Just to give a particular example, the
$\Lambda_{p}$ property of Hadamard lacunary series (which is central
to the results of \cite{Na2}) is replaced by the Khinchin and bilinear
Khinchin inequalities.
\end{rem*}

\subsection{The basic tools and some futher notations\label{subsect:basic-tools}}

Here is the list of the tools we will be using in the proof of Theorem~\ref{thm:Zygmund-type-final}.

\subsubsection{\label{subsubsect:Turan} Turán-type lemma\index{Turán-type lemma|ii}\index{exponential polynomial}}

Let 
\[
p(z)=\sum_{k=0}^{n}a_{k}e^{i\lambda_{k}t},\qquad a_{k}\in\bbc,\quad\lambda_{0}<\,\cdots\,<\lambda_{n}\in\bbr,
\]
be an exponential polynomial. Then for any interval $J\subset\bbr$
and any measurable subset $E\subset J$ of positive measure,
\[
\sup_{J}|p|\le\left(\frac{Cm(J)}{m(E)}\right)^{n}\sup_{E}|p|\,.
\]
For the proof, see \cite[Chapter~I]{Na1}. We will also use the $L^{2}$-bound
that follows from this estimate (see ~\cite[Chapter~III, Lemma~3.3]{Na1}).
It states that, in the same setting,

\begin{equation}
\|p\|_{L^{2}(J)}\le\left(\frac{Cm(J)}{m(E)}\right)^{n+\frac{1}{2}}\|p\|_{L^{2}(E)}\,.\label{eq:turan_lt_est}
\end{equation}

\subsubsection{\label{subsubsect:Khinchin} Khinchin's inequality\index{Khinchin's inequality|ii}}

Let $\left\{ \xi_{k}\right\} $ be independent Rademacher random variables,
and let $\left\{ a_{k}\right\} $ be complex numbers. Then for each
$p\ge2$, we have 
\[
\left(\Ex\left|\sum_{k}a_{k}\xi_{k}\right|^{p}\right)^{1/p}\le C\sqrt{p}\cdot\sqrt{\sum_{k}\left|a_{k}\right|^{2}}.
\]

\subsubsection{\label{subsubsect:bilinear-Khinchin} Bilinear Khinchin's inequality\index{Khinchin's inequality!Bilinear|ii}}

Let $\left\{ \xi_{k}\right\} $ be independent Rademacher random variables,
and let $\left\{ a_{k,\ell}\right\} $ be complex numbers. Then for
each $p\ge2$, we have 
\[
\left(\Ex\left|\sum_{k\ne\ell}a_{k,\ell}\xi_{k}\xi_{\ell}\right|^{p}\right)^{1/p}\le Cp\cdot\sqrt{\sum_{k\ne\ell}\left|a_{k,\ell}\right|^{2}}.
\]
A simple and elegant proof of this inequality can be found in a recent
preprint by L.~Erd\H{o}s, A.~Knowles, H.-T.~Yau, and J.~Yin~\cite[Appendix~B]{EKYY}.

\subsubsection{Notations}

For a set $E\subset Q$, we denote its sections by $E_{\omega}\eqdef\left\{ \theta\in\bbt\,\colon(\omega,\theta)\in E\right\} $,
$\omega\in\Omega$\index{1Eomega@$E_{\omega}$}. The set $E\subset Q$
`shifted' by $t\in\bbt$ is denoted by $E+t\eqdef\left\{ (\omega,\theta)\,\colon(\omega,\theta-t)\in E\right\} $\index{1Eplust@$E+t$}.
Then 
\[
E_{\omega}+t=\left\{ \theta\,\colon\theta-t\in E_{\omega}\right\} =\left(E+t\right)_{\omega}.
\]
The function $g\in L^{2}(Q)$ shifted by $t$ is denoted by $g_{t}$:
$g_{t}(\omega,\theta)=g(\omega,\theta+t)$\index{1gt@$g_{t}$}. Note
that for the indicator function of $E$, we have that $\left(\indf_{E}\right)_{t}=\indf_{E-t}$.

We write $[x]$ for the integral part of $x$. Also, $\Delta_{t}(E)\eqdef\mu\left((E+t)\setminus E\right)$\index{1DeltatE@$\Delta_{t}(E)$},
a notation that will appear many times during the proof.

\subsection{An overview of the proof of Theorem~\ref{thm:Zygmund-type-final}}

For the sake of simplicity we will assume here that $b\equiv0$; the
proof of the general case requires some technical modifications.

Let $f\in\RFfuncs$ be a Rademacher Fourier series, and let $E\subset Q$
be a set of positive measure. We begin with the following observation,
due to Zygmund: using the expansion of $\left|f\right|^{2}$ as a
diagonal and an off-diagonal sums, we can write
\begin{equation}
\intsy_{E}\left|f\right|^{2}\d{\mu}=\mu\left(E\right)\normp[f][2]2+\left\langle A_{E}f,f\right\rangle ,\label{eq:Zyg_met_op_A_E}
\end{equation}
where $A_{E}$ is a certain (compact and self-adjoint) operator (on
the subspace $\RFfuncs$). Now if the operator norm of $A_{E}$ is
small, say $\norm[A_{E}]{\mbox{op}}<\frac{\mu\left(E\right)}{2}$,
then this immediately yields the conclusion of Theorem \ref{thm:Zygmund-type-final}
(with a better dependence on $\mu\left(E\right)$). We will find upper
bounds of this form in two special cases: the first and simpler case
is when the measure of the set $E$ is close to $1$, the second case
is when $E$ is a set with many `long' sections. We say that $E$
has many `long' sections, if there are sufficiently many $\omega$'s
for which the $m$-measure of $E_{\omega}$ is close to $1$. It should
be mentioned that these bounds use the random nature of the function
$f,$ as non-random Fourier series can be norm-concentrated on an
arbitrarily small set. Handling other types of sets $E$ is much more
difficult, and constitutes the main part of the proof; we will call
such sets `spreadable'\index{spreadable set}, for a reason that will
soon become clear. The function $\Delta_{t}(E)$ \index{1DeltatE@$\Delta_{t}(E)$}
gives a quantitative measure for the spreadability of a set $E$.

\subsubsection{The iterative procedure}

The proof of Theorem \ref{thm:Zygmund-type-final} for spreadable
sets is iterative in nature. Given such a set $E\subset Q$, we wish
to find a larger set $\widetilde{E}$, which contains $E$, and for
which we still have a good estimate for $\left|f\right|^{2}$. We
call this procedure \textit{spreading} \index{spreading procedure}\index{spreading procedure|see{spreading lemma}}

We continue this procedure until the new set $\widetilde{E}$ is either
of measure close to $1$, or contains many long sections. At this
point, the procedure is stopped, by the methods that we have previously
mentioned. In order to show that this can be done, it is essential
to prove that we never `get stuck' during this procedure. In mathematical
terms, we show that there exist two positive functions $\delta\left(s\right)$
and $D\left(s\right)$, with $\delta\left(s\right)\downarrow0$, $D\left(s\right)\uparrow\infty$
as $s\downarrow0$, such that for every measurable set $E$ of measure
at least $s$, one can find a set $\widetilde{E}\supset E$ of measure
at least $s+\delta\left(s\right)$ for which
\begin{equation}
\intsy_{\widetilde{E}}\left|f\right|^{2}\d{\mu}\le D\left(s\right)\intsy_{E}\left|f\right|^{2}\d{\mu}.\label{eq:spread_lemma_abst}
\end{equation}
If we have good control on the decay of $\delta\left(s\right)$ and
the growth of $D\left(s\right)$ near $0$, then this procedure will
give Theorem \ref{thm:Zygmund-type-final} as the conclusion. As it
turns out, for spreadable sets\index{spreadable set} we indeed have
good control over these functions.
\begin{rem*}
During the proof we will show that if a set is not spreadable, then
it contains many long sections.
\end{rem*}

\subsubsection{The Spreading Lemma\index{spreading lemma}}

The heart of the iterative procedure is the Spreading Lemma, which
guarantees the existence of the set $\widetilde{E}$ above. The basic
`mechanism' is the\textbf{ }Turán-type inequality (\ref{eq:turan_lt_est})\index{Turán-type lemma}
stated above. For exponential polynomials it gives an estimate of
the type (\ref{eq:spread_lemma_abst}). It is reasonable to assume
that if a function $f$ is `close' (say in $L^{2}$ norm) to an exponential
polynomial, then one can use this inequality to get such an estimate
for $f$ as well. We will show that a function $f\in\RFfuncs$ has
a good \textit{local} approximation by (random) exponential polynomials,
which is sufficient for our purpose.\index{Rademacher Fourier series!local approximation}

For an effective use of the Turán-type inequality we want that, first,
the degree of the approximating polynomial (the number of frequencies)
is sufficiently small, and second, that the ratio $m\left(J\right)/m\left(E\right)$
in (\ref{eq:turan_lt_est}) is not too large. The proof shows that
$f$ can be approximated by (random) exponential polynomials of fixed
degree, which depends only on the measure of the set $E$. We then
make use of the fact that spreadable sets\index{spreadable set} contain
many `good' sections. Good sections are those sections $E_{\omega}$
that contain small intervals $J$, where the ratio $m\left(J\right)/m\left(J\cap E_{\omega}\right)$
is not too large, while there is still `room' for spreading (notice
that the local approximation should be sufficiently accurate on the
scale of the interval $J$). Thereby, the set $\widetilde{E}$ is
made up of a union of such small intervals. The existence of good
sections follows from the fact that the quantity $\Delta_{t}(E)$
is not too small for spreadable sets, for some value of $t$.

\subsubsection{Local approximation by exponential polynomials\index{Rademacher Fourier series!local approximation|ii}}

Exponential polynomials\index{exponential polynomial} of the form
$\sum_{j=1}^{n}c_{j}e(\lambda_{j}x)$, $\lambda_{1}<\cdots<\lambda_{n}\in\bbr$,
are the (homogenous) solutions of the linear differential equation
(of order $n$) $Dg=0$, where $D=\prod_{j=1}^{n}e(\lambda_{j}x)\,\frac{{\rm d}}{{\rm d}x}\, e(-\lambda_{j}x)$.%
\footnote{This notation means: $\left(e(\lambda x)\,\frac{{\rm d}}{{\rm d}x}\, e(-\lambda x)\right)\left(g\left(x\right)\right)=e(\lambda x)\left[\frac{{\rm d}}{{\rm d}x}\left(e(-\lambda x)g\left(x\right)\right)\right]$.%
} In order to obtain the approximation of $f$ by exponential polynomials
we proceed as follows: We write $f$ as a sum $f=f_{0}+f_{1}+\cdots+f_{m}$,
where $f_{1},\ldots,f_{m}$ are certain trigonometric polynomials
(with many frequencies), and $f_{0}$ is a (small) error term. Each
of these trigometric polynomials is a non-homogeneous solution to
an equation $D_{k}f_{k}=g_{k}$, where the number of frequencies in
$D_{k}$ is relatively small, and the norm of $g_{k}$ is small. Since
the norm of $g_{k}$ is small, the homogeneous solution $h_{k}$ of
$D_{k}h_{k}=0$ is a good approximation for $f_{k}$, and is of small
degree.

Our method is similar to the description above, but since $f$ is
random we approximate it by random exponential polynomials, with fixed
(non-random) frequencies and random coefficients $c_{j}$. The key
for the success of this method is to show the the spectrum of $f$
is concentrated on some (finite and `small') set of intervals, depending
on the set $E$.%
\footnote{This is somewhat counterintuitive, since the Fourier coefficients
of $f$ are arbitrary. It means that for each set $E$ there are some
`critical' frequencies specific for this set.%
} This set of intervals will determine the partition of $f$ as mentioned
above.

\subsubsection{Spectral description of $\RFfuncs$ functions - The $\Exploc$ property\index{Rademacher Fourier series!spectral description}}

An exponential polynomial has the important property that a finite
linear combination of its translates (which we call `shifts') vanishes.
Actually, it is readily seen that we only need a linear combination
of $m+1$ different shifts, where $m$ is the degree of the polynomial. 

A crucial ingredient of the proof is an analogous property for $L^{2}\left(Q\right)$
functions. To keep the formulation simple, we give here a non-random
version of the\index{Exploc
 property@$\Exploc$ property}
\begin{defn*}
{[}$\Exploc$ property{]} Let $g\in L^{2}\left(\bbt\right)$ and fix
$n\in\bbn$, $\tau>0$ and $\varkappa>0$. We say that $g$ has the
$\Exploc\left(n,\tau,\varkappa\right)$ property (or $g\in\Exploc\left(n,\tau,\varkappa\right)$)
if for every $t\in\left(0,\tau\right)$ there exists complex numbers
$a_{k}=a_{k}\left(t\right)$, $k\in\{0,\dots,n\}$, with $\sum_{k=0}^{n}\left|a_{k}\right|^{2}=1$,
such that $\norm[\sum_{k=0}^{n}a_{k}g_{kt}]{L^{2}\left(\bbt\right)}<\varkappa$.
\end{defn*}
Functions $f$ with the $\Exploc\left(n,\tau,\varkappa\right)$ property
are in a sense a generalization of exponential polynomials. The above
definition measures this in a quantitative way: $n$ is the degree
of exponential polynomials that are `close' to $f$, $\tau$ is the
`resolution' of the local approximation, and $\varkappa$ is the error
in the approximation.

In order to explain why this definition implies some special spectral
properties, we first note that if $h=\sum a_{k}g_{kt}\in L^{2}\left(\bbt\right)$,
then by Parseval's theorem we have
\[
\intsy_{\bbt}\left|h\right|^{2}\d x=\intsy_{\bbt}\left|\sum a_{k}g_{kt}\right|^{2}\d x=\sum_{m\in\bbz}\left|\widehat{g}\left(m\right)\right|^{2}\bigl|q_{t}\left(e(tm)\right)\bigr|^{2},
\]
where $\widehat{g}$ are the Fourier coefficients of $g$, and $q_{t}$
is a certain polynomial that depends on the numbers $a_{k}\left(t\right)$.
By the definition of property $\Exploc\left(n,\tau,\varkappa\right)$,
the sum on the LHS is smaller than $\varkappa^{2}$ for $t\in\left(0,\tau\right)$.
The proof uses averaging over $t$, to show that $\left|\widehat{g}\right|$
is small outside a certain collection of intervals. The two main observations
are that the polynomial $q_{t}$ is large outside some small exceptional
set, and that there is some point where there is a good lower bound
estimate for this polynomial inside this exceptional set.

The proofs of the spectral description, the local approximation, and
the Spreading Lemma are stated in terms of functions with the $\Exploc$
property. Thus, the main goal of the rest of the proof is to show
that $\RFfuncs$ functions have this property (the parameters will
depend on the function and the set $E$).

\subsubsection{Almost linear dependence for $\RFfuncs$ functions}

Let $\widehat{E}\subset Q$ be some set of positive measure, and recall
the operator $A_{\widehat{E}}$ defined by the equation (\ref{eq:Zyg_met_op_A_E}).
Using the billinear Khinchine inequality we can get an upper bound
for the Hilbert-Schmidt norm of $A_{\widehat{E}}$ of the form
\[
\norm[A_{\widehat{E}}]{\mathrm{HS}}\le B\left(\mu(\widehat{E})\right),
\]
where $B\left(t\right)\downarrow0$ as $t\downarrow0$ (actually we
use an additional parameter $p$, which gives us some flexibililty
later in the proof). Now, let $V_{\widehat{E}}\subset\RFfuncs$ be
a subspace of finite dimension, which is the span of the eigenspaces
corresponding to the largest eigenvalues of $A_{\widehat{E}}$. If
the dimension of $V_{\widehat{E}}$ is large enough, so that 
\[
\norm[A_{\widehat{E}}\mid_{V_{\widehat{E}}^{\perp}}]{\mathrm{HS}}\le\frac{\mu(\widehat{E})}{2},
\]
then for any $g\in\RFfuncs\ominus V_{\widehat{E}}$ 
\[
\intsy_{\widehat{E}}\left|g\right|^{2}\d{\mu}\ge\frac{\mu(\widehat{E})}{2}\normp[g][2]2.
\]
Now if $f\in\RFfuncs$, then we can find numbers $a_{k}=a_{k}\left(t\right)$
such that $g=\sum_{k=0}^{n}a_{k}f_{kt}\in\RFfuncs\ominus V_{E}$.
The problem in this case is how to relate the integrals $\intsy_{E}\left|g\right|^{2}\d{\mu}$
and $\intsy_{E}\left|f\right|^{2}\d{\mu}$. In order to circumvent
this problem we introduce the set
\[
E^{\prime}=E^{\prime}\left(t\right)=\bigcap_{k=0}^{n}\left(E-kt\right)\subset E.
\]
The fact that $E$ is a spreadable set guarantees that there is some
$\tau>0$ such that $\mu\left(E^{\prime}\right)\ge\tfrac{1}{2}\mu\left(E\right)$
for $t\in\left(0,\tau\right)$. Here we use that $\Delta_{t}(E)$
is not too large for these sets, for $t<n\tau$. Now, using the Cauchy-Schwarz
inequality and the invariance of the integral under translations,
we get that for $g\in\RFfuncs\ominus V_{E^{\prime}}$,
\[
\normp[g][2]2\le\frac{2}{\mu(E^{\prime})}\intsy_{E'}\sum_{k=0}^{n}\left|f_{kt}\right|^{2}\d{\mu}\le D\left(\mu\left(E\right)\right)\intsy_{E}\left|f\right|^{2}\d{\mu},
\]
where $D\left(t\right)\uparrow\infty$ as $t\downarrow0$. In particular,
this implies that $f\in\Exploc\left(n,\tau,\varkappa\right)$ with
a certain $\varkappa$ depending on $\mu\left(E\right)$ and $\intsy_{E}\left|f\right|^{2}\d{\mu}$.
For the method to succeed, it is important that we achieve good control
over the size of the parameters $n$, $\varkappa$, and $\Delta_{n\tau}(E)$.

\subsection{The proof of Theorem~\ref{thm:Zygmund-type-final}}

We start with a formal definition of the $\Exploc$ property, in a
general setting. In Section \ref{sec:log_int_spread_lemma} we prove
the spectral description, the local approximation and the Spreading
Lemma for functions with the $\Exploc$ property. In Section \ref{sec:Rad_Four_in_Exp_loc}
we prove that Rademacher Fourier functions have this property (with
effective control of the parameters), and finish the proof by analyzing
the iterative procedure. In this section we also analyze the cases
of set of large measure as well as sets with `long' sections.

\subsubsection{The $\Exploc$ property - functions with almost linearly dependent
small shifts \label{subsect:Exp_loc}}

Let $\calH$ be a Hilbert space. By $L^{2}(\bbt,\calH)$ we denote
the Hilbert space of square integrable $\calH$-valued functions on
$\bbt$ (in the sense of Bochner). Note that the space $L^{2}(\bbt,L^{2}(\Omega))$
can be identified with $L^{2}(Q)$. To define the property of functions
in $L^{2}(\bbt,\calH)$ having almost linearly dependent small shifts,
we introduce the following set of parameters:\index{Exploc
 property@$\Exploc$ property}
\begin{itemize}
\item the order $n\in\bbn$ (a large parameter);
\item the localization parameter $\tau>0$ (a small parameter);
\item the error $\varkappa>0$ (a small parameter).\end{itemize}
\begin{defn}
\noindent {[}$\Exploc${]}\label{def:Exp} We say that $g\in L^{2}(\bbt,\calH)$
has the ${\rm Exp}_{{\tt loc}}(n,\tau,\varkappa,\calH)$ property
if for each $t\in(0,\tau)$ there exist complex numbers $a_{k}=a_{k}(t),\, k\in\{0,\dots,n\}$,
with $\sum_{k=0}^{n}|a_{k}|^{2}=1$, such that 
\[
\left\Vert \sum_{k=0}^{n}a_{k}g_{kt}\right\Vert {}_{L^{2}(\bbt,\calH)}<\varkappa\,.
\]

\end{defn}
In the case $\calH=\bbc$, this property was introduced in~\cite[Chapter~III]{Na1}.
If $g$ has the ${\rm Exp}_{{\tt loc}}(n,\tau,\varkappa,\calH)$ property
we will write $g\in{\rm Exp}_{{\tt loc}}(n,\tau,\varkappa,\calH)$.
`In small' (i.e., on intervals of length comparable with $\tau$),
the functions with this property behave similarly to exponential sums
with $n$ frequencies and with coefficients in $\calH$. On the other
hand, since the translations act continuously in $L^{2}(\bbt,\calH)$,
for any given $g\in L^{2}(\bbt,\calH)$, $n\in\bbn$, $\varkappa>0$,
one can choose the parameter $\tau>0$ so small that $g\in{\rm Exp}_{{\tt loc}}(n,\tau,\varkappa,\calH)$.

\section{Logarithmic Integrability - The Spreading Lemma\label{sec:log_int_spread_lemma}}

In this section, we extend the main results about functions with the
$\Exploc$ property (the spectral description, the local approximability
by exponential sums, and the spreading lemma) from the scalar case
(as proved in \cite[Chapter~III]{Na1}) to the case considered here.

%% proofs %%

\subsection{Spectral description - The Approximate Spectrum Lemma \label{subsect:approx-spectrum}}

The first lemma shows that each function $g\in{\rm Exp}_{{\tt loc}}(n,\tau,\varkappa,\calH)$
has an `approximate spectrum' $\Lambda_{g}$, which consists of $n$
frequencies, so that the Fourier transform of $g$ is small in the
$\ell^{2}$-norm away from these frequencies.

For $m\in\bbz$, $\Lambda\subset\bbr$, let 
\[
\theta_{\tau}(m)=\min(1,\tau|m|),\quad\Theta_{\tau,\Lambda}(m)=\prod_{\lambda\in\Lambda}\theta_{\tau}(m-\lambda)\,.
\]

\begin{lem}
\label{lemma:approx-spectrum} Given $g\in{\rm Exp}_{{\tt loc}}(n,\tau,\varkappa,\calH)$,
there exists a set $\Lambda=\Lambda_{g}\subset\bbr$ of $n$ distinct
frequencies such that 
\[
\sum_{m\in\bbz}\|\,\widehat{g}(m)\,\|_{\calH}^{2}\,\Theta_{\tau,\Lambda}^{2}(m)\le\left(Cn\right)^{4n}\varkappa^{2}\,.
\]

\end{lem}
The proof of Lemma~\ref{lemma:approx-spectrum}, with small modifications,
follows \cite[Section~3.1]{Na1}. We start with the following observation:
if $g\in L^{2}(\bbt,\mathcal{H})$ and $a_{0}(t),\ldots,a_{n}\left(t\right)$
are complex numbers, then the $m$-th Fourier coefficient of the function
\[
x\mapsto\sum_{k=0}^{n}a_{k}(t)g_{kt}(x)=\sum_{k=0}^{n}a_{k}(t)g(x+kt)
\]
equals 
\[
\widehat{g}(m)\cdot\sum_{k=0}^{n}a_{k}(t)e(ktm)=\widehat{g}(m)\cdot q_{t}\left(e(tm)\right),
\]
where $q_{t}(z)=\sum_{k=0}^{n}{\displaystyle a_{k}(t)z^{k}}$. Slightly
perturbing the coefficients $a_{k}(t)$, we may assume, without loss
of generality, that the coefficients $a_{0}(t)$ and $a_{n}(t)$ do
not vanish for $0<t<\tau$ (so that, for every $t$ in this range,
the polynomial $q_{t}$ is exactly of degree $n$ and does not vanish
at the origin), and that the arguments of the roots of $q_{t}$ are
all distinct.

By Parseval's theorem, 
\begin{equation}
\intsy_{\bbt}\,\left\Vert \sum_{k=0}^{n}a_{k}(t)g_{kt}(x)\right\Vert _{\mathcal{H}}^{2}\d x=\sum_{m\in\bbz}\bigl\|\widehat{g}(m)\bigr\|_{\mathcal{H}}^{2}\,\bigl|q_{t}\left(e(tm)\right)\bigr|^{2}.\label{eq:g_t_dual}
\end{equation}
If $g\in{\rm Exp}_{{\tt loc}}(n,\tau,\varkappa,\mathcal{H})$, then
we can choose $a_{0},\ldots,a_{k}$ so that the LHS of \eqref{eq:g_t_dual}
will be small for each $t\in(0,\tau)$. On the other hand, whenever
the norm of $\widehat{g}(m)$ is large, the RHS of \eqref{eq:g_t_dual}
can be small only when $q_{t}(e(tm))$ is small. The proof of Lemma~\ref{lemma:approx-spectrum}
will be based on two facts. The first is that, on average, $|q_{t}(e(tm))|$
is relatively large outside some exceptional set, which can be covered
by at most $n$ intervals of length $\tfrac{1}{4n(n+1)\tau}$. The
second is that there exists a $t_{0}$ such that $q_{t_{0}}(e(tm))$
can be effectively bounded from below on this exceptional set.

We start with a lemma on arithmetic progressions.
\begin{lem}
\label{lemma:arithm} Given a measurable set $G\subset\bbr_{+}$,
put 
\[
V_{G}=\left\{ t\in\bigl(\tfrac{1}{2}\tau,\tau\bigr)\,\colon\exists k\in\bbn{\rm \ s.t.\ }\tfrac{k}{t}\in G\right\} .
\]
Then $m(V_{G})<\tau^{2}m(G)$.

\end{lem}
This lemma shows that if $m(G)<\tfrac{1}{2\tau}$, then there are
significantly many points $t\in(\tfrac{1}{2}\tau,\tau)$ such that
no point $k/t$, $k\in\bbn$, belongs to $G$.

\noindent \medskip{}
\textit{Proof of Lemma~\ref{lemma:arithm}}: We have 
\[
\sum_{k\in\bbn}\indf_{G}\bigl(\tfrac{k}{t}\bigr)\ge\indf_{V_{G}}(t)\,.
\]
Integrating over $t\in\bigl(\tfrac{1}{2}\tau,\tau\bigr)$, we get
\begin{multline*}
m(V_{G})\le\intsy_{\tau/2}^{\tau}\sum_{k\in\bbn}\indf_{G}\bigl(\tfrac{k}{t}\bigr)\d t=\sum_{k\in\bbn}k\,\intsy_{k/\tau}^{2k/\tau}\indf_{G}(s)\,\frac{{\rm d}s}{s^{2}}\\
=\intsy_{0}^{\infty}\indf_{G}(s)\left(\sum_{s\tau/2<k<s\tau}k\right)\,\frac{{\rm d}s}{s^{2}}<\tau^{2}\intsy_{0}^{\infty}\indf_{G}(s)\d s=\tau^{2}m(G)\,,
\end{multline*}
because $\sum_{s\tau/2<k<s\tau}k<\tau^{2}s^{2}$. \hfill{}$\Box$

\noindent %

\medskip{}

\noindent The following lemma shows that the Fourier coefficients
$\widehat{g}(m)$ are small outside $n$ intervals of controlled length.
Put 
\[
\delta=\frac{1}{8n(n+1)}\,.
\]
This choice of $\delta$ will stay fixed till the end of the proof
of Lemma~\ref{lemma:approx-spectrum}.
\begin{lem}
\noindent \label{lem:crit_intervals} There exist $n$ intervals $I_{1},\ldots,I_{n}$
of length $\frac{2\delta}{\tau}$ such that 
\[
\sum_{m\in\bbz\backslash\bigcup I_{j}}\bigl\|\widehat{g}(m)\bigr\|_{\mathcal{H}}^{2}<\Bigl(\frac{C}{\delta}\Bigr)^{2n}\varkappa^{2}.
\]

\end{lem}
\noindent \textit{Proof of Lemma~\ref{lem:crit_intervals}}: By the
continuity of the shift in $L^{2}(\bbt,\mathcal{H})$, we can assume
that the coefficients $a_{k}\left(t\right)$ are piecewise constant
functions of $t$, and hence measurable. Then we can integrate Parseval's
formula~\eqref{eq:g_t_dual} over the interval $\left(0,\tau\right)$.
Recalling that the LHS of~\eqref{eq:g_t_dual} is less than $\varkappa^{2}$,
we get 
\[
\sum_{m\in\bbz}\bigl\|\widehat{g}(m)\bigr\|_{\mathcal{H}}^{2}\,\rho^{2}(m)<\varkappa^{2}\,,
\]
where 
\[
\rho^{2}(m)=\frac{1}{\tau}\,\intsy_{0}^{\tau}|q_{t}(e(tm))|^{2}\d t\,.
\]
Introduce the set 
\[
S=\left\{ m\in\bbz\,\colon\,\rho^{2}(m)<\frac{1}{4(n+1)}\Bigl(\frac{\delta}{A}\Bigr)^{2n}\right\} .
\]
Here and elsewhere in this section, $A$ is the positive numerical
constant from the RHS of the Turán-type Lemma~\ref{subsubsect:Turan}.
Then Lemma~\ref{lem:crit_intervals} will follow from the following
claim: 
\begin{multline}
S{\rm \ cannot\ contain\ }n+1{\rm \ integers\ }m_{1}<\cdots<m_{n+1}\ \mathrm{such\ that}\\
m_{j+1}-m_{j}>\frac{2\delta}{\tau}\,,\quad\forall j\in\{1,\ldots,n\}\,.\label{eq:cond_S}
\end{multline}
Indeed, this condition ensures that the set $S$ can be covered by
at most $n$ intervals $I_{1}$, \ldots{}, $I_{n}$ of length $2\delta/\tau$
and 
\[
\rho^{2}(m)\ge\frac{1}{4(n+1)}\Bigl(\frac{\delta}{A}\Bigr)^{2n},\qquad m\in\bbz\setminus\bigcup_{j}I_{j}\,,
\]
whence 
\[
\sum_{m\in\bbz\backslash\bigcup I_{j}}\bigl\|\widehat{g}(m)\bigr\|_{\mathcal{H}}^{2}\le4(n+1)\Bigl(\frac{A}{\delta}\Bigr)^{2n}\varkappa^{2}<\Bigl(\frac{C}{\delta}\Bigr)^{2n}\varkappa^{2}
\]
with some numerical constant $C$. Thus, we need to prove the claim~\eqref{eq:cond_S}.

\medskip{}
Suppose that \eqref{eq:cond_S} does not hold, i.e., there are $n+1$
integers $m_{1}<\cdots<m_{n+1}$ with $m_{j+1}-m_{j}>2\delta/\tau$
that belong to the set $S$. Then 
\begin{equation}
\intsy_{\tau/2}^{\tau}\sum_{j=1}^{n+1}\bigl|q_{t}(e(tm_{j}))\bigr|^{2}\,{\rm d}t<\frac{\tau}{4}\Bigl(\frac{\delta}{A}\Bigr)^{2n}\,.\label{eq:aver}
\end{equation}
We call the value $t\in\bigl(\tfrac{1}{2}\tau,\tau\bigr)$ \textit{bad}
if 
\[
\sum_{j=1}^{n+1}\bigl|q_{t}(e(tm_{j}))\bigr|^{2}<\Bigl(\frac{\delta}{A}\Bigr)^{2n}\,.
\]
Otherwise, the value $t$ is called \textit{good}. By~\eqref{eq:aver},
the measure of good $t$'s is less than $\tau/4$. In the rest of
the proof we will show that the measure of bad $t$'s is also less
than $\tau/4$, and this will lead us to a contradiction, which will
prove Lemma~\ref{lem:crit_intervals}.

\medskip{}
We will use the following
\begin{claim}
\label{claim:U} Let $q(z)=\sum_{k=0}^{n}a_{k}z^{k}$ with $\sum_{k=0}^{n}|a_{k}|^{2}=1$.
Given $\Delta\in(0,1)$, let 
\[
U=\Bigl\{ s\in\bbt\colon\bigl|q(e(s))\bigr|<\Bigl(\frac{\Delta}{A}\Bigr)^{n}\Bigr\}\,.
\]
Then the set $U$ is a union of at most $n$ intervals of length at
most $\Delta$ each.\end{claim}
\begin{proof}
$U$ is an open subset of $\bbt$ which consists of open intervals
(since $\Delta<1$ and $A\ge1$, we have that $U\ne\bbt$ ). The boundary
points of these intervals satisfy the equation $\bigl|q(e(s))\bigr|^{2}=\left(\frac{\Delta}{A}\right)^{2n}$,
which can be rewritten as 
\[
\left(\sum_{k=0}^{n}a_{k}z^{k}\right)\left(\sum_{k=0}^{n}\overline{a_{k}}z^{-k}\right)=\left(\frac{\Delta}{A}\right)^{2n},\qquad z=e(s)\,.
\]
The LHS of this equation is a rational function of degree at most
$2n$, and therefore the number of solutions is at most $2n$. Hence
$U$ consists of $l\le n$ intervals $J_{1}$, \ldots{}, $J_{l}$.

Next, note that since the sum of squares of the absolute values of
the coefficients of $q$ equals $1$, we have ${\displaystyle \max_{s\in\bbt}|q(e(s)|\ge1}$.
Then, applying Lemma~\ref{subsubsect:Turan}\index{Turán-type lemma}
to the exponential polynomial $s\mapsto q(e(s))$, we get 
\[
1\le\sup_{s\in\bbt}\bigl|q\left(e(s)\right)\bigr|\le\left(\frac{A}{m(J_{i})}\right)^{n}\cdot\sup_{s\in J_{i}}\bigl|q(e(s))\bigr|\le\left(\frac{\Delta}{m(J_{i})}\right)^{n}\,.
\]
Hence, $m(J_{i})\le\Delta$, proving the claim.
\end{proof}
\medskip{}
Note that in the proof of this claim we did not use the full strength
of Turán's lemma. For instance, we could have used the much simpler
Remez' inequality.

\medskip{}
Now for $t\in\bigl(\tfrac{1}{2}\tau,\tau\bigr)$ consider the set
\[
S_{t}=\Bigl\{ m\in\bbz\colon\left|q_{t}\left(e\left(tm\right)\right)\right|<\left(\frac{\delta}{A}\right)^{n}\Bigr\}\,.
\]
By the previous claim (applied with $\Delta=\delta$), there are points
$\xi_{1},\ldots,\xi_{n}\in\bbr$ (centers of the intervals $J_{i}$)
such that, for each $m\in S_{t}$, there exist $i\in\{1,\ldots,n\}$
and $l\in\bbz$ such that 
\begin{equation}
|tm-l-\xi_{i}|<\tfrac{1}{2}\,\delta\,.\label{eq:S_t}
\end{equation}
Suppose that the value $t$ is bad. Then the $n+1$ integers $m_{1},\ldots,m_{n+1}$
belong to the set $S_{t}$, and by the Dirichlet box principle, there
are two of these integers, say $m_{j'}$ and $m_{j''}$ with $j'<j''$,
which satisfy~\eqref{eq:S_t} with the same value $i$. Then for
this pair $|t(m_{j''}-m_{j'})-k|<\delta$, with some non-negative
integer $k$. Thus, 
\[
\Bigl|\frac{k}{t}-(m_{j''}-m_{j'})\Bigr|<\frac{\delta}{t}<\frac{2\delta}{\tau}\,.
\]
Note that since $m_{j''}-m_{j'}>\tfrac{2\delta}{\tau}$, the integer
$k$ must be positive. We conclude that the set of bad values $t$
is contained in the set $V_{G}$, where $G$ is the union of $\tfrac{1}{2}n(n+1)$
intervals of length $\tfrac{4\delta}{\tau}$ centered at all possible
differences $m_{j''}-m_{j'}$ with $j''>j'$. The measure of the set
$G$ is $\tfrac{n(n+1)}{2}\cdot\tfrac{4\delta}{\tau}$, which, due
to the choice of $\delta$, equals $\tfrac{1}{4\tau}$. By Lemma~\ref{lemma:arithm},
$m(V_{G})<\tau^{2}m(G)\le\tfrac{1}{4}\tau$. Thus, the measure of
the set of bad $t$'s is also less than $\tfrac{1}{4}\tau$, which
completes the proof of Lemma~\ref{lem:crit_intervals}.

\hfill{}$\Box$

\medskip{}
\textit{Proof of Lemma \ref{lemma:approx-spectrum}}: We need to find
a set $\Lambda=\Lambda_{g}\subset\bbr$ of $n$ frequencies such that
\[
\sum_{m\in\bbz}\|\,\widehat{g}(m)\,\|_{\mathcal{H}}^{2}\,\Theta_{\tau,\Lambda}^{2}(m)\le\bigl(Cn\bigr)^{4n}\varkappa^{2}\,,
\]
where 
\[
\Theta_{\tau,\Lambda}(m)=\prod_{\lambda\in\Lambda}\theta_{\tau}(m-\lambda)\,,\qquad\theta_{\tau}(m)=\min(1,\tau|m|).
\]
By Lemma~\ref{lem:crit_intervals}, there exists a collection of
$n$ intervals $\{I_{j}\}$, each of length $\tfrac{2\delta}{\tau}$,
such that 
\[
\sum_{m\in\bbz\setminus\bigcup_{j}I_{j}}\|\,\widehat{g}(m)\,\|_{\mathcal{H}}^{2}\,\Theta_{\tau,\Lambda}^{2}(m)\overset{\Theta\le1}{\le}\sum_{m\in\bbz\setminus\bigcup_{j}I_{j}}\|\,\widehat{g}(m)\,\|_{\mathcal{H}}^{2}\le\left(Cn\right)^{4n}\varkappa^{2}\,.
\]
Therefore, it remains to estimate the sum 
\[
\sum_{m\in\bigcup_{j}I_{j}}\|\,\widehat{g}(m)\,\|_{\mathcal{H}}^{2}\,\Theta_{\tau,\Lambda}^{2}(m)\,.
\]
By Parseval's identity~\eqref{eq:g_t_dual}, for every $t\in(0,\tau)$,
\[
\sum_{m\in\bigcup_{j}I_{j}}\|\,\widehat{g}(m)\,\|_{\mathcal{H}}^{2}\,|q_{t}(e(tm))|^{2}<\varkappa^{2}\,.
\]
Hence, it suffices to show that \textit{there exist a value} $t_{0}\in(0,\tau)$
\textit{and a set} $\Lambda$ \textit{of} $n$ \textit{real numbers
such that} $|q_{t_{0}}(e(t_{0}m))|\ge\delta^{n}\Theta_{\tau,\Lambda}(m)$
\textit{for every} $m\in\bigcup_{j}I_{j}$.

\noindent First, we bound the absolute value of the polynomial $q_{t}$
from below by the absolute value of another polynomial $p$ whose
zeros are obtained from the zeros of $q_{t}$ by radial projection
to the unit circle.
\begin{claim}
\noindent \label{clm:zero_normalization} Let $z_{j}\ne0$ for $1\le j\le n$,
and let $g(z)=c\cdot\prod_{j=1}^{n}{\displaystyle \left(z-z_{j}\right)}$
be a polynomial of degree $n$ such that ${\displaystyle \sup_{|z|=1}\left|g\left(z\right)\right|\ge1}$.
Let $h\left(z\right)=\prod_{j=1}^{n}{\displaystyle \left(z-\zeta_{j}\right)}$,
where $\zeta_{j}=z_{j}/|z_{j}|$. Then, for every $z\in\bbt$, 
\[
\left|h\left(z\right)\right|\le2^{n}\left|g\left(z\right)\right|.
\]
\end{claim}
\begin{proof}
The ratio $\left|\frac{z-\zeta_{j}}{z-z_{j}}\right|$ attains its
maximum on $\{|z|=1\}$ at the point $z=-\zeta_{j}$, where it is
equal to $\frac{2}{1+\left|z_{j}\right|}$. Therefore, 
\[
\left|\frac{h\left(z\right)}{g\left(z\right)}\right|\le\frac{1}{\left|c\right|}\prod_{j=1}^{n}\frac{2}{1+\left|z_{j}\right|}.
\]
By our assumption, there is some $z^{\prime}$, $|z'|=1$, such that
$\left|g\left(z^{\prime}\right)\right|\ge1$. Hence, 
\[
1\le\left|c\right|\prod_{j=1}^{n}\left|z^{\prime}+z_{j}\right|\le\left|c\right|\prod_{j=1}^{n}\left(1+\left|z_{j}\right|\right).
\]
Overall, we have 
\[
\left|h\left(z\right)\right|\le2^{n}\left|g\left(z\right)\right|\cdot\frac{1}{\left|c\right|}\cdot\prod_{j=1}^{n}\frac{1}{1+\left|z_{j}\right|}\le2^{n}\left|g\left(z\right)\right|,
\]
proving the claim.
\end{proof}
\medskip{}
Recall that $\sup_{\left|z\right|=1}\left|q_{t}\left(z\right)\right|\ge1$.
Hence, applying Claim~\ref{clm:zero_normalization}, we conclude
that $|q_{t}(z)|\ge2^{-n}\,|p_{t}(z)|$ for $|z|=1$, where $p_{t}$
is a monic polynomial of degree $n$ with all its zeros on the unit
circle.

\medskip{}
To choose $t_{0}$, we consider $n$ intervals $\widetilde{I}_{j}$
of length $4\delta\tau^{-1}$ with the same centers as the intervals
$I_{j}$ of Lemma~\ref{lem:crit_intervals}, and put $\widetilde{S}=\bigcup_{j}\widetilde{I}_{j}$.
Let $\widetilde{G}=\widetilde{S}-\widetilde{S}$ be the difference
set, with $m(\widetilde{G})\le8\delta\tau^{-1}\cdot n^{2}$. We call
the value $t\in\bigl(\tfrac{1}{2}\tau,\tau\bigr)$ \textit{bad} if
there exists an integer $k\ne0$ such that $k/t\in\widetilde{G}$.
Since the set $\widetilde{G}$ is symmetric with respect to $0$,
we can estimate the measure of bad $t$'s by applying Lemma~\ref{lemma:arithm}
to the set $\widetilde{G}\cap\bbr_{+}$. Then the measure of bad values
of $t$ is less than $\tau^{2}\cdot\tfrac{1}{2}\, m(\widetilde{G})\le4\delta\tau\cdot n^{2}<\tfrac{1}{2}\tau$,
since $\delta\cdot8n^{2}<1$. Therefore, there exists at least one
\textit{good} value $t_{0}\in\bigl(\tfrac{1}{2}\tau,\tau\bigr)$ for
which \textit{every arithmetic progression with difference $t_{0}^{-1}$
has at most one point in $\widetilde{S}$}. We fix this value $t_{0}$
till the end of the proof.

\medskip{}
To simplify notation, we put $p=p_{t_{0}}$. The zero set of the function
$x\mapsto p(e(t_{0}x))$ consists of $n$ arithmetic progressions
with difference $t_{0}^{-1}$. By the choice of $t_{0}$, at most
$n$ zeros of this function belong to the set $\widetilde{S}$. We
denote these zeros by $\lambda_{1},\ldots,\lambda_{l}$, $\ell\le n$.
If $\ell<n$, we choose $n-\ell$ zeros $\lambda_{\ell+1},\ldots,\lambda_{n}$
in $\bbr\setminus\widetilde{S}$ so that $\left\{ e(t_{0}\lambda_{j})\right\} _{1\le j\le n}$
is a complete set of zeros of the algebraic polynomial $p$; we recall
that these zeros are all distinct.

It remains to define a set $\Lambda$ of $n$ numbers, and to estimate
from below $|p(e(t_{0}m))|$ when $m\in\bigcup_{j}I_{j}$. Denote
by $d_{j}(m)$ the distance from the integer $m$ to the nearest point
in the arithmetic progression $\left\{ \lambda_{j}+kt_{0}^{-1}\right\} _{k\in\bbz}$.
We have 
\[
\bigl|p(e(t_{0}m))\bigr|=2^{n}\prod_{j=1}^{n}\bigl|\sin(\pi t_{0}(m-\lambda_{j})\bigr|\ge2^{n}\prod_{j=1}^{n}\bigl(2t_{0}\, d_{j}(m)\bigr)\ge2^{n}\tau^{n}\prod_{j=1}^{n}d_{j}(m)\,.
\]
We put $\Lambda=\left\{ \lambda_{j}\right\} _{1\le j\le n}$. Recall
that here $m\in\bigcup_{j}I_{j}$, $\widetilde{S}=\bigcup_{j}\widetilde{I}_{j}$,
and that the arithmetic progression $\left\{ \lambda_{j}+kt_{0}^{-1}\right\} _{k\in\bbz}$
either misses the set $\widetilde{S}$, or has at most one element
in $\widetilde{S}$. In the first case, we get $d_{j}(m)\ge\delta\tau^{-1}$,
while in the second case, $d_{j}(m)\ge\min\bigl\{\tfrac{\delta}{\tau},\left|m-\lambda_{j}\right|\bigr\}$.
Therefore, in both cases, 
\[
d_{j}\left(m\right)\ge\min\Bigl\{\frac{\delta}{\tau},\left|m-\lambda_{j}\right|\Bigr\}\stackrel{\delta\le\frac{1}{2}}{\ge}\frac{\delta}{\tau}\min\bigl\{1,\tau\left|m-\lambda_{j}\right|\bigr\}=\frac{\delta}{\tau}\cdot\theta_{\tau}(m-\lambda_{j})\,.
\]

Tying the ends together, we get 
\begin{multline*}
\bigl|q_{t_{0}}(e(t_{0}m))\bigr|\ge2^{-n}\bigl|p(e(t_{0}m))\bigr|\ge2^{-n}\cdot2^{n}\tau^{n}\prod_{j=1}^{n}d_{j}(m)\\
\ge\tau^{n}\cdot\Bigl(\frac{\delta}{\tau}\Bigr)^{n}\Theta_{\tau,\Lambda}(m)=\delta^{n}\Theta_{\tau,\Lambda}(m)\,.
\end{multline*}
This completes the proof of Lemma~\ref{lemma:approx-spectrum}. \hfill{}$\Box$

\subsection{Local approximation by exponential polynomials with $n$ terms \label{subsect:local-approx}}

Henceforth we will assume that $\calH=L^{2}(\Omega)$. Then $\Exploc(n,\tau,\varkappa,L^{2}(\Omega))\subset L^{2}(Q)$.

For a finite set $\Lambda\subset\bbr$, denote by ${\rm Exp}(\Lambda,\Omega)$
the linear space of exponential polynomials with frequencies in $\Lambda$
and with coefficients depending on $\omega$. %
\begin{comment}
and with coefficients that are measurable functions on $\Omega$. 
\end{comment}
The next lemma shows that, for a.e. $\omega\in\Omega$, the function
$\theta\mapsto g(\omega,\theta)$, $g\in\Exploc(n,\tau,\varkappa,L^{2}(\Omega))$,
can be well approximated by exponential polynomials from ${\rm Exp}(\Lambda,\Omega)$,
on intervals $J\subset[0,1)$ of length comparable with $\tau$.

Suppose that $M>1$ satisfies 
\[
\ell=\frac{1}{M\tau}\in\bbn\,,
\]
and partition $\bbt$ into $l$ intervals of length $M\tau$: 
\[
\bbt=\bigcup_{k=0}^{\ell-1}\left[\frac{k}{\ell},\frac{k+1}{\ell}\right)\,.
\]

\begin{lem}
\label{lemma:loc-approx} Let $M$ be as above and let $g\in\Exploc(n,\tau,\varkappa,L^{2}(\Omega))$.
There exists a non-negative function $\Phi\in L^{2}(Q)$ with 
\[
\|\Phi\|_{2}\le\bigl(Cn\bigr)^{2n}\varkappa,
\]
and with the following property: \smallskip{}
for every interval $J\subset\bbt$ in the above partition there exists
an exponential polynomial $p^{J}\in{\rm Exp}(\Lambda_{g},\Omega)$
such that, for a.e. $\omega\in\Omega$ and a.e. $\theta\in J$, 
\[
\bigl|g(\omega,\theta)-p^{J}(\omega,\theta)\bigr|\le M^{n}\,\Phi(\omega,\theta)\,.
\]

\end{lem}
The proof of Lemma~\ref{lemma:loc-approx} is very close to the proof
of the corresponding result in~\cite[Section~3.2]{Na1}. We start
with a lemma on solutions of ordinary differential equations (cf.
Lemma~3.2 in~\cite{Na1}).
\begin{lem}
\label{lem:sol_of_diff_eq_in_L2_Q} Let 
\[
D=\prod_{j=1}^{n}e(\lambda_{j}x)\,\frac{{\rm d}}{{\rm d}x}\, e(-\lambda_{j}x)\,\qquad\lambda_{1},\ldots,\lambda_{n}\in\bbr\,,\quad\lambda_{i}\ne\lambda_{j}\mbox{ for }i\ne j\,,
\]
be a differential operator of order $n\ge1$, and let $J\subset[0,1]$
be an interval. Suppose that $f\in L^{2}(\Omega\times J)$ and, for
a.e. $\omega\in\Omega$, $x\mapsto f(\omega,x)$ is a $C^{n}(J)$-function
satisfying the differential equation $Df=h$ with $h\in L^{2}(\Omega\times J)$.
Then there exists an exponential polynomial $p$ with spectrum $\lambda_{1},\ldots,\lambda_{n}$,
such that, for a.e. $\omega\in\Omega$, 
\[
\sup_{x\in J}|f(\omega,x)-p(\omega,x)|\le m(J)^{n}\,\frac{1}{m(J)}\,\intsy_{J}|h(\omega,x)|\,{\rm d}x\,.
\]
\end{lem}
\begin{proof}
Let $\phi$ be a particular solution of the equation $D\phi=h$ constructed
by repeated integration: 
\[
\phi=\left(\prod_{j=1}^{n}e(\lambda_{j}x)\,\mathcal{J}e(-\lambda_{j}x)\right)h
\]
where $\mathcal{J}$ is the integral operator 
\[
\bigl(\mathcal{J}\psi\bigr)(\omega,x)=\intsy_{a}^{x}\psi(\omega,t)\,{\rm d}t
\]
and $a$ is the left end-point of the interval $J$. Then, for a.e.
$\omega$ 
\[
|\phi(\omega,x)|\le m(J)^{n}\,\frac{1}{m(J)}\intsy_{J}|h(\omega,x)|\,{\rm d}x\,.
\]
The function $f-\phi$ satisfies the homogeneous equation $D(f-\phi)=0$.
Hence, $p=f-\phi$ is an exponential polynomial\index{exponential polynomial!random}
with coefficients depending on $\omega$: 
\[
p(\omega,x)=\sum_{j=1}^{n}c_{j}(\omega)e(\lambda_{j}x)\,.
\]

\end{proof}
\bigskip{}
Now we turn to the proof of~Lemma~\ref{lemma:loc-approx}. Fix $g\in\Exploc(n,\tau,\varkappa,L^{2}(\Omega))$.
By Lemma~\ref{lemma:approx-spectrum}, the function $g$ has an `approximate
spectrum' $\Lambda=\Lambda_{g}=\left\{ \lambda_{j}\right\} _{1\le j\le n}$
so that 
\[
\sum_{m\in\bbz}\|\,\widehat{g}(m)\,\|_{L^{2}(\Omega)}^{2}\,\Theta_{\tau,\Lambda}^{2}(m)\le\bigl(Cn\bigr)^{4n}\varkappa^{2}\,,
\]
with 
\[
\Theta_{\tau,\Lambda}(m)=\prod_{\lambda\in\Lambda}\theta_{\tau}(m-\lambda)\,,\qquad\theta_{\tau}(m)=\min(1,\tau|m|)\,.
\]
We fix $M>1$ so that $1/(M\tau)$ is a positive integer, and partition
$\bbt$ into intervals $J$ of length $M\tau$.

Put 
\[
I_{k}=\left(\lambda_{k}-\tfrac{1}{\tau},\lambda_{k}+\tfrac{1}{\tau}\right),\ \widetilde{I}_{k}=\left(\lambda_{k}-\tfrac{2}{\tau},\lambda_{k}+\tfrac{2}{\tau}\right),\ E_{0}=\bbr\setminus\bigcup_{k=1}^{n}I_{k}\,,\ E_{k}=I_{k}\setminus\bigcup_{j=1}^{k-1}I_{j}\,.
\]
The sets $E_{k}$, $0\le k\le n$, form a partition of the real line.
Accordingly, we decompose $g$ into the sum $g=\sum_{k=0}^{n}g_{k}$,
where $g_{k}$ is the projection of $g$ onto the closed subspace
of $L^{2}(Q)$ that consists of functions with spectrum contained
in $E_{k}$. For each $k=0,\ldots,n$, we have 
\begin{equation}
\sum_{m\in\bbz}\left\Vert \widehat{g_{k}}(m)\right\Vert _{L^{2}(\Omega)}^{2}\Theta_{\tau,\Lambda}^{2}(m)<(Cn)^{4n}\varkappa^{2}\eqdef\widetilde{\varkappa}^{2}\,.\label{eq:g_k}
\end{equation}
Since, $\Theta_{\tau,\Lambda}^{2}(m)\equiv1$ for $m\in E_{0}$, we
get $\|g_{0}\|_{L^{2}(Q)}\le\widetilde{\varkappa}$.

Now let $1\le k\le n$. Let $n_{k}$ denote the number of points $\lambda_{j}$
lying in $\widetilde{I}_{k}$. We define a differential operator $D_{k}$
of order $n_{k}$ by 
\[
D_{k}\eqdef\prod_{\lambda_{j}\in\widetilde{I}_{k}}e(\lambda_{j}x)\frac{{\rm d}}{{\rm d}x}\, e(-\lambda_{j}x).
\]
The function $g_{k}(x)$ is a trigonometric polynomial with coefficients
depending on $\omega$, hence, for a.e. $\omega$, it is an infinitely
differentiable function of $x$. We set $h_{k}\eqdef D_{k}g_{k}$.
Note that this is a trigonometric polynomial with the same frequencies
as $g_{k}$: 
\[
\widehat{h}_{k}(\omega,m)=\left(2\pi{\rm i}\right)^{n_{k}}\widehat{g}_{k}(\omega,m)\prod_{\lambda_{j}\in\widetilde{I}_{k}}\left(m-\lambda_{j}\right)\,.
\]
Consequently, 
\[
\bigl|\widehat{h}_{k}(\omega,m)\bigr|=(2\pi)^{n_{k}}\,\bigl|\widehat{g}_{k}(\omega,m)\bigr|\prod_{\lambda_{j}\in\widetilde{I}_{k}}|m-\lambda_{j}|\,.
\]
In the product on the RHS, $m\in E_{k}\subset I_{k}$ and $\lambda_{j}\in\widetilde{I}_{k}$.
Recalling the definition of the function $\theta_{\tau}$, we see
that 
\[
|m-\lambda_{j}|\le\tfrac{3}{\tau}\,\theta_{\tau}(m-\lambda_{j})\qquad{\rm for\ }m\in I_{k},\ \lambda_{j}\in\widetilde{I}_{k}\,.
\]
Therefore, 
\[
\bigl|\widehat{h}_{k}(\omega,m)\bigr|\le\Bigl(\frac{6\pi}{\tau}\Bigr)^{n_{k}}\bigl|\widehat{g}_{k}(\omega,m)\bigr|\prod_{\lambda_{j}\in\widetilde{I}_{k}}\theta_{\tau}(m-\lambda_{j})\,.
\]
Note that for $m\in E_{k}$ and for $\lambda_{j}\in\bbz\setminus\widetilde{I}_{k}$,
we have $\theta_{\tau}(m-\lambda_{j})=1$. Thus, 
\[
\bigl|\widehat{h}_{k}(\omega,m)\bigr|\le\Bigl(\frac{6\pi}{\tau}\Bigr)^{n_{k}}\bigl|\widehat{g}_{k}(\omega,m)\bigr|\Theta_{\tau,\Lambda}(m)\,,\qquad\omega\in\Omega\,,
\]
whence, recalling estimate~\eqref{eq:g_k}, we obtain 
\[
\|h_{k}\|_{L^{2}(Q)}\le\Bigl(\frac{6\pi}{\tau}\Bigr)^{n_{k}}\,\widetilde{\varkappa}\,.
\]

Applying Lemma \ref{lem:sol_of_diff_eq_in_L2_Q} to an interval $J$
of length $M\tau$, we obtain an exponential polynomial $p_{k}^{J}$
with spectrum consisting of frequencies $\lambda_{j}\in\widetilde{I}_{k}$
and with coefficients depending on $\omega$, such that, for every
$x\in J$ and almost every $\omega\in\Omega$, 
\[
\bigl|g_{k}(\omega,x)-p_{k}^{J}(\omega,x)\bigr|\le(M\tau)^{n_{k}}\cdot\frac{1}{M\tau}\intsy_{J}|h_{k}(\omega,t)|\,{\rm d}t\,.
\]
We denote by 
\[
\HLMaxF f(\omega,x)=\sup_{L\colon x\in L}\frac{1}{m(L)}\intsy_{L}\left|f(\omega,t)\right|\d t
\]
the Hardy-Littlewood maximal function. The supremum is taken over
all intervals $L\subset[0,1]$ containing $x$, but it is easy to
see that it is enough to restrict ourselves to the intervals with
rational endpoints, which allows us to rewrite $\HLMaxF f$ as $\sup\bigl\{ F_{\alpha,\beta}\colon\alpha,\beta\in\bbq\bigr\}$,
where 
\[
F_{\alpha,\beta}(\omega,x)=\indf_{[\alpha,\beta]}(x)G_{\alpha,\beta}(\omega)\quad{\rm and}\quad G_{\alpha,\beta}(\omega)=\frac{1}{\beta-\alpha}\intsy_{\alpha}^{\beta}|f(t,\omega)|\d t\,.
\]
By the Fubini theorem, $G_{\alpha,\beta}$ are measurable functions
on $\Omega$, so $F_{\alpha,\beta}$ are measurable functions on $Q$,
and consequently, $\HLMaxF$ is measurable on $Q$ as well.

Let $\widetilde{h}_{k}=\tau^{n_{k}}h_{k}$. Then 
\[
\left|g_{k}\left(\omega,x\right)-p_{k}^{J}\left(\omega,x\right)\right|\le M^{n_{k}}\cdot\HLMaxF\widetilde{h}_{k}(\omega,x)\stackrel{M>1}{\le}M^{n}\cdot\HLMaxF\widetilde{h}_{k}(\omega,x)\,.
\]
Using the classical estimate for the $L^{2}$-norm of the maximal
function, we get, for a.e. $\omega$, 
\[
\intsy_{\bbt}\left[\HLMaxF\widetilde{h}_{k}(\omega,x)\right]^{2}\d x\le C\,\intsy_{\bbt}\left|\widetilde{h}_{k}(\omega,x)\right|^{2}\d x\,.
\]
Recalling that $\|\widetilde{h}_{k}\|_{L^{2}(Q)}<C^{n_{k}}\widetilde{\varkappa}$,
we obtain 
\[
\left\Vert \HLMaxF\widetilde{h}_{k}\right\Vert _{L^{2}(Q)}^{2}=\intsy_{\Omega\times\bbt}\left[\HLMaxF\widetilde{h}_{k}(\omega,x)\right]^{2}\,{\rm d}x\,{\rm d}\mathcal{P}(\omega)\le C\,\|\widetilde{h}_{k}\|_{L^{2}(Q)}^{2}\le C^{2n_{k}}\widetilde{\varkappa}^{2}\,.
\]

We now set $p^{J}\eqdef\sum_{k=1}^{n}p_{k}^{J}$. Notice that all
the frequencies of the polynomial $p^{J}$ belong to the set $\Lambda_{g}$.
Then, for every $x\in J$, 
\begin{eqnarray*}
\left|g\left(\omega,x\right)-p^{J}\left(\omega,x\right)\right| & \le & \left|g_{0}\left(\omega,x\right)\right|+\sum_{k=1}^{n}\left|g_{k}\left(\omega,x\right)-p_{k}^{J}\left(\omega,x\right)\right|\\
 & \le & \left|g_{0}\left(\omega,x\right)\right|+M^{n}\sum_{k=1}^{n}\HLMaxF\widetilde{h}_{k}(\omega,\theta)\\
 & \le & M^{n}\left(\left|g_{0}\left(\omega,x\right)\right|+\sum_{k=1}^{n}\HLMaxF\widetilde{h}_{k}(\omega,x)\right)\eqdef M^{n}\Phi\left(\omega,x\right).
\end{eqnarray*}
It remains to bound the norm of the `error function' $\Phi$: 
\[
\bigl\|\Phi\bigr\|_{L^{2}(Q)}\le\bigl\| g_{0}\bigr\|_{L^{2}(Q)}+\sum_{k=1}^{n}\left\Vert \HLMaxF\widetilde{h}_{k}\right\Vert _{L^{2}(Q)}\le\widetilde{\varkappa}+\sum_{k=1}^{n}C^{n_{k}}\widetilde{\varkappa}\le C^{n}\widetilde{\varkappa}\le\left(Cn\right)^{2n}\varkappa\,.
\]
This proves the desired result. \hfill{}$\Box$

\subsection{The Spreading Lemma \label{subsect:spreading}\index{spreading lemma}}

The next lemma is the crux of the proof of Theorem~\ref{thm:Zygmund-type-final}.
Given a set $E\subset Q$ of positive measure, we put $\Delta_{t}(E)=\mu\bigl((E+t)\setminus E\bigr)$.
\begin{lem}
\label{lemma:spreading} Suppose $g\in{\rm Exp}_{{\tt loc}}(n,\tau,\varkappa,L^{2}(\Omega))$
and $E\subset Q$ is a set of positive measure. There exists a set
$\widetilde{E}\supset E$ of measure $\mu(\widetilde{E})\ge\mu(E)+\tfrac{1}{2}\Delta_{n\tau}(E)$
such that, for each $b\in L^{2}(\Omega)$, 
\[
\intsy_{\widetilde{E}}|g-b|^{2}\,{\rm d}\mu\le\left(\frac{Cn^{3}}{\Delta_{n\tau}^{2}(E)}\right)^{2n+1}\left(\intsy_{E}|g-b|^{2}\,{\rm d}\mu+\varkappa^{2}\right)\,.
\]

\end{lem}
This lemma follows from Lemma \ref{lemma:loc-approx} combined with
the Turán-type estimate \eqref{eq:turan_lt_est}. We now turn to the
proofs of the lemmas. Till the end of this section, we fix the function
$g\in\Exploc(n,\tau,\varkappa,L^{2}(\Omega))$, the set $E\subset Q$
of positive measure, and the `random constant' $b\in L^{2}(\Omega)$.

We will use two parameters, $M>1$, with $\tfrac{1}{M\tau}\in\bbn$,
and $\gamma\in(0,1)$; their specific values will be chosen later
in the proof.

\medskip{}

\begin{defn*}
Let $J$ be an interval of length $M\tau$ in the partition of $\bbt$.
The interval $J$ is called \textit{$\omega$-white} if $m(J\cap E_{\omega})\ge\gamma m(J)$;
otherwise, it is called \textit{$\omega$-black}.
\end{defn*}
\noindent \medskip{}
Given $\omega$, the union of all $\omega$-white intervals will be
denoted by $W_{\omega}$. By $W\subset Q$ we denote the union of
all sets $W_{\omega}$. Similarly, we denote by $B_{\omega}$ the
union of all $\omega$-black intervals and by $B\subset Q$ the union
of all sets $B_{\omega}$. Since we can write the set $W$ as 
\[
\bigcup_{J}\bigl\{\omega\colon m(J\cap E_{\omega})\ge\gamma m(J)\bigr\}\times J
\]
and the function $\omega\mapsto m(J\cap E_{\omega})$ is measurable
on $\Omega$ for every interval $J$ in the partition, we see that
$W$ and $B=Q\setminus W$ are measurable subsets of $Q$.

\medskip{}
Let $\Phi$ be the error function given by the Local Approximation
Lemma. The next lemma enables us to extend our estimates for $g-b$
from the set $E$ to the set $W$. 
\begin{lem}
\label{lem:est_for_l2_f_in_white_parts} We have 
\[
\intsy_{W}\left|g-b\right|^{2}\d{\mu}\le\left(\frac{C}{\gamma}\right)^{2n+1}\left[\intsy_{W\cap E}|g-b|^{2}\d{\mu}+M^{2n+1}\intsy_{W}\Phi^{2}\d{\mu}\right].
\]
\end{lem}
\begin{proof}
Let $J$ be one of the $\omega$-white intervals of length $M\tau$.
By~Lemma \ref{lemma:loc-approx}, for almost every $\omega\in\Omega$
and every $\theta\in J$, we have 
\[
\left|(g(\omega,\theta)-b(\omega))-(p^{J}(\omega,\theta)-b(\omega))\right|=\left|g\left(\omega,\theta\right)-p^{J}\left(\omega,\theta\right)\right|\le M^{n}\Phi(\omega,\theta),
\]
where $p^{J}$ is an exponential polynomial with $n$ frequencies
and coefficients depending on $\omega$. Therefore, 
\begin{equation}
\intsy_{J}\left|g-b\right|^{2}\d{\theta}\le2\left(\intsy_{J}\left|p^{J}-b\right|^{2}\d{\theta}+M^{2n}\intsy_{J}\Phi^{2}\d{\theta}\right).\label{eq:approx_for_f_min_b_on_J}
\end{equation}
Applying the $L^{2}$-version of the Turán-type lemma to the exponential
polynomial $p^{J}-b$, which has at most $n+1$ frequencies, we get
\begin{align*}
\intsy_{J}\left|p^{J}-b\right|^{2}\d{\theta} & \le\left(\frac{C\, m(J)}{m(J\cap E_{\omega})}\right)^{2n+1}\intsy_{J\cap E_{\omega}}\left|p^{J}-b\right|^{2}\d{\theta}\\
 & \le\left(\frac{C}{\gamma}\right)^{2n+1}\intsy_{J\cap E_{\omega}}\left|p^{J}-b\right|^{2}\d{\theta}\,.
\end{align*}
Plugging this into \eqref{eq:approx_for_f_min_b_on_J}, we find that
\[
\intsy_{J}\left|g-b\right|^{2}\d{\theta}\le\left(\frac{C}{\gamma}\right)^{2n+1}\intsy_{J\cap E_{\omega}}\left|p^{J}-b\right|^{2}\d{\theta}+2M^{2n}\intsy_{J}\Phi^{2}\d{\theta}\,.
\]
Summing these estimates over all $\omega$-white intervals $J$, and
using that 
\[
|p^{J}-b|\le|g-b|+|g-p_{J}|\le|g-b|+M^{n}\Phi\,,
\]
we get 
\[
\intsy_{W_{\omega}}\left|g-b\right|^{2}\d{\theta}\le\left(\frac{C}{\gamma}\right)^{2n+1}\left[\intsy_{W_{\omega}\cap E_{\omega}}\left|g-b\right|^{2}\d{\theta}+M^{2n}\intsy_{W_{\omega}}\Phi^{2}\d{\theta}\right]
\]
Integrating over $\omega$, we get the result.
\end{proof}
\medskip{}
The effectiveness of this lemma depends on the size of the set $W\cap E^{{\rm c}}$.
The following lemma is very similar to Lemma~3.4 from~\cite{Na1}.
For the reader's convenience, we reproduce its proof. Recall that
$\Delta_{n\tau}(E)=\mu\left((E+n\tau)\setminus E\right)$.
\begin{lem}
\label{lem:white_part_low_bnd} For $\gamma<\tfrac{1}{2}$, 
\[
\mu(W\cap E^{{\rm c}})\ge\Delta_{n\tau}(E)-\left(\gamma+\frac{n}{M}\right).
\]
\end{lem}
\begin{proof}
We have 
\begin{align*}
m\bigl((E_{\omega}+n\tau)\setminus E_{\omega}\bigr) & =m\bigl((E_{\omega}+n\tau)\cap E_{\omega}^{{\rm c}}\bigr)\\
 & =m\bigl((E_{\omega}+n\tau)\cap E_{\omega}^{{\rm c}}\cap W_{\omega}\bigr)+m\bigl((E_{\omega}+n\tau)\cap E_{\omega}^{{\rm c}}\cap B_{\omega}\bigr)\\
 & \le m\bigl(W_{\omega}\cap E_{\omega}^{{\rm c}}\bigr)+m\bigl((E_{\omega}+n\tau)\cap E_{\omega}^{{\rm c}}\cap B_{\omega}\bigr).
\end{align*}
We need to estimate the second term on the RHS. If the interval $J$
is $\omega$-black, then 
\begin{align*}
m\bigl(J\cap E_{\omega}^{{\rm c}}\cap(E_{\omega}+n\tau)\bigr) & \le m\bigl(J\cap(E_{\omega}+n\tau)\bigr)\\
 & \le m\bigl(J\setminus(J+n\tau)\bigr)+m\bigl((E_{\omega}+n\tau)\cap(J+n\tau)\bigr)\\
 & \le n\tau+m\bigl(E_{\omega}\cap J\bigr)\\
 & <n\tau+\gamma m(J)=\Bigl(\frac{n\tau}{m(J)}+\gamma\Bigr)m(J)\,.
\end{align*}
Summing this inequality over all $\omega$-black intervals $J$, and
recalling that $m(J)=M\tau$, we obtain 
\[
m\bigl((E_{\omega}+n\tau)\cap E_{\omega}^{{\rm c}}\cap B_{\omega}\bigr)\le\Bigl(\frac{n\tau}{M\tau}+\gamma\Bigr)\cdot m(B_{\omega})\le\frac{n}{M}+\gamma\,.
\]
Integrating over $\Omega$ we get the required result.
\end{proof}
\medskip{}

\noindent \textit{Proof of Lemma~\ref{lemma:spreading}}: We write
$\Delta=\Delta_{n\tau}(E)$ and put 
\[
M_{1}=\frac{8n}{\Delta}.
\]
We consider two cases, according to whether $M_{1}\tau\le1$ or not.

In the first case, we choose $M\in\left[M_{1},2M_{1}\right]$ so that
$1/(M\tau)$ is an integer. Notice that $M>1$. We set $\gamma=\tfrac{1}{8}\Delta<\tfrac{1}{2}$
and let $\widetilde{E}=E\cup\left(W\cap E^{{\rm c}}\right)=E\cup W$,
where $W$ is the union of the corresponding white intervals. By Lemma~\ref{lem:white_part_low_bnd},
\[
\mu(W\cap E^{{\rm c}})\ge\Delta-\Bigl(\gamma+\frac{n}{M}\Bigr)\ge\Delta-\left(\frac{\Delta}{8}+\frac{\Delta}{8}\right)>\frac{\Delta}{2}\,.
\]
Furthermore, using Lemma~\ref{lem:est_for_l2_f_in_white_parts},
we get 
\[
\intsy_{W}\left|g-b\right|^{2}\d{\mu}\le\left(\frac{C}{\gamma}\right)^{2n+1}\left[\intsy_{W\cap E}\left|g-b\right|^{2}\d{\mu}+M^{2n}\intsy_{W}\,\Phi^{2}\d{\mu}\right].
\]
Inserting here the values of the parameters $\gamma$ and $M$ and
taking into account the bound on the norm of $\Phi$, we find that
the RHS is 
\begin{eqnarray*}
 & \le & \left(\frac{C}{\Delta}\right)^{2n+1}\left[\intsy_{W\cap E}\left|g-b\right|^{2}\d{\mu}+\left(\frac{C\, n}{\Delta}\right)^{2n}\,\intsy_{W}\Phi^{2}\d{\mu}\right]\\
 & \le & \left(\frac{C}{\Delta}\right)^{2n+1}\left[\intsy_{E}\left|g-b\right|^{2}\d{\mu}+\left(\frac{C\, n^{3}}{\Delta}\right)^{2n}\,\varkappa^{2}\right]\\
 & \le & \left(\frac{C\, n^{3}}{\Delta^{2}}\right)^{2n+1}\left[\intsy_{E}\left|g-b\right|^{2}\d{\mu}+\,\varkappa^{2}\right].
\end{eqnarray*}

Now we consider the second case, when $M_{1}\tau>1$. We set $M=\tfrac{1}{\tau}$
(that is, there is only one interval in the `partition') and note
that 
\[
M=\frac{1}{\tau}<M_{1}=\frac{8n}{\Delta}\,.
\]
We set $\gamma=\frac{\Delta}{2}$, and once again $\widetilde{E}=E\cup\left(W\cap E^{{\rm c}}\right)=E\cup W$.
Similarly to the first case, Lemma~\ref{lem:est_for_l2_f_in_white_parts}
gives us 
\[
\intsy_{W}\left|g-b\right|^{2}\d{\mu}\le\left(\frac{C\, n^{3}}{\Delta^{2}}\right)^{2n+1}\left[\intsy_{E}\left|g-b\right|^{2}\d{\mu}+\,\varkappa^{2}\right].
\]
We now show that there are sufficiently many $\omega$-white intervals
that contain a noticeable portion of $E^{{\rm c}}$. We define the
function $\delta(\omega)=m\left((E_{\omega}+n\tau)\setminus E_{\omega}\right)$
and notice that 
\[
\intsy_{\Omega}\delta(\omega)\,{\rm d}\mathcal{P}(\omega)=\Delta\,.
\]
Let $L=\bigl\{\omega\in\Omega\colon\delta(\omega)>\tfrac{1}{2}\Delta\bigr\}$.
It is clear that 
\[
\intsy_{L}\delta(\omega)\,{\rm d}\mathcal{P}(\omega)\ge\frac{\Delta}{2}\,.
\]
For $\omega\in L$ we have that $m(E_{\omega}^{{\rm c}}),m(E_{\omega})\ge\delta(\omega)>\tfrac{\Delta}{2}=\gamma$,
and therefore $L\times\bbt\subset W$. Thus $\left(L\times\bbt\right)\cap E^{{\rm c}}\subset W\cap E^{{\rm c}}$
and 
\[
m(W\cap E^{{\rm c}})\ge m(\left(L\times\bbt\right)\cap E^{{\rm c}})=\intsy_{L}m(E_{\omega}^{{\rm c}})\,{\rm d}\mathcal{P}(\omega)\ge\intsy_{L}\delta(\omega)\,{\rm d}\mathcal{P}(\omega)\ge\frac{\Delta}{2}\,,
\]
proving the lemma. \hfill{}$\Box$

\section{Logarithmic Integrability - $\Exploc$ property for Rademacher Fourier
series \label{sec:Rad_Four_in_Exp_loc}}

In this section we explain how to effectively show that functions
from the subspace $\RFfuncs$ have the $\Exploc$ property. We finish
the proof by analyzing the iterative process.

%% proofs %%

\subsection{Zygmund's premise and the operator $A_{E}$ \label{subsect:Zygmund-idea}}

Suppose that 
\[
f=\sum_{k\in\bbz}a_{k}\phi_{k}\,,\qquad\phi_{k}(\omega,\theta)=\xi_{k}(\omega)e(k\theta)\,,\quad\{a_{k}\}\in\ell^{2}(\bbz)\,,
\]
and that $b\in L^{2}(\Omega)$. Let $E\subset Q$ be a measurable
set of positive measure. Then 
\begin{align*}
\intsy_{E}|f-b|^{2}\,{\rm d}\mu & =\intsy_{E}\Bigl[\sum_{k,\ell}a_{k}\bar{a}_{\ell}\,\phi_{k}\bar{\phi}_{\ell}-2\,\Re(f\bar{b})+|b|^{2}\Bigr]\,{\rm d}\mu\\
 & \ge\intsy_{E}\Bigl[\sum_{k}|a_{k}|^{2}|\phi_{k}|^{2}\Bigr]\,{\rm d}\mu+\intsy_{E}\Bigl[\sum_{k\ne\ell}a_{k}\bar{a}_{\ell}\,\phi_{k}\bar{\phi}_{\ell}\Bigr]\,{\rm d}\mu-2\,\Re\langle f,\indf_{E}b\rangle\\
 & =\mu(E)\|f\|_{2}^{2}+\langle A_{E}f,f\rangle-2\,\Re\langle f,\indf_{E}b\rangle,
\end{align*}
where $A_{E}$ is a bounded self-adjoint operator on $\RFfuncs$,
whose matrix $\left(A_{E}(k,\ell)\right)_{k,\ell\in\bbz}$ in the
orthonormal basis $\left\{ \phi_{k}\right\} $ is given by 
\[
A_{E}(k,\ell)=\begin{cases}
\langle\indf_{E},\phi_{k}\bar{\phi}_{\ell}\rangle\,, & k\ne\ell;\\
0\,, & k=\ell\,.
\end{cases}
\]
To estimate the Hilbert-Schmidt norm of $A_{E}$, we observe that
the functions $\left\{ \phi_{k}\bar{\phi}_{\ell}\right\} _{k\ne\ell}$
form an orthonormal system in $L^{2}(Q)$, and that each function
from this system is orthogonal to the function $\indf$. Then 
\[
\sum_{k\ne\ell}|A_{E}(k,\ell)|^{2}+\bigl|\langle\indf_{E},\indf\rangle\bigr|^{2}\le\|\indf_{E}\|_{2}^{2}=\mu(E)\,,
\]
and therefore, 
\[
\|A_{E}\|_{\mathrm{HS}}=\sqrt{\sum_{k\ne\ell}|A_{E}(k,\ell)|^{2}}\le\sqrt{\mu(E)-\mu(E)^{2}}\,.
\]
This estimate is useful for sets $E$ of large measure.

\subsection{The sets $E$ of large measure \label{subsect:large-meas}}

For each $\mu\in(0,1)$, let $D(\mu)\in(1,+\infty]$ be the smallest
value such that the inequality 
\[
\intsy_{Q}|f|^{2}\,{\rm d}\mu\le D(\mu)\intsy_{E}|f-b|^{2}\,{\rm d}\mu
\]
holds for every $E\subset Q$ with $\mu(E)\ge\mu$, every $f\in\RFfuncs$,
and every random constant $b\in L^{\infty}(\Omega)$ with $\|b\|_{\infty}<\tfrac{1}{20}\|f\|_{2}$.

Using the estimates from the previous section, we get 
\begin{align*}
\intsy_{E}|f-b|^{2}\,{\rm d}\mu & \ge\bigl(\mu(E)-\|A\|\bigr)\|f\|_{2}^{2}-2\,\|\indf_{E}b\|_{2}\,\|f\|_{2}\\
 & \ge\bigl(\mu(E)-\sqrt{\mu(E)-\mu(E)^{2}}-\tfrac{1}{10}\bigr)\|f\|_{2}^{2}\ge\tfrac{1}{2}\,\|f\|_{2}^{2}\,,
\end{align*}
provided that $\mu(E)\ge\tfrac{9}{10}$. That is, $D(\mu)\le2$ for
$\mu\ge\tfrac{9}{10}$.

\medskip{}
In order to get an upper bound for $D(\mu)$ for smaller values of
$\mu$, we first of all need to get a better bound for the Hilbert-Schmidt
norm of the operator $A_{E}$.

\subsection{A better bound for the Hilbert-Schmidt norm of $A_{E}$ \label{subsect:better-bound-HSnorm}}

Here, using the bilinear Khinchin inequalities~\ref{subsubsect:bilinear-Khinchin},
we show that \textit{for each $p\ge1$,} 
\[
\|A_{E}\|_{\mathrm{HS}}\le Cp\cdot\mu(E)^{1-\frac{1}{2p}}\,.
\]
For sets $E$ of small measure, this bound is better than the one
we gave in~\ref{subsect:Zygmund-idea}.

\medskip{}

\begin{proof}
First, using duality and then Hölder's inequality, we get 
\begin{align*}
\sqrt{\sum_{k\ne\ell}|A_{E}(k,\ell)|^{2}} & =\sup\left\{ \Bigl|\sum_{k\ne\ell}A_{E}(k,\ell)g_{k,\ell}\Bigr|\colon\sum_{k\ne\ell}|g_{k,\ell}|^{2}\le1\right\} \\
 & =\sup\left\{ \Bigl|\intsy_{Q}\indf_{E}\bar{g}\,{\rm d}\mu\Bigr|\colon g\in\opspan\left\{ \phi_{k}\bar{\phi}_{\ell}\right\} _{k\ne\ell}\,,\ \|g\|_{2}\le1\right\} \\
 & \le\mu(E)^{1-\frac{1}{2p}}\cdot\sup\left\{ \|g\|_{2p}\colon g\in\opspan\left\{ \phi_{k}\bar{\phi}_{\ell}\right\} _{k\ne\ell}\,,\ \|g\|_{2}\le1\right\} \,.
\end{align*}
Now, using the bilinear Khinchin inequality, we will bound $\|g\|_{2p}$
by $Cp\|g\|_{2}$. Since $g\in\opspan\left\{ \phi_{k}\bar{\phi}_{\ell}\right\} _{k\ne\ell}$,
\[
g(\omega,\theta)=\sum_{k\ne\ell}g_{k,\ell}\xi_{k}(\omega)\xi_{\ell}(\omega)e((k-l)\theta)\,,
\]
whence, 
\begin{align*}
\intsy_{Q}|g|^{2p}\,{\rm d}\mu & =\intsy_{\bbt}{\rm d}m(\theta)\,\intsy_{\Omega}{\rm d}\Pr(\omega)\left|\sum_{k\ne\ell}g_{k,\ell}\xi_{k}(\omega)\xi_{\ell}(\omega)e((k-l)\theta)\right|^{2p}\\
 & \le\intsy_{\bbt}{\rm d}m(\theta)\,(Cp)^{2p}\left(\sum_{k\ne\ell}\bigl|g_{k,\ell}e((k-\ell)\theta)\bigr|^{2}\right)^{p}\\
 & =(Cp)^{2p}\left(\sum_{k\ne\ell}|g_{k,\ell}|^{2}\right)^{p}=(Cp)^{2p}\|g\|_{2}^{2p}\,,
\end{align*}
completing the proof.
\end{proof}

\subsection{The subspace $V_{E,b}$ \label{subsect:V_E}}

\noindent Let $p\ge1$. We now show that there exists a positive numerical
constant $C'$ with the following property. \textit{If $E\subset Q$
is a set of positive measure and $b\in L^{2}(Q)$, then there exists
a subspace $V_{E,b}\subset\RFfuncs$ of dimension at most 
\[
n=\Bigl[\,\frac{C'p^{2}}{\mu(E)^{1/p}}\,\Bigr]
\]
such that for each function $g\in\RFfuncs\ominus V_{E,b}$ and each
$b_{1}=c\cdot b$ with $c\in\bbc$, we have} 
\[
\intsy_{Q}|g|^{2}\,{\rm d}\mu\le\frac{2}{\mu(E)}\,\intsy_{E}|g-b_{1}|^{2}\,{\rm d}\mu\,.
\]

\medskip{}

\begin{proof}
This result is a rather straightforward consequence of the estimates
from Subsections~\ref{subsect:Zygmund-idea} and~\ref{subsect:better-bound-HSnorm}.
We enumerate the eigenvalues of the operator $A_{E}$ so that their
absolute values form a non-increasing sequence: $|\sigma_{1}|\ge|\sigma_{2}|\ge\cdots$.
Let $h_{1},h_{2},\ldots$ be the corresponding eigenvectors. Let $m\in\bbz$
and denote by $\widetilde{V}_{E}$ the linear span of $h_{1},h_{2},\ldots,h_{m}$.
Then the norm of the restriction $A_{E}$ to $L_{{\tt RF}}^{2}\ominus\widetilde{V}_{E}$
equals $|\sigma_{m+1}|$. Therefore, if the function $g\in\RFfuncs\ominus\widetilde{V}_{E}$,
then $\bigl|\langle A_{E}g,g\rangle\bigr|\le|\sigma_{m+1}|\cdot\|g\|_{2}^{2}$.

\noindent Next, 
\begin{align*}
\sigma_{m+1}^{2} & \le\frac{1}{m+1}\sum_{j=1}^{m+1}\sigma_{j}^{2}\le\frac{1}{m+1}\sum_{j=1}^{\infty}\sigma_{j}^{2}\\
 & =\frac{1}{m+1}\,\|A_{E}\|_{\mathrm{HS}}^{2}\le\frac{Cp^{2}}{m+1}\cdot\mu(E)^{2-\frac{1}{p}}<\frac{1}{4}\,\mu(E)^{2}\,,
\end{align*}
provided that 
\[
m\ge\Bigl[\,\frac{C^{\prime}p^{2}}{\mu(E)^{1/p}}\,\Bigr]-1
\]
and $C^{\prime}$ is chosen large enough.

Denote by $U_{E,b}$ the one-dimensional space spanned by the projection
of the function $\indf_{E}\cdot b$ to $\RFfuncs$, and put $V_{E,b}=\widetilde{V}_{E}+U_{E,b}$.
Then, assuming that $g\in\RFfuncs\ominus V_{E,b}\subset\RFfuncs\ominus\widetilde{V}_{E}$
and applying the estimate from~\ref{subsect:Zygmund-idea}, we get
\begin{align*}
\intsy_{E}|g-b_{1}|^{2}\,{\rm d}\mu & \ge\mu(E)\|g\|_{2}^{2}+\langle A_{E}g,g\rangle-2\,\Re{\langle g,\indf_{E}b_{1}\rangle}\\
 & \ge\mu(E)\|g\|_{2}^{2}-\frac{1}{2}\mu(E)\|g\|_{2}^{2}=\frac{\mu(E)}{2}\,\|g\|_{2}^{2}.
\end{align*}
Since $\dim V_{E,b}$ is at most $\Bigl[C^{\prime}p^{2}\mu(E)^{-1/p}\Bigr]$,
the proof is complete.
\end{proof}
\medskip{}
Note that it suffices to take $C'=4C^{2}+1$, where $C$ is the constant
that appears in the bilinear Khinchin inequality~\ref{subsubsect:bilinear-Khinchin},
though this is not essential for our purposes.

\subsection{Functions $f\in\RFfuncs$ have the $\Exploc$ property. Condition
$(C_{\tau})$ \label{subsect:putting-f-into-Exp_loc}}

Introduce the function\index{Exploc
 property@$\Exploc$ property} 
\[
n(p,\mu)\eqdef\bigl[\, C''p^{2}\cdot\mu^{-\frac{1}{p}}\,\bigr]
\]
where $C''>C'$ is a sufficiently large numerical constant. Fix $p\ge1$
and let $E\subset Q$ be a given set of positive measure. Put $n=n\bigl(p,\tfrac{1}{2}\mu(E)\bigr)$
and choose the small parameter $\tau$ so that, for every $t\in(0,\tau]$,
\[
\mu\left(\bigcap_{k=0}^{n}(E-kt)\right)\ge\frac{1}{2}\mu(E)\,.\eqno(C_{\tau})
\]
This is possible since the function $t\mapsto\mu\left((E-t)\bigcap E\right)$
is continuous and equals $\mu(E)$ at $0$.

\medskip{}
Now we prove that \textit{given a set $E\subset Q$ of positive measure,
$b\in L^{2}(Q)$, and $p\ge1$, each function $f\in\RFfuncs$ has
the $\Exploc(n,\tau,\varkappa,L^{2}(\Omega))$ property with 
\[
n=n(p,\tfrac{1}{2}\mu(E))\,,\qquad\varkappa^{2}=\frac{4(n+1)}{\mu(E)}\,\intsy_{E}|f-b|^{2}\,{\rm d}\mu\,,
\]
and arbitrary $\tau$ satisfying condition $(C_{\tau})$.}

\medskip{}

\begin{proof}
To shorten the notation, we put 
\[
E'=E'_{t}=\bigcap_{k=0}^{n}(E-kt)\,.
\]
Then for every $k\in\{0,\ldots,n\}\,,$ 
\[
\intsy_{E'}|f_{kt}-b|^{2}\,{\rm d}\mu\le\intsy_{E-kt}|f_{kt}-b|^{2}\,{\rm d}\mu=\intsy_{E}|f-b|^{2}\,{\rm d}\mu,
\]
since $b$ depends only on $\omega$, and so, $b_{kt}=b$.

\noindent Given $t\in(0,\tau]$, we choose $a_{0},\ldots,a_{n}\in\bbc$
with $\sum_{k=0}^{n}|a_{k}|^{2}=1$ so that the function $g=\sum_{k=0}^{n}a_{k}f_{kt}$
belongs to the linear space $L_{{\tt RF}}^{2}\ominus V_{E',\, b}$.
This is possible since 
\[
\dim V_{E',\, b}\le n(p,\mu(E'))\le n(p,\tfrac{1}{2}\mu(E))=n\,.
\]
Since the function $g$ is orthogonal to the subspace $V_{E',\, b}$,
we can control its norm applying the estimate from~\ref{subsect:V_E}
with $b_{1}=b\cdot\sum_{k}a_{k}$: 
\begin{align*}
\intsy_{Q}|g|^{2}\,{\rm d}\mu & \le\frac{2}{\mu(E')}\intsy_{E'}|g-b_{1}|^{2}\,{\rm d}\mu\\
 & \le\frac{4}{\mu(E)}\intsy_{E'}\Bigl|\sum_{k=0}^{n}a_{k}\bigl(f_{kt}-b\bigr)\Bigr|^{2}\,{\rm d}\mu\\
 & \le\frac{4}{\mu(E)}\intsy_{E'}\sum_{k=0}^{n}\bigl|f_{kt}-b\bigr|^{2}\,{\rm d}\mu\le\frac{4(n+1)}{\mu(E)}\,\intsy_{E}|f-b|^{2}\,{\rm d}\mu\,.
\end{align*}
That is, 
\[
\Bigl\|\sum_{k=0}^{n}a_{k}f_{kt}\Bigr\|_{2}\le\varkappa\,,
\]
and we are done.
\end{proof}

\subsection{Spreading the $L^{2}$-bound. Condition $(C_{E})$ \label{subsect:spreading-C_E}}

We apply the spreading Lemma~\ref{lemma:spreading} to the function
$f$ and the set $E$. It provides us with a set $\widetilde{E}\supset E$,
such that $\mu(\widetilde{E})\ge\mu(E)+\tfrac{1}{2}\Delta_{n\tau}(E)$
and 
\begin{align*}
\intsy_{\widetilde{E}}|f-b|^{2}\,{\rm d}\mu & \le\Bigl(\frac{Cn^{3}}{\Delta_{n\tau}^{2}(E)}\Bigr)^{2n+1}\,\left(\intsy_{E}|f-b|^{2}\,{\rm d}\mu+\varkappa^{2}\right)\\
 & \le\Bigl(\frac{Cn^{3}}{\Delta_{n\tau}^{2}(E)}\Bigr)^{2n+1}\cdot\frac{C(n+1)}{\mu(E)}\,\intsy_{E}|f-b|^{2}\,{\rm d}\mu\,,
\end{align*}
where $n=n\bigl(p,\tfrac{1}{2}\mu(E)\bigr)\le2C''p^{2}\cdot\mu(E)^{-\frac{1}{p}}$.
There is not much value in this spreading until we learn how to control
the parameter $\Delta_{n\tau}(E)$ in terms of our main parameters
$\mu(E)$ and $p$. Clearly, the bigger $\Delta_{n\tau}(E)$, the
better our spreading estimate is. Recall that till this moment, our
only assumption on the value of $\tau$ has been condition $(C_{\tau})$
at the beginning of Section~\ref{subsect:putting-f-into-Exp_loc}.

Now we will need the following condition on our set $E$: 
\[
\max_{t}\Delta_{t}(E)\ge\frac{1}{2n}\mu(E).\eqno(C_{E})
\]
If condition $(C_{E})$ holds, then we can find $\tau>0$ such that
$\Delta_{n\tau}(E)=\tfrac{1}{2n}\,\mu(E)$, while for all $t\in(0,n\tau)$,
$\Delta_{t}(E)<\tfrac{1}{2n}\,\mu(E)$.

\medskip{}
Such $\tau$ will automatically satisfy condition $(C_{\tau})$ used
in the derivation of the spreading estimate. Indeed, 
\begin{align*}
\mu\Bigl(\bigcap_{k=0}^{n}(E-kt)\Bigr) & =\mu\Bigl(E\setminus\bigcup_{k=1}^{n}\bigl(E\setminus(E-kt)\bigr)\Bigr)\\
 & \ge\mu(E)-\sum_{k=1}^{n}\mu\bigl(E\setminus(E-kt)\bigr)\\
 & =\mu(E)-\sum_{k=1}^{n}\mu\bigl((E+kt)\setminus E\bigr)\\
 & =\mu(E)-\sum_{k=1}^{n}\Delta_{kt}(E)\ge\mu(E)-\sum_{k=1}^{n}\frac{\mu(E)}{2n}=\frac{1}{2}\mu(E).
\end{align*}

It is easy to see that there are sets $E\subset Q$ of arbitrary small
positive measure that do not satisfy condition $(C_{E})$. We assume
now that condition $(C_{E})$ is satisfied, putting aside the question
``\textit{What to do with the sets $E$ for which $(C_{E})$ does
not hold}?'' till the next section.

Substituting the value $\Delta_{n\tau}=\tfrac{1}{2n}\,\mu(E)$ into
the spreading estimate and taking into account that $n\le2C''p^{2}\,\mu(E)^{-\frac{1}{p}}$,
we finally get 
\begin{align*}
\intsy_{\widetilde{E}}|f-b|^{2}\,{\rm d}\mu & \le\Bigl(\frac{Cn^{5}}{\mu(E)^{2}}\Bigr)^{2n+1}\cdot\frac{n+1}{\mu(E)}\,\intsy_{E}|f-b|^{2}\,{\rm d}\mu\\
 & \le\Bigl(\frac{Cp}{\mu(E)}\Bigr)^{Cp^{2}\mu(E)^{-1/p}}\,\intsy_{E}|f-b|^{2}\,{\rm d}\mu\,,
\end{align*}
while 
\[
\mu(\widetilde{E})\ge\mu(E)+\frac{c}{p^{2}}\,\mu(E)^{1+\frac{1}{p}}\,.
\]
This is the spreading estimate that we will use for the sets $E$
satisfying condition $(C_{E})$.

\subsection{The case of sets $E$ that do not satisfy condition $(C_{E})$ \label{subsect:sets-without-C_E}}

Now, let us assume that $E\subset Q$ is a set of positive measure
that does not satisfy condition $(C_{E})$, that is, for each $t\in[0,1]$,
$\Delta_{t}(E)<\tfrac{1}{2n}\,\mu(E)$. The simplest example is any
set of the form $E=\Omega_{1}\times\bbt$, $\Omega_{1}\subset\Omega$.
For these sets, $\Delta_{t}(E)=0$ for every $t$. We will show that
this example is typical, i.e., the sets $E$ that do not satisfy condition
$(C_{E})$ must have sufficiently many `long' sections $E_{\omega}$.
More precisely, let 
\[
\Omega_{1}=\bigl\{\omega\in\Omega\colon m(E_{\omega})>1-\tfrac{1}{n}\bigr\}\,.
\]
We show that $\Pr\{\Omega_{1}\}>\tfrac{1}{2}\,\mu(E)$.

\medskip{}

\begin{proof}
Let 
\[
\Omega_{2}=\Omega\setminus\Omega_{1}=\bigl\{\omega\in\Omega\colon m(E_{\omega})\le1-\tfrac{1}{n}\bigr\}.
\]
Since condition $(C_{E})$ is not satisfied, we have 
\[
\intsy_{0}^{1}\Delta_{t}(E)\,{\rm d}t<\frac{1}{2n}\,\mu(E)\,.
\]
A straightforward computation shows that 
\[
\intsy_{0}^{1}m\left(\left(E_{\omega}+t\right)\setminus E_{\omega}\right)\,{\rm d}t=m(E_{\omega})\left(1-m(E_{\omega})\right).
\]
Since $m(E_{\omega})\le1-\tfrac{1}{n}$ implies that $m(E_{\omega})\le nm(E_{\omega})(1-m(E_{\omega}))$,
we get 
\begin{align*}
\intsy_{\Omega_{2}}m(E_{\omega})\,{\rm d}\Pr(\omega) & \le n\intsy_{\Omega_{2}}m(E_{\omega})\left(1-m(E_{\omega})\right)\,{\rm d}\Pr(\omega)\\
 & \le n\intsy_{\Omega}m(E_{\omega})\left(1-m(E_{\omega})\right)\,{\rm d}\Pr(\omega)\\
 & =n\intsy_{0}^{1}\Delta_{t}(E)\,{\rm d}t<\frac{1}{2}\mu(E).
\end{align*}
Therefore, 
\[
\pr{\Omega_{1}}\ge\intsy_{\Omega_{1}}m(E_{\omega})\,{\rm d}\Pr(\omega)=\mu(E)-\intsy_{\Omega_{2}}m(E_{\omega})\,{\rm d}\Pr(\omega)>\frac{1}{2}\mu(E).
\]
\end{proof}
\begin{rem*}
\noindent \textit{Since $n=n\left(p,\tfrac{1}{2}\mu(E)\right)\ge2$
if $C^{\prime\prime}\ge2$, we trivially have} 
\[
\pr{\Omega_{1}}\le\frac{n}{n-1}\intsy_{\Omega_{1}}m(E_{\omega})\,{\rm d}\Pr(\omega)\le\frac{n}{n-1}\,\mu(E)\le2\,\mu(E).
\]

\end{rem*}

\subsection{Many `long' sections \label{subsect:many-long-sections}}

Assume that the set $E$ does not satisfy condition $(C_{E})$. We
will show that 
\[
\intsy_{Q}|f|^{2}\,{\rm d}\mu\le\frac{4}{\mu(E)}\,\intsy_{E}|f-b|^{2}\,{\rm d}\mu\,,
\]
where, as above, $b=b(\omega)$ is a random constant, $\|b\|_{\infty}<\tfrac{1}{20}\|f\|_{2}$.

Let $\mu=\mu(E)$ and $\Omega_{1}$ be as above. We have 
\begin{align*}
\intsy_{E}|f-b|^{2}\,{\rm d}\mu & \ge\intsy_{\Omega_{1}}\left(\intsy_{E_{\omega}}\,|f-b|^{2}\,{\rm d}m\right)\,{\rm d}\Pr(\omega)\\
 & =\intsy_{\Omega_{1}}\,\intsy_{\bbt}|f-b|^{2}\,{\rm d}m\,{\rm d}\Pr(\omega)-\intsy_{\Omega_{1}}\,\intsy_{\bbt\setminus E_{\omega}}\,|f-b|^{2}\,{\rm d}m\,{\rm d}\Pr(\omega)\\
 & =({\rm I})-({\rm II})\,.
\end{align*}
Notice that by the result of Section~\ref{subsect:spreading-C_E},
we have $2\mu\ge\pr{\Omega_{1}}\ge\tfrac{1}{2}\mu$.

Bounding integral $({\rm I})$ from below is straightforward: we have
\[
\intsy_{\bbt}|f-b|^{2}\,{\rm d}m\ge\bigl(\|f\|_{2}-\|b\|_{\infty}\bigr)^{2}\ge\frac{9}{10}\,\|f\|_{2}^{2}\,,
\]
whence, 
\[
({\rm I})\ge\frac{9}{10}\,\|f\|_{2}^{2}\,\pr{\Omega_{1}}\ge\frac{9}{20}\,\cdot\mu\,\|f\|_{2}^{2}\,.
\]

Now let us estimate the integral (II) from above. We have 
\[
({\rm II})\le2\intsy_{\Omega_{1}}\,\intsy_{\bbt\setminus E_{\omega}}|f|^{2}\,{\rm d}m\,{\rm d}\Pr(\omega)+2\intsy_{\Omega_{1}}\intsy_{\bbt\setminus E_{\omega}}|b|^{2}\,{\rm d}m\,{\rm d}\Pr(\omega)=({\rm II}_{a})+({\rm II}_{b})\,.
\]
Estimating the second integral is also straightforward: 
\[
({\rm II}_{b})\le\frac{4\mu}{n}\,\|b\|_{\infty}^{2}<\frac{1}{10}\,\mu\,\|f\|_{2}^{2}
\]
(recall that $n\ge2$ and $\|b\|_{\infty}<\tfrac{1}{20}\|f\|_{2}$).
Furthermore, 
\[
({\rm II}_{a})=2\intsy_{\Omega_{1}}\,\intsy_{\bbt}\indf_{\bbt\setminus E_{\omega}}\,|f|^{2}\,{\rm d}m\,{\rm d}\Pr(\omega)\le2\Bigl(\intsy_{\Omega_{1}}\,\intsy_{\bbt}\indf_{\bbt\setminus E_{\omega}}\Bigr)^{\frac{1}{r}}\Bigl(\intsy_{\Omega_{1}}\,\intsy_{\bbt}|f|^{2s}\Bigr)^{\frac{1}{s}}
\]
with $\tfrac{1}{r}+\tfrac{1}{s}=1$. By Khinchin's inequality, 
\[
\Bigl(\intsy_{\Omega}\,\intsy_{\bbt}|f|^{2s}\Bigr)^{\frac{1}{s}}\le Cs\,\|f\|_{2}^{2}\,.
\]
Hence, 
\[
({\rm II}_{a})\le\Bigl(\frac{2\mu}{n}\Bigr)^{\frac{1}{r}}\, Cs\,\|f\|_{2}^{2}\,.
\]
Letting $\tfrac{1}{r}=\tfrac{p}{p+1},\tfrac{1}{s}=\tfrac{1}{p+1}$
and recalling that $n\ge\tfrac{1}{2}C''\, p^{2}\mu^{-1/p}$ and that
$p\ge1$, we continue the estimate as 
\[
({\rm II}_{a})\le\Bigl(\frac{4\mu^{1+\frac{1}{p}}}{C''\, p^{2}}\Bigr)^{\frac{p}{p+1}}\,2Cp\,\|f\|_{2}^{2}<\frac{8C}{\sqrt{C''}}\,\mu\, p^{-\frac{p-1}{p+1}}\,\|f\|_{2}^{2}<\frac{1}{10}\,\mu\,\|f\|_{2}^{2}\,,
\]
provided that the constant $C''$ in the definition of $n$ was chosen
sufficiently big. Finally, 
\[
\intsy_{E}|f-b|^{2}\ge({\rm I})-({\rm II}_{a})-({\rm II}_{b})\ge\Bigl(\frac{9}{20}-\frac{4}{20}\Bigr)\mu\|f\|_{2}^{2}=\frac{1}{4}\,\mu\,\|f\|_{2}^{2}\,,
\]
completing the argument. \hfill{}$\Box$

\subsection{End of the proof of Theorem~\ref{thm:Zygmund-type-final}: solving
a difference inequality \label{subsect:difference_ineq}}

Recall that by $D(\mu)$ we denote the smallest value such that the
inequality 
\[
\intsy_{Q}|f|^{2}\,{\rm d}\mu\le D(\mu)\intsy_{E}|f-b|^{2}\,{\rm d}\mu
\]
holds for every $E\subset Q$ with $\mu(E)\ge\mu$, every $f\in\RFfuncs$,
and every random constant $b\in L^{\infty}(\Omega)$ satisfying $\|b\|_{\infty}<\tfrac{1}{20}\|f\|_{2}$.

By~\ref{subsect:large-meas}, $D(\mu)\le2$ for $\mu\ge\tfrac{9}{10}$,
and by the estimates proven in~\ref{subsect:spreading-C_E} and~\ref{subsect:many-long-sections},
for $0<\mu<\tfrac{9}{10}$ we have\index{spreading procedure} 
\[
D(\mu)<\max\Bigl\{\Bigl(\frac{Cp}{\mu}\Bigr)^{Cp^{2}\mu^{-\frac{1}{p}}}D\bigl(\mu+\frac{c}{p^{2}}\,\mu^{1+\frac{1}{p}}\bigr),\frac{4}{\mu}\Bigr\}\,.
\]
Increasing, if needed, the constant $C$ in the exponent, and taking
into account that $\frac{p}{\mu}\ge\frac{1}{\nicefrac{9}{10}}>1$
and $D\ge1$, we simplify this to 
\[
D(\mu)<\Bigl(\frac{p}{\mu}\Bigr)^{Cp^{2}\mu^{-\frac{1}{p}}}D\left(\mu+\frac{c}{p^{2}}\,\mu^{1+\frac{1}{p}}\right).
\]
Put 
\[
\delta(\mu)=\frac{c}{p^{2}}\,\mu^{1+\frac{1}{p}}\,.
\]
Making the constant $c$ on the right-hand side small enough, we assume
that $\delta\bigl(\tfrac{9}{10}\bigr)<\tfrac{1}{10}$ (it suffices
to take $c<\tfrac{1}{10}$). Then, for $0<\mu<\tfrac{9}{10}$, 
\[
\log D(\mu)-\log D(\mu+\delta(\mu))<Cp^{2}\mu^{-\frac{1}{p}}\log\Bigl(\frac{p}{\mu}\Bigr)<C\,\delta(\mu)\, p^{4}\mu^{-1-\frac{2}{p}}\log\Bigl(\frac{p}{\mu}\Bigr).
\]
To solve this difference inequality, we define the sequence $\mu_{0}=\mu$,
$\mu_{k+1}=\mu_{k}+\delta(\mu_{k})$, $k\ge0$, and stop when $\mu_{s-1}<\tfrac{9}{10}\le\mu_{s}$.
Since we assumed that $\delta\bigl(\tfrac{9}{10}\bigr)<\tfrac{1}{10}$,
the terminal value $\mu_{s}$ will be strictly less than $1$. We
get 
\begin{align*}
\log D(\mu) & =\log D(\mu_{s})+\sum_{k=0}^{s-1}\bigl[\log D(\mu_{k})-\log D(\mu_{k+1})\bigr]\\
 & <1+Cp^{4}\,\sum_{k=0}^{s-1}\delta(\mu_{k})\mu_{k}^{-1-\frac{2}{p}}\log\Bigl(\frac{p}{\mu}_{k}\Bigr)\\
 & <1+Cp^{4}\,\log\Bigl(\frac{p}{\mu}\Bigr)\,\sum_{k=0}^{s-1}\delta(\mu_{k})\mu_{k}^{-1-\frac{2}{p}}\,.
\end{align*}
Since $\mu_{k+1}=\mu_{k}+c\, p^{-2}\,\mu_{k}^{1+\frac{1}{p}}<C\mu_{k}$,
we have $\mu_{k}^{-1-\frac{2}{p}}<C\mu_{k+1}^{-1-\frac{2}{p}}$. Therefore,
\begin{align*}
\sum_{k=0}^{s-1}\delta(\mu_{k})\mu_{k}^{-1-\frac{2}{p}} & <C\,\sum_{k=0}^{s-1}\delta(\mu_{k})\mu_{k+1}^{-1-\frac{2}{p}}\\
 & <C\,\sum_{k=0}^{s-1}\intsy_{\mu_{k}}^{\mu_{k+1}}\frac{{\rm d}x}{x^{1+\frac{2}{p}}}<C\,\intsy_{\mu}^{1}\frac{{\rm d}x}{x^{1+\frac{2}{p}}}<C\, p\,\mu^{-\frac{2}{p}}\,,
\end{align*}
whence 
\[
\log D(\mu)<1+C\, p^{5}\,\mu^{-\frac{2}{p}}\,\log\left(\frac{p}{\mu}\right).
\]
This holds for any $p\ge1$. Letting $p=2\log\bigl(\tfrac{2}{\mu}\bigr)$,
we finally get $\log D(\mu)<C\log^{6}\Bigl(\frac{2}{\mu}\Bigr)$.
This completes the proof of Theorem~\ref{thm:Zygmund-type-final}.
\hfill{}$\Box$

\section{On the power $6$ in the statement of Theorem \ref{thm:Zygmund-type-final}
\label{sect_log_int_examples}}

In this section we will present an example that shows that the constant
$6$ in the exponent on the RHS of the inequality proven in Theorem~\ref{thm:Zygmund-type-final}
cannot be replaced by any number smaller than $2$.

Let 
\[
g_{N}(\theta)=\left(\sin(2\pi\theta)\right)^{2N}=\left(\frac{e(\theta)-e(-\theta)}{2i}\right)^{2N}=\sum_{|n|\le N}a_{n}e(2n\theta).
\]
The function $g_{N}$ satisfies 
\begin{equation}
|g_{N}(\theta)|\le e^{-cN^{2}}\qquad\mbox{for}\,\,|\theta|\le e^{-CN},\label{eq:g_N}
\end{equation}
provided that $C$ is large enough.

Now consider the Rademacher trigonometric polynomial 
\[
f_{N}(\theta)=\sum_{|n|\le N}\xi_{n}a_{n}e(2n\theta)\,,
\]
denote by $X_{N}$ the event that $\xi_{n}=+1$ for all $n\in\bigl\{-N,...,N\bigr\}$,
and put $E_{N}=X_{N}\times T_{N}$, where $T_{N}=[-e^{CN},e^{CN}]\subset\bbt$
is the set from~\eqref{eq:g_N}. Then 
\[
\mu(E_{N})\ge2^{-(2N+1)}\cdot e^{-CN}\ge e^{-CN}\,,
\]
while 
\[
\intsy_{E_{N}}|f_{N}|^{2}\,{\rm d}\mu\le e^{-cN^{2}}\mu(E_{N})\le e^{-cN^{2}}
\]
and 
\[
\intsy_{\Omega\times\bbt}|f_{N}|^{2}\,{\rm d}\mu=\intsy_{\bbt}|g_{N}|^{2}\,{\rm d}m\,.
\]
It is not difficult to see that the integral on the RHS is not less
than $\tfrac{c}{N}$, for some constant $c>0$. Recalling that $|\log\mu(E_{N})|\le CN$,
we see that for every $\epsym>0$, $C>0$, the inequality 
\[
\intsy_{Q}|f_{N}|^{2}\,{\rm d}\mu\le e^{C|\log\mu(E_{N})|^{2-\epsym}}\intsy_{E_{N}}|f_{N}|^{2}\,{\rm d}\mu
\]
fails when $N\ge N_{0}(\epsym,C)$. This shows that one cannot replace
$6$ by any number less than $2$. \hfill{}$\Box$

\section{Some open problems\index{open problems}}

Following the last section, it is not clear what is the optimal constant
in the statement of Theorem \ref{thm:Zygmund-type-final}. Proving
that the constant is strictly larger than $2$ seems to be an interesting
problem (in case it is true).

It should be mentioned that the result of the previous section (the
optimal constant in the power of the logarithm) does not contradict
the possibility that the distributional inequality can be improved
if one is ready to discard an event of small probability. Here is
a sample 
\begin{problem}
Consider Rademacher trigonometric polynomials $f_{n}$ of large degree
$n$. Given a small $\delta>0$, does there exist an event $\Omega^{\prime}(n,\delta)\subset\Omega$
with $\pr{\Omega\backslash\Omega^{\prime}}\to0$ as $n\to\infty$,
such that
\[
\intsy_{\Omega^{\prime}\times\bbt}\left|f_{n}\right|^{-\delta}\d{\mu}<\infty\,?
\]

\end{problem}
It would be interesting to find a simple probabilistic proof of an
inequality of the form
\[
\normp[f][2]2\le C\left(\mu\left(E\right)\right)\intsy_{E}\left|f\right|^{2}\d{\mu},
\]
even without a quantitative dependence on the measure of the set $E$.

A (related) very natural problem is to prove a generalization of Theorem
\ref{thm:Zygmund-type-final} for other types of random variables.
It is possible that one should modify the statement of theorem to
consider the event where the sum $\sum_{n\in\bbz}\left|\xi_{n}\right|^{2}\left|a_{n}\right|^{2}$
is small.

A known open problem for non-random, `half-lacunary' Fourier series
is the following 
\begin{problem}
Consider the random Fourier series of the form
\[
g\left(\theta\right)=\sum_{k>0}a_{n_{k}}e^{2\pi in_{k}\theta}+\sum_{n\ge0}a_{n}e^{2\pi in\theta},
\]
where $\left\{ a_{n}\right\} \in\ell^{2}\left(\bbz\right)$, and $\left\{ n_{k}\right\} \subset\bbz^{-}$
is lacunary in the sense of Hadamard (i.e., $\liminf_{k\to\infty}\frac{n_{k+1}}{n_{k}}>1$).
Is it true that $g\left(\theta\right)=0$ only on a set of measure
$0$? Is it true that $\log\left|g\right|\in L^{p}\left(\bbt\right)$,
for some $p>0$?
\end{problem}
The same question can be asked for `half-Rademacher' Fourier series
of the form
\[
g\left(\theta\right)=\sum_{n<0}\xi_{n}a_{n}e^{2\pi in\theta}+\sum_{n\ge0}a_{n}e^{2\pi in\theta},
\]
where, as usual, $\xi_{n}$ is a sequence of independent Bernoulli
random variables, which take the values $\pm1$, with probability
$\tfrac{1}{2}$ each.
\begin{problem}
Let $g$ be a `half-Rademacher' Fourier series. Is it true that $\log\left|g\right|\in L^{p}\left(Q\right)$
for some $p>0$?

\newpage{}

\thispagestyle{empty}$\,$\newpage{}
\end{problem}

\part{The value distribution of random analytic functions}

The theory that is known today as the value distribution theory of
holomorphic (or meromorphic) functions, was originally developed by
Nevanlinna in the 1920s (see \cite{GO}). Compared to the more classical
and qualitative Picard Theorem, it tries to get a more quantitative
measure of the number of times an analytic function $f\left(z\right)$
assumes some value $b$, as $z$ grows. This theory can also be generalized
to several complex variables. One can think about the value distribution
of random analytic functions as a representation of the behavior of
`standard' analytic functions. In particular, it is possible to give
very precise versions of the general results of the theory.

One of the basic notions of the general theory is that of an exceptional
value. Little Picard Theorem is a classical example, as it states
that every entire function must take every complex value, except maybe
one exceptional value. In general we expect that random functions
will be ``unable to distinguish'' between different values, and
thus that they will not have any exceptional values (one must define
the correct model of course, as $e^{f\left(z\right)}\ne0$ always).
In the first section of this part we prove a result of this kind,
which is an answer to an old question of J.-P. Kahane.

A central part of the theory of entire functions was developed for
the goal of counting the number of zeros of entire functions, in large
domains (motivated, in no small part, by Riemann's Hypothesis). In
the second section we give a short proof of a strong law of large
numbers for the number of zeros of Rademacher entire functions in
large disks around $0$, as the radius of the disk tends to infinity.
This is the analog of the simpler result for Gaussian analytic functions.

In the third section, we use more accurate tools than the ones used
in the second section to show that random entire functions grow very
regularly. In particular, their zeros are equidistributed in sectors
centered at zero. Under certain restrictions functions with equidistributed
zeros are a special case of a certain class of entire functions ---
those of completely regular growth in the sense of Levin and Pfluger
(see \cite{Le1}). We note that our result lies beyond the scope of
the (classical) theory of functions of completely regular growth.

In the last two sections we give an example which demonstrates the
tightness of the results of the third section, as well as some further
work and related open problems.

\section{Kahane's Problem - The Range of Random Taylor Series in the Unit
Disk \label{sect_Kahane}\index{Kahane's question|(}}

%% Previous results

We begin with a short description of some of the previous results.
In 1972, Offord~\cite{Of3} proved the analog of Theorem \ref{thm:Kahane_prob}
in the case where $\zeta_{k}=u_{k}a_{k}$, and $u_{k}$ are uniformly
distributed on the unit circle. The proof he gave also works for Gaussian
Taylor series; see also Kahane \cite[Section~12.3]{Ka2} for a different
proof. In the special case $\zeta_{k}=\xi_{k}a_{k}$, where $\xi_{k}$
are independent Rademacher random variables, the result was known
under some additional restrictions on the growth of the deterministic
coefficients $a_{k}$. In 1981, Murai \cite{Mu2} proved it assuming
that $\liminf|a_{k}|>0$. Soon afterwards, Jacob and Offord \cite{JO}
weakened this assumption to 
\[
\liminf_{N\to\infty}\frac{1}{\log N}\sum_{k=0}^{N}|a_{k}|^{2}>0\,.
\]
To the best of our knowledge, since then there was no further improvement.

\subsection{Solution to Kahane's problem - Proof of Theorem \ref{thm:Kahane_prob}}

We start by proving the theorem in the special case when $\zeta_{n}=\xi_{n}a_{n}$,
where $\xi_{n}$ are independent Rademacher random variables, and
$\left\{ a_{n}\right\} $ is a non-random sequence of complex numbers
satisfying the conditions $\limsup_{n}|a_{n}|^{1/n}=1$ and $\sum_{n}|a_{n}|^{2}=\infty$.
The proof is based on the logarithmic integrability of the Rademacher
Fourier series (corollary (\ref{eq:log_int_Rad_Four_gen_case}) of
Theorem~\ref{thm:Zygmund-type-final}) combined with Jensen's formula.
Then using `the principle of reduction' as stated in Kahane's book~\cite[Section~1.7]{Ka2},
we get the result in the general case.

\medskip{}
Let us introduce some notation. For $b\in\bbc$, $0<r<1$, we denote
by $n_{F}(r,b)$ \index{1nF@$n_{F}$!$n_{F}(r,b)$} the number of solutions
to the equation $F(z)=b$ in the disk $r\bbd$, the solutions being
counted according to their multiplicities. In this section it will
be convenient to set\index{1NF@$N_{F}$!$N_{F}(r,b)$} 
\[
N_{F}(r,b)\eqdef\intsy_{1/2}^{r}\frac{n_{F}(t,b)}{t}\,{\rm d}t\,.
\]
By Jensen's formula, 
\begin{equation}
N_{F}(r,b)=\intsy_{\bbt}\log|F(re(\theta))-b|\,{\rm d}m(\theta)-\intsy_{\bbt}\log|F(\tfrac{1}{2}e(\theta))-b|\,{\rm d}m(\theta)\,.\label{eq:Jensen}
\end{equation}
We will prove that a.s. we have 
\[
\lim_{r\to1}N_{F}(r,b)=\infty\,,\enskip\forall b\in\bbc\,,
\]
which is equivalent to Theorem \ref{thm:Kahane_prob}.

%% Proof

\subsubsection{Proof of Theorem~\ref{thm:Kahane_prob} in the Rademacher case}

We define the functions $\sigma_{F}$ and $\widehat{F}$ by\index{1sigmaF@$\sigma_{F}$}\index{1Fhat@$\widehat{F}$}
\[
\sigma_{F}^{2}(r)\eqdef\sum_{n\ge0}|a_{n}|^{2}r^{2n},\qquad\widehat{F}(z)\eqdef\frac{F(z)}{\sigma_{F}(|z|)}\,,
\]
and note that $\|\widehat{F}(re(\theta))\|_{L^{2}(\bbt)}=1$.

Let $M\in\bbn$. For every $r\in(\tfrac{1}{2},1)$, the function $(\omega,b)\mapsto N_{F}(r,b)$
on $\Omega\times\bbc$ is measurable in $\omega$ for fixed $b$ and
continuous in $b$ for fixed $\omega$. Therefore, we can find a measurable
function $b^{*}=b^{*}(\omega)$ such that $\left|b^{*}\right|\le M$
and 
\[
\inf_{|b|\le M}N_{F}(r,b)\ge N_{F}(r,b^{*})-1.
\]
Then 
\begin{multline*}
\inf_{|b|\le M}N_{F}(r,b)\ge\intsy_{\bbt}\log|F(re(\theta))-b^{*}|\,{\rm d}m(\theta)-\intsy_{\bbt}\log|F(\tfrac{1}{2}e(\theta))-b^{*}|\,{\rm d}m(\theta)-1\\
=({\rm I}_{1})-({\rm I}_{2})-1\,.
\end{multline*}
Note that 
\[
({\rm I}_{2})\le\frac{1}{2}\log\left(\intsy_{\bbt}|F(\tfrac{1}{2}e(\theta))-b^{*}|^{2}\,{\rm d}m(\theta)\right)\le\frac{1}{2}\log\left(2\sigma_{F}^{2}(\tfrac{1}{2})+2M^{2}\right).
\]
For the integral $({\rm I}_{1})$, we have the following lower bound:
\begin{multline*}
({\rm I}_{1})=\log\sigma_{F}(r)+\intsy_{\bbt}\log\bigl|\widehat{F}(re(\theta))-\sigma_{F}^{-1}(r)\cdot b^{*}\bigr|\,{\rm d}m(\theta)\\
\ge\log\sigma_{F}(r)-\intsy_{\bbt}\Bigl|\log\bigl|\widehat{F}(re(\theta))-\sigma_{F}^{-1}(r)\cdot b^{*}\bigr|\,\Bigr|\,{\rm d}m(\theta)\,.
\end{multline*}
If we assume that $r$ is so close to $1$ that $\sigma_{F}(r)\ge20M$,
then, using our result \eqref{eq:log_int_Rad_Four_gen_case} on the
logarithmic integrability of the Rademacher Fourier series, we get
\begin{multline*}
\Pr\Bigl\{\intsy_{\bbt}\Bigl|\log\bigl|\widehat{F}(re(\theta))-\sigma_{F}^{-1}(r)\cdot b^{*}\bigr|\,\Bigr|\,{\rm d}m(\theta)>T\Bigr\}\\
\le\frac{1}{T}\,\Ex\Bigl(\intsy_{\bbt}\Bigl|\log\bigl|\widehat{F}(re(\theta))-\sigma_{F}^{-1}(r)\cdot b^{*}\bigr|\,\Bigr|\,{\rm d}m(\theta)\Bigr)\le\frac{C}{T}\,,
\end{multline*}
for all $T>0$.

Taking $r=r_{m}$ so that $\log\sigma_{F}(r_{m})=2m^{2}$ and $T=m^{2}$,
and applying the Borel-Cantelli lemma, we see that, for a.e. $\omega\in\Omega$,
there exists $m_{0}=m_{0}(\omega,M)$ such that, for each $m\ge m_{0}$,
\[
\intsy_{\bbt}\Bigl|\log\bigl|\widehat{F}(r_{m}e(\theta))-\sigma_{F}^{-1}(r_{m})\cdot b^{*}\bigr|\,\Bigr|\,{\rm d}m(\theta)<m^{2}\,,
\]
whence, 
\[
\inf_{|b|\le M}N_{F}(r_{m},b)\ge m^{2}-\frac{1}{2}\log\left(2\sigma_{F}^{2}(\tfrac{1}{2})+2M^{2}\right)-1,\qquad\forall m\ge m_{0}\,.
\]
Therefore, for every $M\in\bbn$, there is a set $A_{M}\subset\Omega$
with $\Pr(A_{M})=1$ such that, for every $\omega\in A_{M}$ and every
$b\in\bbc$ with $|b|\le M$, we have 
\begin{equation}
\lim_{r\to1}N_{F}(r,b)=\infty\,.\label{eq:N(r,b)}
\end{equation}
Let $A=\bigcap_{M}A_{M}$. Then $\Pr(A)=1$, and for every $\omega\in A$,
$b\in\bbc$, we have~\eqref{eq:N(r,b)}. Thus, the theorem is proved
in the Rademacher case.

\subsubsection{Proof of Theorem~\ref{thm:Kahane_prob} in the general case}

For every $M\in\bbn$, consider the event 
\[
B_{M}=\Bigl\{\omega\colon\lim_{r\to1}\,\inf_{|b|\le M}\, N_{F}(r,b)=+\infty\Bigr\}.
\]
Given $r\in\bigl(\tfrac{1}{2},1\bigr)$, the function $\inf_{\left|b\right|\le M}\, N_{F}(r,b)$
is measurable in $\omega$ (note that here the infimum can be taken
over any dense countable subset of the disk $\{|b|\le M\}$). Thus,
the set $B_{M}$ is measurable and so is the set $B=\bigcap_{M}B_{M}$,
and for every $\omega\in B$, $b\in\bbc$, we have~\eqref{eq:N(r,b)}.
It remains to show that $B$ holds almost surely.

To that end, we extend the probability space to $\Omega\times\Omega'$
and introduce a sequence of independent Rademacher random variables
$\left\{ \xi_{n}(\omega')\right\} $, $\omega'\in\Omega'$, which
are also independent from the random variables $\left\{ \zeta_{n}(\omega)\right\} $,
$\omega\in\Omega$, and consider the random analytic function 
\[
G(z)=G(z;\omega,\omega')=\sum_{n\ge0}\xi_{n}(\omega')\zeta_{n}(\omega)z^{n}\,,\qquad(\omega,\omega')\in\Omega\times\Omega'\,.
\]
By the previous section, for fixed $\zeta_{n}$'s (outside a set of
probability zero in $\Omega$), the event 
\[
\left\{ \omega'\in\Omega'\colon\lim_{r\to1}\,\inf_{|b|\le M}N_{G}(r,b)=+\infty\right\} 
\]
occurs with probability $1$. Hence, by Fubini's theorem, the event
$B_{M}$ occurs a.s. and so does the event $B$. Note that due to
the symmetry of the distribution of $\zeta_{n}$'s, the random variables
$\{\xi_{n}(\omega')\zeta_{n}(\omega)\}$ are equidistributed with
$\{\zeta_{n}(\omega)\}$. This yields the theorem in the general case
of symmetric random variables. \hfill{}$\Box$ 
\index{Kahane's question|)}

\subsection{Some further problems \index{open problems}}

Curiously enough, even in the case when $\zeta_{k}=\xi_{k}a_{k}$
with the standard complex Gaussian $\xi_{k}$'s, the question when
$F(\mathbb{D})=\mathbb{C}$ almost surely is not completely settled.
In \cite{Mu1} Murai proved Paley's conjecture, which states that
if $F$ is a (non-random) Taylor series with Hadamard gaps and with
the radius of convergence $1$, then $F$ assumes every complex value
infinitely often, provided that $\sum_{k\ge0}|a_{k}|=+\infty$. Therefore,
the same holds for random Taylor series with Hadamard gaps. However,
even the case of sequences $a_{k}$ with a regular behaviour remains
open:
\begin{problem}
\label{quest:gaussian} Suppose that $\xi_{k}$ are independent standard
complex Gaussian random variables. In addition, assume that the non-random
sequence $\{a_{k}\}$ is decaying regularly and satisfies 
\begin{equation}
\sum_{k\ge0}|a_{k}|^{2}<\infty\,,\qquad\sum_{k\ge0}\Bigl|\frac{a_{k}}{\sqrt{k}}\Bigr|=\infty\,.\label{eq:Kahane_reg}
\end{equation}
Does the range of the random Taylor series $F(z)=\sum_{k\ge0}\xi_{k}a_{k}z^{k}$
fill the whole complex plane $\mathbb{C}$ a.s.?
\end{problem}
Note that convergence of the first series in (\ref{eq:Kahane_reg})
yields that, a.s., the function $F$ belongs to all Hardy spaces $H^{p}$
with $p<\infty$. Moreover, by the Paley-Zygmund theorem \cite[Chapter~5]{Ka2},
a.s. we have $e^{\lambda|\widehat{F}|^{2}}\in L^{1}(\mathbb{T})$
for every positive $\lambda$, where $\widehat{F}$ denotes the non-tangential
boundary values of $F$ on $\mathbb{T}$. On the other hand, by Fernique's
theorem~\cite[Chapter~15]{Ka2}, divergence of the second series
in (\ref{eq:Kahane_reg}) yields that, a.s., $F$ is unbounded in
$\mathbb{D}$.

\section{Almost sure convergence of the number of zeros\label{sec:as_conv_of_n(r)}}

In this section, we explain how to use Theorem \ref{thm:log_int_for_Rade_Four_series}
to find almost sure asymptotics for the number of zeros of a Rademacher
entire function $f$ in a disk $\disk r=\left\{ \left|z\right|\le r\right\} $,
as $r$ tends to infinity (this can be viewed as a strong law of large
numbers for the number of zeros). In the next section we use a more
elaborate argument to prove the equidistribution of the zeros.

\subsection{Short Background}

There is a large body of work about the real zeros of random polynomials%
\footnote{It should be mentioned, however, that we are not aware of any `asymptotic
rate of growth' result about the real zeros of polynomials with general
(say) Gaussian coefficients.%
}, see the books \cite{BS,Far}. However, there are not too many works
that are concerned with complex zeros of random analytic functions.
Hammersley \cite{Ham} gave an expression for the joint distribution
of the complex (and real) zeros of random polynomials with general
Gaussian coefficients. However, it is not easy to deduce the asymptotic
behavior of the zeros from his results. Shepp and Vanderbei \cite{SV}
studied the very special case of $\sum_{k=0}^{n}\xi_{k}z^{k}$, with
$\xi_{n}$ i.i.d. standard real Gaussians. They extended the results
of Kac and Rice about the distribution of the real zeros to the distribution
of the complex ones. Using the special structure of these polynomials
they found asymptotic expressions for the distribution of the zeros.
Ibragimov and Zeitouni \cite{IZ} then proved generalizations of these
results for random variables $\xi_{k}$ that belong to the domain
of attraction of a stable law. Among several other results relating
to the above polynomials, Edelman and Kostlan also proved in \cite[Theorem 8.2]{EK}
the formula (\ref{eq:Ex_n_f(r)}) for the expected number of zeros
of a (Gaussian) entire function in the disk $\disk r$ that we have
mentioned in the introduction.

\subsection{Notations}

Let

\[
f\left(z\right)=\sum_{n\ge0}a_{n}\xi_{n}z^{n}
\]
be a Rademacher entire function. In order to simplify the argument,
we will assume in this section that 
\[
a_{0}=1.
\]
If $n_{f}\left(r\right)$ \index{1nF@$n_{F}$}is the number of zeros
of $f$ inside the disk $\disk r=\left\{ \left|z\right|\le r\right\} $
(including multiplicities), then by Jensen's formula\index{1NF@$N_{F}$}\index{Jensen's formula}
\[
\intsy_{0}^{1}\log\left|f\left(re{}^{2\pi i\theta}\right)\right|\, d\theta=\intsy_{0}^{r}\frac{n_{f}\left(t\right)}{t}\, dt\eqdef N_{f}\left(r\right),
\]
We thus have $n_{f}\left(r\right)=\frac{dN_{f}\left(r\right)}{d\log r}=r\cdot\frac{\dd N_{f}\left(r\right)}{\dd r}.$
We also define the following functions:\index{1sigmaF@$\sigma_{F}$}\index{1sF@$s_{F}$}
\begin{eqnarray*}
\sigma_{f}^{2}\left(r\right) & = & \sum_{k\ge0}\left|a_{n}\right|^{2}r^{2n},\\
\zrf_{f}\left(r\right) & = & \frac{d\log\sigma_{f}\left(r\right)}{d\log r}=r\cdot\frac{\sigma_{f}^{\prime}\left(r\right)}{\sigma_{f}\left(r\right)}.
\end{eqnarray*}
The first of these functions measures the `expected' rate of growth
of $N_{f}\left(r\right)$, while the second one measures the rate
of growth of $n_{f}\left(r\right)$. A set $E\subset\left[1,\infty\right)$
is called a set of finite logarithmic measure if\index{finite logarithmic measure}
\[
\intsy_{E}\,\frac{\dd t}{t}<\infty.
\]
We write\index{1Fhat@$\widehat{F}$}
\[
\widehat{f}\left(re^{2\pi i\theta}\right)=\frac{f\left(re^{2\pi i\theta}\right)}{\sigma_{f}\left(r\right)},
\]
and notice that
\[
\Ex\left\{ \intsy_{\bbt}\left|\widehat{f}\left(re^{2\pi i\theta}\right)\right|^{2}\d{\theta}\right\} =1,
\]
Hence, by Theorem \ref{thm:log_int_for_Rade_Four_series},
\begin{equation}
\Ex\left\{ \left|N_{f}\left(r\right)-\log\sigma_{f}\left(r\right)\right|^{p}\right\} \le\Ex\left\{ \intsy_{\bbt}\left|\log\left|\widehat{f}\left(re^{2\pi i\theta}\right)\right|\right|^{p}\right\} \le\left(Cp\right)^{6p},\label{eq:moment_est_for_N_f(r)}
\end{equation}
for every $p\ge1$. Using this result we prove
\begin{thm}
\label{thm:as_conve_of_n_f(r)}Almost surely, for every $r\ge r_{0}\left(\omega\right)\ge1$
\[
\left|n_{f}\left(r\right)-\zrf_{f}\left(r\right)\right|\le C\sqrt{\zrf_{f}\left(r\right)}\log^{4}\zrf_{f}\left(r\right),\quad r\notin E,
\]
where $E$ is a (non-random) set of finite logarithmic measure.
\end{thm}

\subsection{Proof of Theorem \ref{thm:as_conve_of_n_f(r)}}

After the change variables $\rho=\log r$ the exceptional set that
will appear has to be of finite (Lebesgue) measure in $\rho.$

\subsubsection{The sequences $\left\{ \rho_{m}\right\} $ and $\left\{ \rho_{m}^{\pm}\right\} $}

Since the function $\zrf_{f}\left(r\right)$ tends monotonically to
infinity, we can choose a sequence $\rho_{m}\uparrow\infty$ such
$\zrf_{f}\left(e^{\rho_{m}}\right)=m^{\alpha}\log^{\beta}m$, where
$\alpha>1$, $\beta>0$ are some constants that will be chosen later.
We write $\Delta_{m}=\rho_{m+1}-\rho_{m}$ and let $\left\{ \delta_{m}\right\} $
be a sequence such that $0\le\delta_{m}\le\tfrac{1}{2}\Delta_{m}$,
again to be chosen later.

Let $\left\{ \rho_{m}^{\prime}\right\} $ be the sequence $\rho_{1}^{-},\rho_{1},\rho_{1}^{+},\rho_{2}^{-},\rho_{2},\rho_{2}^{+},\rho_{3}^{-},\rho_{3},\rho_{3}^{+},\ldots$,
where
\[
\rho_{m}^{-}=\rho_{m}-\delta_{m},\enskip\rho_{m}^{+}=\rho_{m}+\delta_{m}.
\]
The sequence $\left\{ \rho_{m}^{\prime}\right\} $ is where we have
`good sampling' of the function $N_{f}\left(e^{t}\right)$ (that is,
it is close to $\log\sigma_{f}\left(e^{t}\right)$), then using the
convexity of these functions we show that they cannot be far apart
for any $t$, outside an exceptional set of finite measure.

\subsubsection{How to `differentiate' $N_{f}\left(e^{t}\right)$?}

By (\ref{eq:moment_est_for_N_f(r)}) and Chebyshev\textquoteright{}s
inequality,
\begin{eqnarray*}
\pr{\left|N_{f}\left(e^{\rho_{m}^{\prime}}\right)-\log\sigma_{f}\left(e^{\rho_{m}^{\prime}}\right)\right|>t} & = & \pr{\left|N_{f}\left(e^{\rho_{m}^{\prime}}\right)-\log\sigma_{f}\left(e^{\rho_{m}^{\prime}}\right)\right|^{p}>t^{p}}\\
 & \le & \left(\frac{Cp^{6}}{t}\right)^{p},
\end{eqnarray*}
and choosing $t=Cep^{6}$, $p=2\log m$ we get
\[
\pr{\left|N_{f}\left(e^{\rho_{m}^{\prime}}\right)-\log\sigma_{f}\left(e^{\rho_{m}^{\prime}}\right)\right|>C\log^{6}m}\le\frac{1}{m^{2}}.
\]
Therefore, by the Borel-Cantelli lemma we have that a.s. for every
$m\ge m_{0}\left(\omega\right)$,
\begin{equation}
\left|N_{f}\left(e^{\rho_{m}^{\prime}}\right)-\log\sigma_{f}\left(e^{\rho_{m}^{\prime}}\right)\right|\le C\log^{6}m.\label{eq:samp_of_func_N_f(r)}
\end{equation}
Now, since $t\mapsto N_{f}\left(e^{t}\right)$ and $t\mapsto\log\sigma_{f}\left(e^{t}\right)$
are convex functions, we get for $\rho\in\left(\rho_{m}^{+},\rho_{m+1}^{-}\right)$,
\begin{eqnarray*}
n_{f}\left(e^{\rho}\right) & \le & n_{f}\left(e^{\rho_{m+1}^{-}}\right)\le\frac{N_{f}\left(e^{\rho_{m+1}}\right)-N_{f}\left(e^{\rho_{m+1}^{-}}\right)}{\rho_{m+1}-\rho_{m+1}^{-}}\\
 & \le & \frac{\log\sigma_{f}\left(e^{\rho_{m+1}}\right)-\log\sigma_{f}\left(e^{\rho_{m+1}^{-}}\right)}{\rho_{m+1}-\rho_{m+1}^{-}}+\frac{C\log^{6}m}{\delta_{m}}\\
 & \le & \zrf_{f}\left(e^{\rho_{m+1}}\right)+\frac{C\log^{6}m}{\delta_{m}}.
\end{eqnarray*}
In the other direction, using a similar argument, we get 
\[
n_{f}\left(e^{\rho}\right)\ge\zrf_{f}\left(e^{\rho_{m}}\right)-\frac{C\log^{6}m}{\delta_{m}}.
\]
In addition, by the choice of the sequence $\rho_{m}$,
\[
\zrf_{f}\left(e^{\rho_{m+1}}\right)\le\zrf_{f}\left(e^{\rho_{m}}\right)\left(1+\frac{C}{m}\right)\le\zrf_{f}\left(e^{\rho}\right)\left(1+\frac{C}{m}\right),
\]
and
\[
\zrf_{f}\left(e^{\rho_{m}}\right)\ge\zrf_{f}\left(e^{\rho_{m+1}}\right)\left(1-\frac{C}{m}\right)\ge\zrf_{f}\left(e^{\rho}\right)\left(1-\frac{C}{m}\right).
\]
Therefore, for $\rho\in\left(\rho_{m}^{+},\rho_{m+1}^{-}\right)$
we have
\[
\left|n_{f}\left(e^{\rho}\right)-\zrf_{f}\left(e^{\rho}\right)\right|\le\frac{C\cdot\zrf_{f}\left(e^{\rho}\right)}{m}+\frac{C\log^{6}m}{\delta_{m}}m^{\alpha-1}\le Cm^{\alpha-1}\log^{\beta}m+\frac{C\log^{6}m}{\delta_{m}}.
\]
The exceptional set is contained in $\bigcup_{m}\left(\rho_{m}^{-},\rho_{m}^{+}\right)$,
and thus we have to require $\sum\delta_{m}<\infty$. Hence, we have
the conditions:
\begin{eqnarray}
Cm^{\alpha-1}\log^{\beta}m+\frac{C\log^{6}m}{\delta_{m}} & = & o\left(m^{\alpha}\log^{\beta}m\right),\label{eq:req_diff_error}\\
\sum\delta_{m} & < & \infty,\label{eq:req_finite_excp_set}
\end{eqnarray}
and
\[
0\le\delta_{m}\le\tfrac{1}{2}\Delta_{m}.
\]

\subsubsection{Choice of the parameters}

Recall that $\sum\Delta_{m}=\infty$. We balance the requirement (\ref{eq:req_diff_error})
by choosing
\[
\delta_{m}=\frac{C\log^{6-\beta}m}{m^{\alpha-1}}.
\]
Consequently, if we want to satisfy the requirement (\ref{eq:req_finite_excp_set}),
we can select for example $\alpha=2$ and $\beta=8$. Now we must
have
\begin{equation}
\delta_{m}=\frac{C}{m\log^{2}m}\le\frac{1}{2}\Delta_{m}.\label{eq:upp_bnd_cond_for_delta_m}
\end{equation}
Notice that in principle $\Delta_{m}$ might be very small for a lot
of values of $m$, but then we can discard the corresponding intervals!
(that is, add them to the exceptional set). We denote by $E=\left\{ m_{k}\right\} $
the sequence of indices such that following holds
\[
\Delta_{m_{k}}<\frac{2C}{m_{k}\log^{2}m_{k}},
\]
and we add the sequence of intervals $\left[\rho_{m},\rho_{m+1}\right]$,
$m\in E$, to the exceptional set (the sum of their lengths is finite).
For the rest we can choose $\delta_{m}$ so that (\ref{eq:upp_bnd_cond_for_delta_m})
holds (it is clear that there are infinitely many indices that satisfy
(\ref{eq:upp_bnd_cond_for_delta_m}), since $\sum\Delta_{m}=\infty$
).

We conclude that for $\rho\in\left(\rho_{m}^{+},\rho_{m+1}^{-}\right)$,
$m\notin E$, and $m\ge m_{0}\left(\omega\right)$, we have, by the
choice of the sequence $\left\{ \rho_{m}\right\} $,
\[
\left|n_{f}\left(e^{\rho}\right)-\zrf_{f}\left(e^{\rho}\right)\right|\le\frac{C\cdot\zrf_{f}\left(e^{\rho_{m}}\right)}{m}\le C\sqrt{\zrf_{f}\left(e^{\rho_{m}}\right)}\log^{4}\zrf_{f}\left(e^{\rho_{m}}\right).
\]
Finally, for $r\ge r_{0}\left(\omega\right)$ and outside a set of
finite logarithmic measure
\[
\left|n_{f}\left(r\right)-\zrf_{f}\left(r\right)\right|\le C\sqrt{\zrf_{f}\left(r\right)}\log^{4}\zrf_{f}\left(r\right).
\]
 \hfill{}$\Box$ 
\begin{rem*}
By a similar argument to the above, and using (\ref{eq:samp_of_func_N_f(r)})
(with a different choice of the sequence $\rho_{k}^{\prime}$), one
can prove that a.s., for $r\ge r_{0}\left(\omega\right)$ we have
\begin{equation}
N_{f}\left(r\right)=\log\sigma_{f}\left(r\right)+\O\left(\log^{6}\left(\log\sigma_{f}\left(r\right)\right)\right).\label{eq:a.s._est_for_N_f(r)}
\end{equation}
No exceptional set is required in this case.\end{rem*}

\subsection{On The Exceptional Set in Theorem \ref{thm:as_conve_of_n_f(r)}}

In this section we construct an example which shows that the exceptional
set in Theorem \ref{thm:as_conve_of_n_f(r)} cannot always be avoided.
The idea is to construct a lacunary entire function $f$, so that
the number of zeros `jumps' at certain radii, while the function $s_{f}\left(r\right)$
increases more smoothly. This is an example which is essentially non-random.

Let $f$ be the following lacunary series 
\begin{equation}
f\left(z\right)=\sum_{n\ge0}\xi_{n}\frac{z^{2^{n}}}{\exp\left(2^{n}n\right)},\label{eq:f_lacunary_examp}
\end{equation}
with $\xi_{n}\in\left\{ \pm1\right\} $ (they need not be random).
We thus have
\[
\sigma_{f}^{2}\left(r\right)=\sum_{n\ge0}\frac{r^{2^{n+1}}}{\exp\left(2^{n+1}n\right)},
\]
and
\begin{eqnarray*}
s_{f}\left(r\right) & = & \frac{\dd\log\sigma_{f}}{\dd\log r}=\frac{r}{2\sigma_{f}^{2}\left(r\right)}\cdot\sum_{n\ge0}\frac{\left(2^{n+1}\right)r^{2^{n+1}-1}}{\exp\left(2^{n+1}n\right)}\\
 & = & \frac{1}{\sigma_{f}^{2}\left(r\right)}\cdot\sum_{n\ge0}\frac{\left(2^{n}\right)r^{2^{n+1}}}{\exp\left(2^{n+1}n\right)}.
\end{eqnarray*}
We show that there exists a set $E\subset\left[1,\infty\right)$ of
infinite Lebesgue measure, such that for any choice of signs $\xi_{n}\in\left\{ \pm1\right\} $,
we have
\[
\left|n_{f}\left(r\right)-s_{f}\left(r\right)\right|\ge cs_{f}\left(r\right),\quad r\in E,
\]
with some numerical constant $c>0$.

To simplify notation let us write $s=\log r$, and for $\left|z\right|=r=e^{s}$
\begin{eqnarray*}
B_{n}\left(s\right) & = & \frac{r^{2^{n+1}}}{\exp\left(2^{n}n\right)}=\exp\left(2^{n}\left(s-n\right)\right),\\
b_{n}\left(s\right) & = & \log B_{n}\left(s\right)=2^{n}\left(s-n\right),
\end{eqnarray*}
so now we have
\begin{eqnarray*}
s_{f}\left(r\right) & = & \sum_{n\ge0}\left[2^{n}B_{n}^{2}\left(s\right)\right]/\sum_{n\ge0}B_{n}^{2}\left(s\right).
\end{eqnarray*}
We note that for $s$ which is near an integer $k\in\bbn$ the largest
terms in (\ref{eq:f_lacunary_examp}) are those with indices $k-2$
and $k-1$. The rest we can put to the error term. We have
\begin{eqnarray*}
b_{k-1}\left(s\right) & = & 2^{k-1}\left(s-k+1\right),\\
b_{k-2}\left(s\right) & = & 2^{k-2}\left(s-k+2\right)=2^{k-1}\cdot\frac{s-k+2}{2},
\end{eqnarray*}
we also use the following notation for the sum of the remaining terms
\begin{eqnarray*}
E & = & E\left(z\right)=\sum_{n\notin\left\{ k-2,k-1\right\} }\xi_{n}\frac{z^{2^{n}}}{\exp\left(2^{n}n\right)},\\
E^{\prime} & = & E^{\prime}\left(s\right)=\sum_{n\notin\left\{ k-2,k-1\right\} }B_{n}^{2}\left(s\right),\\
E^{\dprime} & = & E^{\dprime}\left(s\right)=\sum_{n\notin\left\{ k-2,k-1\right\} }2^{n}B_{n}^{2}\left(s\right).
\end{eqnarray*}
We are going to find $\delta=\delta\left(k\right)$ such that $f\left(z\right)$
has at most $2^{k-2}$ zeros in $\disk{\left(k-\delta\left(k\right)\right)}$
and $f\left(z\right)$ has at least $2^{k-1}$ zeros in $\disk{\left(k+\delta\left(k\right)\right)}$.
The exceptional set will be around the points $k-\delta\left(k\right)$
(too few zeros) and $k+\delta\left(k\right)$ (too many zeros).

We start with a bound for the error terms, first for $E$ and $E^{\prime}$.
\begin{claim}
\label{clm:upp_bnd_tail_err}Let $k\in\bbn$ be sufficiently large,
and let $\left|z\right|=e^{s}$. If $\left|s-k\right|<\tfrac{1}{4}$,
then
\begin{eqnarray*}
\left|E\right| & \le & k\exp\left(\frac{s-k+3}{4}\cdot2^{k-1}\right),\\
E^{\prime} & \le & k\exp\left(\frac{s-k+3}{4}\cdot2^{k}\right).
\end{eqnarray*}
\end{claim}
\begin{proof}
The modulus of $E$ is bounded as follows
\[
\left|E\right|\le\sum_{0\le n\le k-3}\exp\left(b_{n}\left(s\right)\right)+\sum_{n\ge k}\exp\left(b_{n}\left(s\right)\right)=S_{1}+S_{2}.
\]
We write $m=k-n$. For the first sum we have
\[
S_{1}=\sum_{3\le m\le k}\exp\left(2^{k}\cdot2^{-m}\left(s-k+m\right)\right)\le\left(k-2\right)\exp\left(\frac{s-k+3}{4}\cdot2^{k-1}\right)
\]
and for the second one 
\[
S_{2}=\sum_{m\le0}\exp\left(2^{k}\cdot2^{-m}\left(s-k+m\right)\right)=\sum_{l\ge0}\exp\left(2^{k}\cdot2^{l}\left(s-k-l\right)\right)=\left(\star\right).
\]
Now, if $\left|s-k\right|<1$ we have
\begin{eqnarray*}
\left(\star\right) & \le & \exp\left(2^{k}\left(s-k\right)\right)+\sum_{l\ge1}\exp\left(2^{k}\cdot2\left(s-k-l\right)\right)\le\exp\left(2^{k}\left(s-k\right)\right)+2\\
 & \le & 2\exp\left(2\left(s-k\right)\cdot2^{k-1}\right).
\end{eqnarray*}
Since for $\left|s-k\right|<\tfrac{1}{4}$ we have $\frac{s-k+3}{4}>2\left(s-k\right)$
we get the required result for $E$. It is clear that exactly the
same method leads to the estimate for $E^{\prime}$.
\end{proof}
We now need similar estimates for the error term in the numerator
of $s_{f}\left(r\right)$.
\begin{claim}
\label{clm:upp_bnd_numer_tail_err}Let $k\in\bbn$ be sufficiently
large. If $\left|s-k\right|<\tfrac{1}{4}$, then
\[
E^{\dprime}\le4^{k+1}\exp\left(\frac{s-k+3}{4}\cdot2^{k}\right).
\]
\end{claim}
\begin{proof}
Let us write
\[
E^{\dprime}=\sum_{0\le n\le k-3}2^{n}B_{n}^{2}\left(s\right)+\sum_{n\ge k}2^{n}B_{n}^{2}\left(s\right)=S_{1}+S_{2},
\]
and again write $m=k-n$. Thus we have the following upper bounds
\[
S_{1}=\sum_{3\le m\le k}2^{k-m}\exp\left(2^{k+1}\cdot2^{-m}\left(s-k+m\right)\right)\le4^{k}\exp\left(\frac{s-k+3}{4}\cdot2^{k}\right),
\]
and similarly
\begin{eqnarray*}
S_{2} & = & \sum_{m\le0}2^{k-m}\exp\left(2^{k+1}\cdot2^{-m}\left(s-k+m\right)\right)\\
 & = & \sum_{l\ge0}2^{k+l}\exp\left(2^{k+1}\cdot2^{l}\left(s-k-l\right)\right)\\
 & \le & 2^{k}\exp\left(2^{k+1}\left(s-k\right)\right)+\sum_{l\ge1}2^{k+l}\exp\left(2^{k+1}\cdot2\left(\frac{1}{4}-l\right)\right)\\
 & \le & 2^{k+1}\exp\left(2\left(s-k\right)\cdot2^{k}\right)+2^{2+k}\exp\left(-3\cdot2^{k+2}\right)\\
 & \le & 2^{k+2}\exp\left(2\left(s-k\right)\cdot2^{k}\right).
\end{eqnarray*}
Again using the fact that for $\left|s-k\right|<\tfrac{1}{4}$ we
have that $\frac{s-k+3}{4}>2\left(s-k\right)$ we get the required
estimate for $E^{\dprime}$.
\end{proof}
The conclusion of these claims is that if $\left|z\right|=e^{s}$,
$\left|s-k\right|<\frac{1}{4}$, then
\begin{eqnarray*}
f\left(z\right) & = & \xi_{k-2}\frac{z^{2^{k-2}}}{\exp\left(2^{k-2}\left(k-2\right)\right)}+\xi_{k-1}\frac{z^{2^{k-1}}}{\exp\left(2^{k-1}\left(k-1\right)\right)}\\
 &  & +O\left(k\cdot\exp\left(\frac{s-k+3}{4}\cdot2^{k-1}\right)\right),
\end{eqnarray*}
and
\[
s_{f}\left(e^{s}\right)=\frac{2^{k-2}B_{k-2}^{2}\left(s\right)+2^{k-1}B_{k-1}^{2}\left(s\right)+O\left(4^{k+1}\exp\left(\frac{s-k+3}{4}\cdot2^{k}\right)\right)}{B_{k-2}^{2}\left(s\right)+B_{k-1}^{2}\left(s\right)+\O\left(k\cdot\exp\left(\frac{s-k+3}{4}\cdot2^{k}\right)\right)}.
\]
Now let us set $\delta_{k}=2^{2-k}\log2$. The following claims show
that $n_{f}\left(r\right)$ has a `jump' somewhere between the points
$s=k-\delta_{k}$ and $s=k+\delta_{k}$, while the function $s_{f}\left(r\right)$
increases more smoothly. Therefore, there is some discrepancy between
these functions near those points.
\begin{claim}
Let $k\in\bbn$ be sufficiently large. We have 
\[
n_{f}\left(e^{s}\right)=\begin{cases}
2^{k-2} & ,\, s\in\left(k-2\delta_{k},k-\delta_{k}\right),\\
2^{k-1} & ,\, s\in\left(k+\delta_{k},k+2\delta_{k}\right).
\end{cases}
\]
\end{claim}
\begin{proof}
We will consider only the case $s\in\left(k-2\delta_{k},k-\delta_{k}\right)$,
the second case being similar. By Rouché's Theorem it is sufficient
to show that
\[
\left|\xi_{k-2}\frac{z^{2^{k-2}}}{\exp\left(2^{k-2}\left(k-2\right)\right)}\right|>\left|\xi_{k-1}\frac{z^{2^{k-1}}}{\exp\left(2^{k-1}\left(k-1\right)\right)}\right|+\left|E\right|,
\]
for $\left|z\right|=e^{s}$, which in our notation translates to the
inequality
\[
B_{k-2}\left(s\right)>B_{k-1}\left(s\right)+\left|E\right|,
\]
which is equivalent to
\[
\exp\left(b_{k-2}\left(s\right)\right)>\exp\left(b_{k-1}\left(s\right)\right)+\left|E\right|.
\]
Let us write $s=k-\delta$, for some $\delta\in\left(\delta_{k},2\delta_{k}\right)$.
Thus,
\[
b_{n}\left(s\right)=2^{n}\left(s-n\right)=2^{n}\left(k-n-\delta\right),
\]
and so
\begin{eqnarray*}
b_{k-2}\left(s\right)-b_{k-1}\left(s\right) & = & 2^{k-2}\left(2-\delta\right)-2^{k-1}\left(1-\delta\right),\\
 & = & 2^{k-2}\delta.
\end{eqnarray*}
By our choice of $\delta_{k}$ it now follows that
\[
B_{k-2}\left(s\right)\ge2B_{k-1}\left(s\right),\quad s\in\left(k-2\delta_{k},k-\delta_{k}\right),
\]
we also note that for all such $s$ we have
\[
B_{k-2}\left(s\right)\ge\exp\left(2^{k-2}\left(2-2\delta_{k}\right)\right)=\frac{1}{4}\exp\left(2^{k-1}\right).
\]
Now, by Claim \ref{clm:upp_bnd_tail_err}, for any $s\in\left(k-2\delta_{k},k-\delta_{k}\right)$
we have for $k$ sufficiently large
\[
\left|E\right|\le k\cdot\exp\left(\frac{s-k+3}{4}\cdot2^{k-1}\right)\le k\cdot\exp\left(\frac{3}{4}\cdot2^{k-1}\right)<\frac{1}{2}B_{k-2}\left(s\right).
\]
This concludes the proof of the claim.
\end{proof}
On the other hand
\begin{claim}
If $k\in\bbn$ is sufficiently large, then
\[
s_{f}\left(e^{s}\right)\ge\frac{18}{17}\cdot2^{k-2}\left(1+o\left(1\right)\right),\quad s\in\left(k-2\delta_{k},k-\delta_{k}\right)
\]
and
\[
s_{f}\left(e^{s}\right)\le\frac{33}{34}\cdot2^{k-1}\left(1+o\left(1\right)\right),\quad s\in\left(k+\delta_{k},k+2\delta_{k}\right).
\]
\end{claim}
\begin{proof}
We again consider only the case $s\in\left(k-2\delta_{k},k-\delta_{k}\right)$.
Note that
\[
b_{k-2}\left(k-2\delta_{k}\right)-b_{k-1}\left(k-2\delta_{k}\right)=2^{k-2}\left(2\delta_{k}\right)=2\log2,
\]
and thus
\begin{eqnarray*}
4B_{k-1}\left(k-2\delta_{k}\right) & = & B_{k-2}\left(k-2\delta_{k}\right),
\end{eqnarray*}
in addition,
\[
B_{k-2}\left(k-2\delta_{k}\right)=\exp\left(2^{k-2}\left(2-2\delta_{k}\right)\right)=\frac{1}{4}\exp\left(2^{k-1}\right).
\]
By combining the estimates for the error terms from Claim \ref{clm:upp_bnd_tail_err}
and Claim \ref{clm:upp_bnd_numer_tail_err} we get
\begin{eqnarray*}
s_{f}\left(e^{k-2\delta_{k}}\right) & = & \frac{2^{k-2}B_{k-2}^{2}\left(k-2\delta_{k}\right)+2^{k-1}B_{k-1}^{2}\left(k-2\delta_{k}\right)+\O\left(k\exp\left(\frac{3}{4}\cdot2^{k}\right)\right)}{B_{k-2}^{2}\left(k-2\delta_{k}\right)+B_{k-1}^{2}\left(k-2\delta_{k}\right)+O\left(4^{k+1}\exp\left(\frac{3}{4}\cdot2^{k}\right)\right)}\\
 & = & \frac{2^{k-2}\cdot16+2^{k-1}+o\left(1\right)}{16+1+o\left(1\right)}=\frac{18}{17}\cdot2^{k-2}\left(1+o\left(1\right)\right).
\end{eqnarray*}
Since the function $s_{f}\left(r\right)$ is increasing we finished
the proof of this claim.
\end{proof}
This concludes the example for the necessity of an exceptional set.
\begin{rem*}
It should be mentioned that the Lebesgue measure of the exceptional
set is infinite. This follows from the fact that the series
\[
\sum_{k\ge1}\left[\exp\left(k+2\delta_{k}\right)-\exp\left(k+\delta_{k}\right)\right],\,\,\sum_{k\ge1}\left[\exp\left(k-\delta_{k}\right)-\exp\left(k-2\delta_{k}\right)\right],
\]
are both divergent (in fact their terms tend to infinity).\end{rem*}

\section{Equidistribution of the zeros - Proof of Theorem \ref{eq:Ex_n_f(r)}}

In this part we are interested in employing the tools developed in
the previous part of the work to study in more detail the distribution
of the zeros of Rademacher entire functions 
\[
F(z)=\sum_{n\ge0}\xi_{n}a_{n}z^{n}.
\]
By our knowledge on the behavior of GAFs, and specifically by the
Edelman-Kostlan formula \eqref{eq:Edel-Kos_for}, we expect that the
function\index{1sigmaF@$\sigma_{F}$} 
\[
\sigma_{F}^{2}(r)=\Ex\bigl\{|F(re^{{\rm i}\theta}|^{2})\bigr\}=\sum_{n\ge0}|a_{n}|^{2}r^{2n},
\]
and its logarithmic derivative\index{1sF@$s_{F}$} 
\[
s_{F}\left(r\right)=\frac{\mathrm{d}\log\sigma_{F}(r)}{\mathrm{d}\log r},
\]
will determine the rate of growth of the number of zeros of $F$,
and this is confirmed by the results of the previous section. In the
same way as we did in the Introduction, let us denote by $n_{F}(r,\alpha,\beta)$
the number of zeros of $F$ in the sector $\left\{ z\colon\alpha\le\arg z<\beta,\,\left|z\right|\le r\right\} $,
counted with multiplicities. Then the integrated counting function
is given by \index{1nF@$n_{F}$!n_{F}(r,alpha,beta)
@$n_{F}(r,\alpha,\beta)$} \index{1NF@$N_{F}$!N_{F}(r,alpha,beta)
@$N_{F}(r,\alpha,\beta)$} 
\[
N_{F}(r,\alpha,\beta)=\intsy_{0}^{r}\frac{n_{F}(r,\alpha,\beta)}{t}\d t\,.
\]
In the pioneering paper \cite{LO2}, Littlewood and Offord considered
Rademacher Taylor series that represent entire functions of a finite
and positive order of growth. That is, they assumed that\index{Rademacher Taylor series}
\[
0<\limsup_{n\to\infty}\frac{n\log n}{\log(1/|a_{n}|)}<\infty,
\]
which is equivalent to 
\[
0<\limsup_{r\to\infty}\frac{\log\log\sigma_{F}(r)}{\log r}<\infty.
\]
Under this assumption, they discovered that, for every $\epsym>0$,
a.s., 
\begin{equation}
\log|F(re^{{\rm i}\theta})|\ge\log\max_{n\ge0}\bigl(|a_{n}|r^{n}\bigr)-\O\left(r^{\epsym}\right)\label{eq:L-O}
\end{equation}
everywhere in the plane, except for the union of simply connected
domains of small diameter, in the interior of which the function $F$
has one or more zeros, and on the boundary of which $|F|$ is a constant.
They called these domains `pits'. Later, Offord \cite{Of1,Of4} extended
this result to random Taylor series with rather general independent
coefficients that represent entire functions of positive or infinite
order of growth. Other results obtained by Littlewood and Offord provided
additional information about the size of the pits and their location
(see~\cite[Theorem~4]{LO2}). However, it appears that they could
not deduce from~\eqref{eq:L-O} the asymptotics for the counting
function $n_{F}(r,\alpha,\beta)$, similar to the one we obtain in
this part of the work.

In the next section we will use a more general form of Jensen's formula
to show that, asymptotically, the size of $N_{F}(r,\alpha,\beta)$
is proportional to $\frac{\beta-\alpha}{2\pi}$, and that, actually,
the error can be bounded uniformly. This is the conclusion of the
first result:

\textit{Suppose that $a>2$ and $\gamma\in\bigl(\tfrac{1}{2}+\tfrac{1}{a},1\bigr)$.
Then, a.s. and in mean, when $r\to\infty$ and $\log\sigma_{F}(r)>\log^{a}r$,
}
\[
\sup_{0\le\alpha<\beta\le2\pi}\,\Bigl|N_{F}(r,\alpha,\beta)-\frac{\beta-\alpha}{2\pi}\,\log\sigma_{F}(r)\Bigr|=\O\left(\left(\log\sigma_{F}\left(r\right)\right)^{\gamma}\right).
\]
We then `differentiate' the above estimate (in a similar spirit to
the proof of Theorem \ref{thm:as_conve_of_n_f(r)}) to get the following
asymptotic estimate for $n_{F}\left(r,\alpha,\beta\right)$

\textit{Suppose that $a>4$ and $\gamma\in\bigl(\tfrac{3}{4}+\tfrac{1}{a},1\bigr)$.
Then there exists a set $E\subset[1,\infty)$ of finite logarithmic
measure such that, a.s. and in mean,} 
\[
\sup_{0\le\alpha<\beta\le2\pi}\,\Bigl|n_{F}(r,\alpha,\beta)-\frac{\beta-\alpha}{2\pi}\, s_{F}\left(r\right)\Bigr|=\O\left(\left(s_{F}\left(r\right)\right)^{\gamma}\right)
\]
\textit{when $r\to\infty$ , $r\notin E$ and} \textit{$\log\sigma_{F}(r)>\log^{a}r$}
.\index{Rademacher Taylor series!equidistribution of the zeros}

We mention here again that some lower bound on the growth of $\sigma_{F}$
is necessary, as we will show in Section \ref{sec:Rade_Func_with_many_real_zeros}
that certain functions have only real zeros (deterministically).

\subsubsection*{Regular coefficients $a_{n}$}

It is also worth mentioning that if the coefficients $a_{n}$ have
a very regular behaviour, namely, for some $\rho>0$ and $\Delta>0$,
\[
a_{n}=(\Delta+o(1))^{n}e^{-\frac{n}{\rho}\log n}\qquad{\rm as\ }\ n\to\infty\,,
\]
then the results of Theorem~\ref{thm:Rade_ent_funcs_ang_dist_of_zero}
follow from the lower bound \eqref{eq:L-O} of Littlewood and Offord
combined with some classical results on entire functions of completely
regular growth (see \cite[Chapter~III]{Le1}). In this case, the exceptional
set $E\subset[1,\infty)$ in the statement of the theorem is not needed.
In a recent work \cite{KZ}, Kabluchko and Zaporozhets gave an elegant
proof of similar results for the same class of very regular non-random
sequences $a_{n}$ and for arbitrary i.i.d. random variables $\xi_{k}$
satisfying $\mathcal{E}\left\{ \log^{+}|\xi|\right\} \!<\!\infty$.
Their proof relies upon estimates of a concentration function combined
with some potential theory. Apparently, it would
not work for sequences $a_{n}$, that do not behave very regulary.

\subsubsection*{Equidistribution for zeros of polynomials}

Let us denote (only for the purpose of this explanation) the number
of zeros of the polynomial $g_{m}\left(z\right)=\sum_{k=0}^{m}c_{k}z^{k}$
inside the sector $0\le\alpha<\arg z<\beta\le2\pi$ by $n\left(\alpha,\beta\right)$.
Erd\H{o}s and Tur�n considered the uniform distribution of the zeros
of $g_{m}$. They showed in \cite{ET} that
\[
\left|n\left(\alpha,\beta\right)-\frac{\left(\beta-\alpha\right)m}{2\pi}\right|<16\sqrt{m\cdot\log\left(\frac{\sum_{k=0}^{m}\left|c_{k}\right|}{\sqrt{\left|c_{0}c_{m}\right|}}\right)},\quad\forall\,0\le\alpha<\beta\le2\pi,
\]
which, for example, implies the equidistribution of the zeros in the
case where all the coefficients are bounded: $c<\left|c_{k}\right|<C$.
Unfortunately, it seems that this result cannot be used to prove the
equidistribution of the zeros of entire functions. The key difference
is that in \cite{ET}, the limiting measure of the zero counting measure
is $m_{\bbt}$, the Lebesgue measure on the circle. In that case equidistribution
can be phrased in terms of the absolute values $\left|c_{k}\right|$.
In other cases (such as $m_{\bbd}$), this is not the case. In general,
for entire functions the arguments of the coefficients are important
for determining the equidistribution of the zeros.

\subsection{Proof of Theorem \ref{thm:Rade_ent_funcs_ang_dist_of_zero}}

In the rest of this section we will prove Theorem~\ref{thm:Rade_ent_funcs_ang_dist_of_zero},
by estimating the `error functions' 
\begin{align*}
E_{F}(r) & =\sup_{0\le\alpha<\beta\le2\pi}\,\Bigl|N_{F}(r,\alpha,\beta)-\frac{\beta-\alpha}{2\pi}\,\log\sigma_{F}(r)\Bigr|\,,\\
e_{F}(r) & =\sup_{0\le\alpha<\beta\le2\pi}\,\Bigl|n_{F}(r,\alpha,\beta)-\frac{\beta-\alpha}{2\pi}\,\zrf_{F}(r)\Bigr|\,.
\end{align*}
The proof starts with a Jensen-type estimate for $|E_{F}(r)|$. This
part uses some arguments that are customary in the theory of entire
functions, and is non-probabilistic. Then, combining the Jensen-type
estimate with the log-integrability of Rademacher Fourier series,
we estimate the error function $E_{F}$, first in mean, and then almost
surely. We proceed with a growth lemma about functions of a real variable.
This lemma is a version of the classical Borel-Nevanlinna lemma, which
is widely used in the theory of entire and meromorphic functions.
This part is also non-probabilistic. At last, applying the growth
lemma, we `differentiate' estimates for $E_{F}(r)$ and obtain their
counterparts for $e_{F}(r)$.

\subsection{Jensen-type estimate}

In this section 
\[
G(z)=\sum_{n\ge0}g_{n}z^{n}
\]
is a non-random analytic function in the closed disk $\{|z|\le r\}$
with $r\ge2$. We assume that $|G(0)|=1$ and fix the value of $\arg G(0)$.
If $G$ does not vanish on the segment $\{z=te^{{\rm i}\theta}\colon0\le t\le r\}$,
then we take a continuous branch of $\arg G(te^{{\rm i}\theta})$
and let 
\[
v(t,\theta)=\arg G(te^{{\rm i}\theta})-\arg G(0)\,.
\]
By $A_{G}$ we denote various positive constants, which depend only
on the sequence of absolute values $|g_{n}|$ of the Taylor coefficients
of $G$.

Put 
\[
\sigma_{G}^{2}(t)=\sum_{n\ge0}|g_{n}|^{2}t^{2n},\qquad L_{G}(t,\theta)=\log|G(te^{{\rm i}\theta})|=\widehat{L}_{G}(t,\theta)+\log\sigma_{G}(t),
\]
and let 
\[
L_{G}^{*}(t,\delta)=\sup_{m(I)=\delta}\,\intsy_{I}|\widehat{L}_{G}(t,\theta)|\d{m\left(\theta\right)}\,,
\]
where the supremum is taken over all arcs $I\subset\bbt$ of length
$\delta$.
\begin{lem}
\textup{{[}Jensen-type estimate{]}}\label{lem:Jensen-type-ineq} For
every $r\ge2$ and every $\delta\in(0,2\pi]$, we have 
\begin{multline*}
\sup_{0\le\alpha<\beta\le2\pi}\,\bigl|N_{G}(r,\alpha,\beta)-\frac{\beta-\alpha}{2\pi}\,\log\sigma_{G}(r)\bigr|\\
\lesssim\delta\log\sigma_{G}(r)+\frac{1}{\delta^{2}}\,\intsy_{1}^{r}L_{G}^{*}(t,\delta)\,\log\Bigl(\frac{r}{t}\Bigr)\,\frac{\dd t}{t}+\intsy_{\bbt}|\widehat{L}_{G}(t,\theta)|\d{m(\theta)}+A_{G}\log r\,.
\end{multline*}

\end{lem}
Our starting point is a Jensen-type integral formula~\cite[Chapter~III, Eq.~(3.04)]{Le1},
which is a straightforward consequence of the argument principle and
the Cauchy-Riemann equations.\index{Jensen's formula!for sectors}
\begin{lem}
\textup{{[}Jensen-type integral formula{]}}\label{lem:Jensen-type-formula}
Suppose the function $G$ does not vanish on the segments $\{z=te^{{\rm i}\alpha}\colon0\le t\le r\}$
and $\{z=te^{{\rm i}\beta}\colon0\le t\le r\}$. Then 
\[
N_{G}(r,\alpha,\beta)=\frac{1}{2\pi}\intsy_{\alpha}^{\beta}L_{G}(r,\theta)\d{\theta}+\frac{1}{2\pi}\intsy_{0}^{r}\frac{v(t,\alpha)-v(t,\beta)}{t}\d t\,.
\]

\end{lem}
\noindent \textit{Proof of Lemma~\ref{lem:Jensen-type-ineq}:} First,
we observe that in order to prove Lemma~\ref{lem:Jensen-type-ineq},
it suffices to prove the upper bound 
\begin{multline}
N_{G}(r,\alpha,\beta)-\frac{\beta-\alpha}{2\pi}\,\log\sigma_{G}(r)\\
\lesssim\delta\log\sigma_{G}(r)+\frac{1}{\delta^{2}}\,\intsy_{1}^{r}L_{G}^{*}(t,\delta)\,\log\bigl(\frac{r}{t}\bigr)\,\frac{\dd t}{t}+\intsy_{\bbt}|\widehat{L}_{G}(t,\theta)|\d{m(\theta)}+A_{G}\log r\,,\label{eq:upper_bound_for_N}
\end{multline}
uniformly with respect to $\alpha$ and $\beta$. Indeed, having~\eqref{eq:upper_bound_for_N},
observing that $N_{G}(r,\alpha,\beta)=N_{G}(r,0,2\pi)-N_{G}(r,\beta,\alpha+2\pi)$,
and recalling that, by the classical Jensen formula, 
\[
N_{G}(r,0,2\pi)=\intsy_{\bbt}\log|G(re^{{\rm i}\theta})|\d{\theta}-\log|G(0)|=\log\sigma_{G}(r)+\intsy_{\bbt}\widehat{L}_{G}(t,\theta)\d{m(\theta)}\,,
\]
we get the lower bound for $N(r,\alpha,\beta)$, which matches the
upper one.

By Lemma \ref{lem:Jensen-type-formula}, 
\begin{multline*}
N_{G}(r,\alpha,\beta)-N_{G}(1,\alpha,\beta)\\
=\frac{1}{2\pi}\intsy_{\alpha}^{\beta}L_{G}(r,\theta)\d{\theta}-\frac{1}{2\pi}\intsy_{\alpha}^{\beta}L_{G}(1,\theta)\d{\theta}+\frac{1}{2\pi}\intsy_{1}^{r}\frac{v(t,\alpha)-v(t,\beta)}{t}\d{t\,.}
\end{multline*}
Once again, using the classical Jensen formula (and recalling that
$|G(0|=1$), we see that the terms $N_{G}(1,\alpha,\beta)$ and ${\displaystyle \intsy_{\alpha}^{\beta}|L_{G}(1,\theta)|\d{m(\theta)}}$
are bounded by $C\cdot\log\sigma_{G}(2)$, whence 
\[
N_{G}(r,\alpha,\beta)\le\frac{1}{2\pi}\intsy_{\alpha}^{\beta}L_{G}(r,\theta)\d{\theta}+\frac{1}{2\pi}\intsy_{1}^{r}\frac{v(t,\alpha)-v(t,\beta)}{t}\d t\,+A_{G}\,.
\]
By the Cauchy-Riemann equations, for $t\ge1$ we have 
\[
v(t,\theta)=v(1,\theta)-\frac{\dd}{\dd\theta}\intsy_{1}^{t}L_{G}(s,\theta)\,\frac{\dd s}{s}\,,
\]
whence 
\begin{multline*}
\frac{1}{2\pi}\intsy_{1}^{r}\frac{v(t,\alpha)-v(t,\beta)}{t}\d t\\
=\bigl[v(1,\alpha)-v(1,\beta)\bigr]\log r+\frac{1}{2\pi}\Bigl[\frac{\dd}{\dd\theta}\intsy_{1}^{r}L_{G}(t,\theta)\log\Bigl(\frac{r}{t}\Bigr)\,\frac{\dd t}{t}\Bigr]_{\alpha}^{\beta}\,.
\end{multline*}
By a standard complex analysis estimate (see, for instance,~\cite[Lemma~6, Chapter~VI]{Le1}),
for every $\theta$ such that $G$ does not vanish on the segment
$\bigl\{ te^{{\rm i}\theta}\colon0\le t\le1\bigr\}$, we have $\bigl|v(1,\theta)\bigr|\le A_{G}$
. Therefore, 
\begin{equation}
N_{G}(r,\alpha,\beta)\le\frac{1}{2\pi}\intsy_{\alpha}^{\beta}L_{G}(r,\theta)\d{\theta}+\frac{1}{2\pi}\Bigl[\frac{\dd}{\dd\theta}\intsy_{1}^{r}L_{G}(t,\theta)\log\Bigl(\frac{r}{t}\Bigr)\,\frac{\dd t}{t}\Bigr]_{\alpha}^{\beta}+A_{G}\log r\,.\label{eq:Jensen_long}
\end{equation}

Given $\delta\in(0,2\pi]$, we have 
\[
N_{G}(r,\alpha,\beta)\le\frac{1}{\delta}\,\intsy_{0}^{\delta}N_{G}(r,\alpha-\phi,\beta+\phi)\d{\phi}\,.
\]
Without loss of generality we assume that the function $G$ does not
vanish on the unit circle $\bbt$. Then the derivative 
\[
\frac{\dd}{\dd\theta}\intsy_{1}^{r}L_{G}(t,\theta)\log\Bigl(\frac{r}{t}\Bigr)\,\frac{\dd t}{t}=\Re\Bigl\{{\rm i}e^{{\rm i}\theta}\intsy_{1}^{r}\frac{G'(te^{{\rm i}\theta})}{G(te^{{\rm i}\theta})}\log\Bigl(\frac{r}{t}\Bigr)\,\frac{\dd t}{t}\Bigr\}
\]
is a bounded function of $\theta$, which may have at most finitely
many points of discontinuity%
\footnote{This follows from the fact that, for $r>1$ and $1<\rho\le r$, 
\[
\phi\mapsto\intsy_{1}^{r}\frac{1}{t-\rho e^{{\rm i}\phi}}\,\log\Bigl(\frac{r}{t}\Bigr)\frac{\dd t}{t}
\]
is a bounded function continuous everywhere except at the point $\phi=0$.%
}. Therefore, we can apply the Newton-Leibnitz formula when averaging
the RHS of inequality~\eqref{eq:Jensen_long}. We get 
\begin{align*}
 & N_{G}(r,\alpha,\beta)\le\frac{1}{\delta}\,\intsy_{0}^{\delta}\dd\phi\,\frac{1}{2\pi}\intsy_{\alpha-\phi}^{\beta+\phi}L_{G}(r,\theta)\d{\theta}+A_{G}\log r\\
 & \qquad+\frac{1}{\delta}\,\frac{1}{2\pi}\,\intsy_{1}^{r}\bigl[L_{G}(t,\beta+\delta)-L_{G}(t,\beta)-L_{G}(t,\alpha)+L_{G}(t,\alpha-\delta)\bigr]\log\Bigl(\frac{r}{t}\Bigr)\,\frac{\dd t}{t}\,.
\end{align*}
Applying the same average one more time, we get 
\begin{align*}
N_{G}(r,\alpha,\beta) & \le\frac{1}{\delta^{2}}\,\intsy_{0}^{\delta}\dd\phi_{1}\intsy_{0}^{\delta}\dd\phi\,\frac{1}{2\pi}\intsy_{\alpha-\phi_{1}-\phi}^{\beta+\phi_{1}+\phi}L_{G}(r,\theta)\d{\theta}+A_{G}\log r\\
 & \quad+\frac{1}{\delta^{2}}\,\frac{1}{2\pi}\,\intsy_{1}^{r}\log\Bigl(\frac{r}{t}\Bigr)\,\frac{\dd t}{t}\Bigl[\intsy_{\beta+\delta}^{\beta+2\delta}-\intsy_{\beta}^{\beta+\delta}+\intsy_{\alpha-2\delta}^{\alpha-\delta}-\intsy_{\alpha-\delta}^{\alpha}\Bigr]L_{G}(t,\theta)\d{\theta}\\
 & \le\frac{1}{\delta^{2}}\intsy_{0}^{\delta}\dd\phi_{1}\intsy_{0}^{\delta}\dd\phi\,\frac{\beta-\alpha+2\phi_{1}+2\phi}{2\pi}\,\log\sigma_{G}(r)+A_{G}\log r\\
 & \quad+\frac{1}{\delta^{2}}\,\intsy_{0}^{\delta}\dd\phi_{1}\intsy_{0}^{\delta}\dd\phi\,\frac{1}{2\pi}\intsy_{\alpha-\phi_{1}-\phi}^{\beta+\phi_{1}+\phi}\widehat{L}_{G}(r,\theta)\d{\theta}\\
 & \quad+\frac{1}{\delta^{2}}\,\frac{1}{2\pi}\,\intsy_{1}^{r}4L_{G}^{*}(t,\delta)\,\log\Bigl(\frac{r}{t}\Bigr)\,\frac{\dd t}{t}\\
 & \le\frac{\beta-\alpha}{2\pi}\log\sigma_{G}(r)+{\rm RHS\ of\ }\eqref{eq:upper_bound_for_N}\,,
\end{align*}
completing the proof of Lemma~\ref{lem:Jensen-type-ineq}. \hfill{}$\Box$

\subsection{Proof of the first part of Theorem~\ref{thm:Rade_ent_funcs_ang_dist_of_zero}}

In what follows, we denote by $A$ various positive constants depending
on the absolute values of the non-random coefficients $|a_{n}|$ in
the Taylor expansion of $F(z)$.

\subsubsection{Mean estimate}

We start with an estimate for the random maximal function 
\[
L_{F}^{*}(t,\delta)=\sup\left\{ \intsy_{I}|\widehat{L}_{G}(t,\theta)|\d{m(\theta)}\,:\, I\subset\bbt\mbox{ is an arc, }m(I)=\delta\right\} \,.
\]

\begin{claim}
\label{claim:max-function}

(i) For every $t\ge1$ and every $0<\delta\le2\pi$, $\Ex\left\{ L_{F}^{*}(t,\delta)\right\} \lesssim\delta\,\log^{6}\left(C\delta^{-1}\right)$.

(ii) For every $p\ge1$, every $0<\delta\le2\pi$, and every $r\ge2$,
\[
\Bigl\|\intsy_{1}^{r}L_{F}^{*}(t,\delta)\,\log\Bigl(\frac{r}{t}\Bigr)\,\frac{\dd t}{t}\Bigr\|_{L^{p}(\Omega)}\le Cp^{6}\cdot\delta^{1-\frac{1}{p}}\cdot\log^{2}r\,.
\]
\end{claim}
\begin{proof}
\noindent Fix arbitrary $p\ge1$, $\delta\le1$, and $t\ge1$. Then,
for every arc $I\subset\bbt$ of measure $\delta$, we have 
\[
\Ex\Bigl(\intsy_{I}|\widehat{L}_{F}(t,\theta)|\d{m(\theta)}\Bigr)^{p}\le\Ex\Bigl(\intsy_{\bbt}|\widehat{L}_{F}(t,\theta)|^{p}\d{m(\theta)}\cdot|I|^{p-1}\Bigr)\le\bigl(Cp\bigr)^{6p}\delta^{p-1}\,.
\]
This follows from Hölder's inequality combined with the log-integrability
of Rademacher Fourier series (Theorem~\ref{thm:log_int_for_Rade_Four_series}).
Therefore, 
\[
\bigl\| L_{F}^{*}(t,\delta)\|_{L^{p}(\Omega)}\le Cp^{6}\,\delta^{1-\frac{1}{p}}\,.
\]
Letting $p=\log(C/\delta)$, we get estimate (i).

Estimate (ii) follows by application of the integral Minkowski inequality:
\begin{eqnarray*}
\Bigl\|\intsy_{1}^{r}L_{F}^{*}(t,\delta)\log\Bigl(\frac{r}{t}\Bigr)\frac{\dd t}{t}\Bigr\|_{L^{p}(\Omega)} & \le & \intsy_{1}^{r}\bigl\| L_{F}^{*}(t,\delta)\bigr\|_{L^{p}(\Omega)}\log\Bigl(\frac{r}{t}\Bigr)\frac{\dd t}{t}\\
 & \stackrel{{\rm (i)}}{\le} & Cp^{6}\delta^{1-\frac{1}{p}}\cdot\log^{2}r\,,
\end{eqnarray*}
completing the proof.
\end{proof}
\medskip{}
Now, combining Lemma~\ref{lem:Jensen-type-ineq} with Theorem~\ref{thm:log_int_for_Rade_Four_series}
and using the previous claim, we get 
\begin{multline*}
\Ex\bigl\{ E_{F}(r)\bigr\}\lesssim\delta\,\log\sigma_{F}(r)+\frac{1}{\delta^{2}}\intsy_{1}^{r}\Ex\bigl\{ L_{F}^{*}(t,\delta)\bigr\}\,\log\Bigl(\frac{r}{t}\Bigr)\,\frac{\dd t}{t}+A\log r\\
\lesssim\delta\,\log\sigma_{F}(r)+\frac{1}{\delta}\,\log^{6}\Bigl(\frac{C}{\delta}\Bigr)\cdot\log^{2}r+A\log r\,.
\end{multline*}
Choosing $\delta$ that balances the RHS and recalling that $\log\sigma_{F}(r)\ge\log^{a}r$
with some $a>2$, we obtain 
\begin{equation}
\Ex\bigl\{ E_{F}(r)\bigr\}\le A\bigl(\log\sigma_{F}(r)\bigr)^{\frac{1}{2}}\cdot\bigl(\log\log\sigma_{F}(r)\bigr)^{3}\cdot\log r\,,\label{eq:E_F-mean}
\end{equation}
which does not exceed $A\log^{\gamma}\sigma_{F}(r)$ with $\tfrac{1}{2}+\tfrac{1}{a}<\gamma<1$
and $r$ sufficiently large. \hfill{}$\Box$

\subsubsection{Almost sure estimate}

The proof of the a.s. estimate uses an $L^{p}(\Omega)$-estimate of
$E_{F}(r)$ with large values of $p$. Put 
\[
I_{p}(r)=\bigl(\log\sigma_{F}(r)\bigr)^{\frac{p+1}{2p+1}}\cdot\bigl(\log^{2}r\bigr)^{\frac{p}{2p+1}}\,.
\]

\begin{lem}
\label{lemma:L^p} For every $p\ge1$ and every $r\ge2$ with $\log\sigma_{F}(r)\ge\log^{2}r$,
we have 
\begin{equation}
\|E_{F}(r)\|_{L^{p}(\Omega)}\le Ap^{6}\cdot I_{p}(r)\,.\label{eq:L^p}
\end{equation}
\end{lem}
\begin{proof}
Combining Lemma~\ref{lem:Jensen-type-ineq}, Theorem~\ref{thm:log_int_for_Rade_Four_series}
and Claim~\ref{claim:max-function}, we get 
\begin{multline*}
\|E_{F}(r)\|_{L^{p}(\Omega)}\lesssim\delta\,\log\sigma_{F}(r)+\frac{1}{\delta^{2}}\Bigl\|\intsy_{1}^{r}L_{F}^{*}(t,\delta)\,\log\bigl(\frac{r}{t}\bigr)\,\frac{\dd t}{t}\Bigr\|_{L^{p}(\Omega)}+p^{6}+A\log r\\
\stackrel{{\rm (ii)}}{\lesssim}\delta\,\log\sigma_{F}(r)+p^{6}\,\delta^{-1-\frac{1}{p}}\,\log^{2}r+p^{6}+A\log r\,.
\end{multline*}
To balance the RHS, we let 
\[
\delta=\Bigl(\frac{p^{6}\log^{2}r}{\log\sigma_{F}(r)}\Bigr)^{\frac{p}{2p+1}};
\]
this choice is possible provided that $\log\sigma_{F}(r)\ge p^{6}\cdot\log^{2}r$.
This gives us 
\[
\|E_{F}(r)\|_{L^{p}(\Omega)}\le Ap^{6}\cdot I_{p}(r)\,.
\]
If $\log\sigma_{F}(r)\le p^{6}\,\log^{2}r$, we let $\delta=1$. In
this case, 
\[
\|E_{F}(r)\|_{L^{p}(\Omega)}\le Ap^{6}\,\log^{2}r\,,
\]
which is again consistent with the RHS of~\eqref{eq:L^p}. This proves
Lemma~\ref{lemma:L^p}.
\end{proof}
\noindent Now we are ready to prove the a.s. estimate for $E_{F}(r)$.
Using Chebyshev's inequality and Lemma~\ref{lemma:L^p}, for every
$T>0$, we get 
\[
\Pr\bigl\{ E_{F}(r)>T\bigr\}\le T^{-p}\,\|E_{F}(r)\|_{L^{p}(\Omega)}^{p}\le T^{-p}\left(Ap^{6}\cdot I_{p}(r)\right)^{p}\,,
\]
whence, for every $p\ge1$ and every $r\ge2$, 
\[
\Pr\bigl\{ E_{F}(r)>e\cdot Ap^{6}\cdot I_{p}(r)\bigr\}\le e^{-p}\,.
\]
For sufficiently large $m\in\bbn$, we let $p_{m}=2\log m$, and choose
$r_{m}$ so that $\log\sigma_{F}(r_{m})=m$. Then 
\[
\Pr\bigl\{ E_{F}(r_{m})\ge A\log^{6}m\cdot I_{p_{m}}(r_{m})\bigr\}\le\frac{1}{m^{2}}\,.
\]
Therefore, by the Borel-Cantelli lemma, there exists an event $\Omega_{0}\subset\Omega$
of probability zero such that, for each $\omega\in\Omega\setminus\Omega_{0}$,
there exists $m_{0}(\omega)\in\bbn$ with the following property:
\begin{equation}
E_{F}(r_{m})\le AI_{p_{m}}(r_{m})\cdot\log^{6}m\,,\qquad m\ge m_{0}(\omega)\,.\label{eq:Borel-Cantelli}
\end{equation}
Fix $\omega\in\Omega\setminus\Omega_{0}$, and assume that $r_{m}\le r\le r_{m+1}$
and $m\ge m_{0}(\omega)$. For each $0\le\alpha<\beta\le2\pi$, we
have 
\begin{multline*}
N_{F}(r_{m},\alpha,\beta)-\frac{\beta-\alpha}{2\pi}\log\sigma_{F}(r_{m})-1\le N_{F}(r,\alpha,\beta)-\frac{\beta-\alpha}{2\pi}\log\sigma_{F}(r)\\
\le N_{F}(r_{m+1},\alpha,\beta)-\frac{\beta-\alpha}{2\pi}\log\sigma_{F}(r_{m+1})+1\,.
\end{multline*}
Therefore, 
\begin{multline*}
E_{F}(r)\le\max\bigl\{ E_{F}(r_{m}),E_{F}(r_{m+1})\bigr\}+1\\
\stackrel{\eqref{eq:Borel-Cantelli}}{\le}A\max\bigl\{ I_{p_{m}}(r_{m}),I_{p_{m+1}}(r_{m+1})\bigr\}\cdot\bigl(\log\log\sigma_{F}(r_{m})\bigr)^{6}\,.
\end{multline*}
Note that 
\begin{multline*}
I_{p_{m}}(r_{m})=\bigl(\log\sigma_{F}(r_{m})\bigr)^{\frac{2\log m+1}{4\log m+1}}\cdot\bigl(\log^{2}r_{m}\bigr)^{\frac{2\log m}{4\log m+1}}\\
=m^{\frac{1}{2(4\log m+1)}}\cdot\bigl(\log\sigma_{F}(r_{m})\bigr)^{\frac{1}{2}}\cdot\bigl(\log r_{m}\bigr)^{1-\frac{1}{4\log m+1}}\lesssim\bigl(\log\sigma_{F}(r_{m})\bigr)^{\frac{1}{2}}\,\log r_{m}\,.
\end{multline*}
Due to the convexity of the function $t\mapsto\log\sigma_{F}(e^{t})$,
we have 
\[
\frac{\log\sigma_{F}(r_{m})-\log\sigma_{F}(1)}{\log r_{m}}\le\frac{\log\sigma(r_{m+1})-\log\sigma(r_{m})}{\log r_{m+1}-\log r_{m}}\,,
\]
that is, $(m-c)(\log r_{m+1}-\log r_{m})\le\log r_{m}$, whence $\log r_{m+1}\le2\log r_{m}$,
at least when $m$ is large enough. Hence, $I_{p_{m+1}}(r_{m+1})\lesssim\bigl(\log\sigma_{F}(r_{m})\bigr)^{\frac{1}{2}}\,\log r_{m}$.
Recalling that $\log\sigma_{F}(r)>\log^{a}r$, we obtain 
\begin{equation}
E_{F}(r)\le A\,\bigl(\log\sigma_{F}(r)\bigr)^{\frac{1}{2}}\cdot\bigl(\log\log\sigma_{F}(r)\bigr)^{6}\cdot\log r\le A\log^{\gamma}\sigma_{F}(r)\,,\label{eq:E_F-a.s.}
\end{equation}
as above, with $\tfrac{1}{2}+\tfrac{1}{a}<\gamma<1$. This concludes
the proof of the first part of Theorem~\ref{thm:Rade_ent_funcs_ang_dist_of_zero}.
\mbox{}\hfill{}$\Box$

\subsection{A growth lemma}

Given $b>1$, a continuous non-decreasing function $\Phi\colon[1,\infty)\to(1,\infty)$,
${\displaystyle \lim_{r\to\infty}\Phi(r)=\infty}$, and a continuous
non-increasing function $\errfunc\colon[1,\infty)\to(0,1]$, we define
$X\subset[1,\infty)$ as the set of $r$'s for which 
\begin{multline*}
(1-\errfunc(r))\Phi(r)<\Phi\bigl((1+\errfunc(r)\log^{-b}\Phi(r))^{-1}r\bigr)\\
\le\Phi\bigl((1+\errfunc(r)\log^{-b}\Phi(r))r\bigr)<(1+\errfunc(r))\Phi(r)\,.
\end{multline*}

\begin{lem}
\label{lemma:B-N-type} The complement $[1,\infty)\setminus X$ has
finite logarithmic measure.\end{lem}
\begin{proof}
\noindent Letting 
\[
E=\left\{ r\in[1,\infty)\colon\Phi\bigl(r+r\,\errfunc(r)\log^{-b}\Phi(r)\bigr)\ge(1+\errfunc(r))\Phi(r)\right\} ,
\]
we show that the set $E$ has finite logarithmic measure. This gives
half of the statement. The proof of the other half is very similar
and we skip it. We suppose that the set $E$ is non-empty and unbounded,
otherwise, there is nothing to prove. We construct inductively sequences
$r_{n}<r_{n}'<r_{n+1}$ such that $E$ is contained in the union of
the intervals $[r_{n},r_{n}']$, and then show that the series $\sum_{n}(r_{n}'-r_{n})/r_{n}$
converges.

Let $r_{1}=\inf\left\{ r\colon r\in E\right\} $. Since $E$ is closed,
$r_{1}\in E$. Suppose that the values $r_{1},\ldots,r_{n}$ have
already been defined. Put 
\[
r_{n}'=\bigl(1+\errfunc(r_{n})\log^{-b}\Phi(r_{n})\bigr)\cdot r_{n},\qquad r_{n+1}=\inf\bigl\{ r\ge r_{n}'\colon r\in E\bigr\}\,,
\]
and let 
\[
P_{n+1}=\sum_{1\le j\le n}\log\bigl(1+\errfunc(r_{j})\bigr)\le\sum_{1\le j\le n}\errfunc(r_{j})\,.
\]
Note that $r_{n}\uparrow\infty$. Otherwise, $r_{n}\uparrow r^{*}$,
and then also $r_{n}'\uparrow r^{*}$. Hence, 
\[
\errfunc(r_{n})\log^{-b}\Phi(r_{n})r_{n}=r_{n}'-r_{n}\to0\,.
\]
On the other hand, by the continuity of $\Phi$ and $\errfunc$, the
LHS converges to $\errfunc(r^{*})\log^{-b}\Phi(r^{*})r^{*}$, which
cannot be zero. By construction, $E\subset\bigcup_{n\ge1}\left[r_{n},r_{n}^{\prime}\right]$.

Next, 
\begin{eqnarray*}
\Phi(r_{n+1})\ge\Phi(r_{n}') & = & \Phi\bigl(r_{n}+\errfunc(r_{n})\log^{-b}\Phi(r_{n})r_{n}\bigr)\\
 & \stackrel{r_{n}\in E}{\ge} & \bigl(1+\errfunc(r_{n})\bigr)\Phi(r_{n})\ge\,\ldots\\
 & \ge & \prod_{1\le j\le n}\bigl(1+\errfunc(r_{j})\bigr)\cdot\Phi(r_{1})\ge e^{P_{n+1}}\,.
\end{eqnarray*}
Using this estimate, we get 
\[
\intsy_{E}\frac{\dd r}{r}\le\sum_{n\ge1}\frac{r_{n}'-r_{n}}{r_{n}}=\sum_{n\ge1}\errfunc(r_{n})\log^{-b}\Phi(r_{n})\le\sum_{n\ge1}\errfunc(r_{n})P_{n}^{-b}\,.
\]

Now we consider two cases. If the $P_{n}$'s converge to a finite
limit, then, by the definition of that sequence, the series $\sum_{n}\errfunc(r_{n})$
also converges, and therefore, the set $E$ has finite logarithmic
measure. If the $P_{n}$'s increase to $\infty$, then we note that,
for $n$ large enough, 
\[
P_{n}^{b}=\left(\sum_{1\le j\le n-1}\log\bigl(1+\errfunc(r_{j})\bigr)\right)^{b}\ge c_{b}\left(\sum_{1\le j\le n}\errfunc(r_{j})\bigr)\right)^{b}\,.
\]
Since $b>1$, the series 
\[
\sum_{n\ge1}\errfunc(r_{n})\left(\sum_{1\le j\le n}\errfunc(r_{j})\right)^{-b}
\]
always converges. Once again, we conclude that the set $E$ also has
finite logarithmic measure.
\end{proof}

\subsection{Proof of the second part of Theorem~\ref{thm:Rade_ent_funcs_ang_dist_of_zero}}

Fix $\alpha$ and $\beta$ such that $0\le\alpha<\beta\le2\pi$, and
fix $b>1$. Let $r_{1}=r\left(1-\errfunc(r)\log^{-b}\zrf_{F}(r)\right)\ge2$,
and $r_{2}=r\left(1+\errfunc(r)\log^{-b}\zrf_{F}(r)\right)$. The
continuous decreasing function $\errfunc(r)$ will be specified later.

Since the functions $t\mapsto N_{F}(e^{t})$ and $t\mapsto\log\sigma_{F}(e^{t})$
are convex, we have 
\begin{multline*}
n_{F}(r_{1},\alpha,\beta)\le\frac{N_{F}(r,\alpha,\beta)-N_{F}(r_{1},\alpha,\beta)}{\log r-\log r_{1}}\\
\le n_{F}(r,\alpha,\beta)\le\frac{N_{F}(r_{2},\alpha,\beta)-N_{F}(r,\alpha,\beta)}{\log r_{2}-\log r}\le n_{F}(r_{2},\alpha,\beta)\,,
\end{multline*}
and 
\[
\zrf_{F}(r_{1})\le\frac{\log\sigma_{F}(r)-\log\sigma_{F}(r_{1})}{\log r-\log r_{1}}\le\zrf_{F}(r)\le\frac{\log\sigma_{F}(r_{2})-\log\sigma_{F}(r)}{\log r_{2}-\log r}\le\zrf_{F}(r_{2})\,.
\]
Therefore, 
\begin{gather*}
n_{F}(r,\alpha,\beta)-\frac{\beta-\alpha}{2\pi}\,\zrf_{F}(r)\\
=n_{F}(r,\alpha,\beta)-\frac{\beta-\alpha}{2\pi}\,\zrf_{F}(r_{2})+\frac{\beta-\alpha}{2\pi}\,\zrf_{F}(r_{2})-\frac{\beta-\alpha}{2\pi}\,\zrf_{F}(r)\\
\le\frac{\bigl[N_{F}(r_{2},\alpha,\beta)-\frac{\beta-\alpha}{2\pi}\log\sigma_{F}(r_{2})\bigr]-\bigl[N_{F}(r,\alpha,\beta)-\frac{\beta-\alpha}{2\pi}\log\sigma_{F}(r)\bigr]}{\log r_{2}-\log r}\\
+\bigl[\zrf_{F}(r_{2})-\zrf_{F}(r)\bigr]\\
\le\frac{E_{F}(r_{2})+E_{F}(r)}{\log r_{2}-\log r}+\bigl[\zrf_{F}(r_{2})-\zrf_{F}(r)\bigr]\,.
\end{gather*}
Applying the growth Lemma~\ref{lemma:B-N-type} with the function
$\Phi(r)=\zrf_{F}(r)$, we conclude that 
\[
n_{F}(r,\alpha,\beta)-\frac{\beta-\alpha}{2\pi}\,\zrf_{F}(r)\le C\frac{E_{F}(r)+E_{F}(r_{2})}{\errfunc(r)}\,\bigl(\log\zrf_{F}(r)\bigr)^{b}+\errfunc(r)\zrf_{F}(r)\,,
\]
provided that $r\in[2,\infty)\setminus E$, where $E$ is a set of
finite logarithmic measure. Similarly, 
\begin{align*}
n(r)-\frac{\beta-\alpha}{2\pi}\,\zrf_{F}(r) & \ge-\frac{S(r)+S(r_{1})}{\log r-\log r_{1}}-\bigl[\zrf_{F}(r)-\zrf_{F}(r_{1})\bigr]\\
 & \ge-C\frac{E_{F}(r)+E_{F}(r_{1})}{\errfunc(r)}\,\bigl(\log\zrf_{F}(r)\bigr)^{b}-\errfunc(r)\zrf_{F}(r)\,,
\end{align*}
provided that $r\in[2,\infty)\setminus E$. Therefore, 
\begin{equation}
e_{F}(r)\lesssim\frac{E_{F}(r)+E_{F}(r_{1})+E_{F}(r_{2})}{\errfunc(r)}\,\bigl(\log\zrf_{F}(r)\bigr)^{b}+\errfunc(r)\zrf_{F}(r)\label{eq:s(r)}
\end{equation}
provided that $r\in[2,\infty)\setminus E$, where $E$ is a set of
finite logarithmic measure. The choice of the decreasing function
$\errfunc(r)$ is at our disposal. Now we can readily deduce estimates
for $e_{F}(r)$ in mean and a.s. from the corresponding estimates~\eqref{eq:E_F-mean}
and~\eqref{eq:E_F-a.s.} for $E_{F}(r)$.

By estimates~\eqref{eq:E_F-mean} and~\eqref{eq:E_F-a.s.},
\begin{equation}
E_{F}(r)\le A\bigl(\log\sigma_{F}(r)\bigr)^{\frac{1}{2}}\cdot\bigl(\log\log\sigma_{F}(r)\bigr)^{C}\cdot\log r\,\quad\mbox{a.s.}\label{eq:upper_bound_for_S(r)}
\end{equation}
and also in mean, provided that $r$ is sufficiently big. Since $\log\sigma_{F}(r)\le\zrf_{F}(r)\cdot\log r+A$,
and since we assume that $\log\sigma_{F}(r)\ge\log^{a}r$ with some
$a>4$, by~\eqref{eq:upper_bound_for_S(r)} we have that a.s., $E_{F}(r)\le A\zrf_{F}^{\gamma_{1}}(r)$
with $\frac{1}{2}+\frac{2}{a}<\gamma_{1}<1$. The same bound holds
for $E_{F}(r_{1})$. Since $r\notin E$, this also holds for $E_{F}(r_{2})$.
Then, making use of~\eqref{eq:s(r)}, we see that $e_{F}(r)$ does
not exceed $A\errfunc(r)^{-1}\cdot\zrf_{F}^{\gamma_{2}}(r)+\errfunc(r)\zrf_{F}(r)$,
both in mean and a.s., with the exponent $\gamma_{2}$ lying in the
same range as $\gamma_{1}$. Letting $\errfunc(r)=\left(\zrf_{F}(r)\right)^{-(1-\gamma_{2})/2}$,
we get part (ii) of Theorem~\ref{thm:Rade_ent_funcs_ang_dist_of_zero}
with $\gamma=\frac{1}{2}\left(1+\gamma_{2}\right)$. \hfill{}$\Box$

\section{Random entire functions with a fixed portion of real zeros\label{sec:Rade_Func_with_many_real_zeros}}

In this section we construct an example of a class of entire functions
with only real zeros. In particular this implies that the zeros are
not equidistributed in sectors.

The next theorem is based on an idea that goes back to Hardy \cite{Har}.
It is worth mentioning that Littlewood and Offord used a similar construction
in \cite[Theorem 3 (iii)]{LO1}.\index{Rademacher entire functions!having only real zeros}
\begin{thm}
\label{thm:real_zeros_example}Let
\[
f\left(z\right)=\sum_{k\ge0}\xi_{k}e^{-\alpha k^{2}}z^{k},
\]
be a Rademacher entire function. For $\alpha\ge\log3$ the function
$f$ has only real zeros.\end{thm}
\begin{rem*}
This result is not probabilstic, as it does not depend on the values
of the sequence $\left\{ \xi_{n}\right\} $. It is possible to show
that for every $\alpha>0$ there is, almost surely, a positive portion
(depending on $\alpha$) of the zeros in $\disk r$ that are real.
\end{rem*}
Let $\alpha=\beta^{2}/2$. We first show that $\log\sigma_{f}\left(r\right)$
is of order $\log^{2}r$. Note that
\[
\sigma_{f}^{2}\left(r\right)=\sum_{k\ge0}e^{-2\alpha k^{2}}r^{2k}=\sum_{k\ge0}e^{-2\alpha k^{2}+2k\log r},
\]
and thus
\[
\sigma_{f}^{2}\left(r\right)=\exp\left(\frac{\log^{2}r}{2\alpha}\right)\sum_{k\ge0}\exp\left(-2\alpha\left(k-\frac{\log r}{2\alpha}\right)^{2}\right).
\]
Hence for $r\ge e$, and using the fact that
\[
\sum_{k\ge1}e^{-2\alpha k^{2}}\le\sum_{k\ge0}\exp\left(-2\alpha\left(k-\frac{\log r}{2\alpha}\right)^{2}\right)\le2\cdot\sum_{k\ge0}e^{-2\alpha k^{2}},
\]
we get that
\[
C_{1,\alpha}\cdot\exp\left(\frac{\log^{2}r}{2\alpha}\right)\le\sigma_{f}^{2}\left(r\right)\le C_{2,\alpha}\cdot\exp\left(\frac{\log^{2}r}{2\alpha}\right),
\]
where $C_{1,\alpha},C_{2,\alpha}$ are certain positive constants
that depend only on $\alpha$. Therefore,
\[
\log\sigma_{f}\left(r\right)=\frac{\log^{2}r}{4\alpha}+\O_{\alpha}\left(1\right),
\]
and indeed the function $f$ lies outside of the scope of Theorem
\ref{thm:Rade_ent_funcs_ang_dist_of_zero}.

Before we prove Theorem we first need a lemma that locates the zeros
of $f$. Without changing notation, we now assume that $\left\{ \xi_{n}\right\} $
is an arbitrary sequence, which consists of the numbers $\pm1$ (there
is no probability involved).
\begin{lem}
Suppose that For $\alpha\ge\log3$. Then given any $m\in\bbn$, the
function $f$ has exactly $m$ zeros inside the disk $\left\{ \left|z\right|\le e^{2\alpha m}\right\} $.\end{lem}
\begin{proof}
Let $m\in\bbn$ and $r_{m}=e^{2\alpha m}$. Then for any $z$ on the
circle $\left\{ \left|z\right|=r_{m}\right\} $ we have
\begin{eqnarray*}
\left|f\left(z\right)-\xi_{m}e^{-\alpha m^{2}}z^{m}\right| & < & \sum_{k\ne m}e^{-\alpha k^{2}}\left|z\right|^{k}=\exp\left(\alpha m^{2}\right)\sum_{k\ne m}e^{-\alpha\left(k-m\right)^{2}}\\
 & \le & \exp\left(\alpha m^{2}\right)\cdot2\sum_{k\ge1}e^{-\alpha k^{2}}.
\end{eqnarray*}
We now observe that for $\alpha\ge\log3$ %
\footnote{With a more precise calculation this can be improved to $\alpha\ge0.7855$.%
}
\[
\sum_{k\ge1}e^{-\alpha k^{2}}\le\sum_{k\ge1}3^{-k^{2}}<\sum_{k\ge1}3^{-k}=\frac{1}{2}.
\]
Hence, we conclude that
\[
\left|f\left(z\right)-\xi_{m}e^{-\alpha m^{2}}z^{m}\right|<\left|\xi_{m}e^{-\alpha m^{2}}z^{m}\right|=\exp\left(\alpha m^{2}\right).
\]
Now Rouché\textquoteright{}s Theorem yields the required result.
\end{proof}
\textit{Proof of Theorem \ref{thm:real_zeros_example}}: By the lemma,
each annulus $\left\{ e^{2\alpha m}<\left|z\right|\le e^{2\alpha\left(m+1\right)}\right\} $,
$m\in\bbn$, contains exactly one zero of the function $f$. Since
the Taylor coefficients of the function $f$ are real, this zero must
be real as well. The same holds for the disk $\left\{ \left|z\right|\le e^{2\alpha}\right\} $.
Hence, all the zeros of $f$ are real.\hfill{}$\Box$

\section{Further Work and Open Problems\index{open problems}}

In a new work \cite{NNS2} we prove stronger versions of Theorem \ref{thm:Rade_ent_funcs_ang_dist_of_zero}.
Let $F\left(z\right)=\sum_{n\ge0}\xi_{n}a_{n}z^{n}$ be a Rademacher
entire function, which is not a polynomial and such that $a_{0}\ne0$.
Let $\varphi\in C^{2}\left(\bbr\right)$, $\varphi\left(\theta\right)\in\left[0,1\right]$
be a $2\pi$-periodic test function. Our zero counting function is
\[
n_{F}\left(r;\varphi\right)=\sum_{\zeta\in Z_{F}\cap\disk r}\varphi\left(\arg\zeta\right),\quad Z_{F}=F^{-1}\left\{ 0\right\} ,
\]
where the zero set $Z_{F}$ includes multiplicities. We mention that
$n_{F}\left(r\right)=n_{F}\left(r;\mathbf{1}\right)$. We will use
the following notation:
\[
F_{r}\left(\theta\right)=F\left(re^{i\theta}\right),\quad\widehat{F}_{r}=\frac{F_{r}}{\sigma_{F}\left(r\right)},
\]
and for any integrable $g\colon\left[0,2\pi\right]\to\bbr$
\[
\left\langle g\right\rangle =\frac{1}{2\pi}\intsy_{0}^{2\pi}g\left(t\right)\d t.
\]
Fix some parameters $\gamma_{0}\in\left(\tfrac{1}{2},1\right)$, $q_{0}>1$
and put $p_{0}=\frac{q_{0}}{q_{0}-1}$. Here $c$ and $C$ denote
various positive constants that may depend (only) on these parameters.
We denote by $A_{F}$ various positive constants that depend only
on the sequence$\left\{ \left|a_{n}\right|\right\} $ of absolute
values of the Taylor coefficients of the Rademacher entire function
$F$, and on the parameters $\gamma_{0},q_{0}$.

The first theorem is a more accurate version of Theorem \ref{thm:Rade_ent_funcs_ang_dist_of_zero},
where we study the concentration of $n_{F}\left(r;\varphi\right)$
around its mean, for a certain test function $\varphi$.
\begin{thm*}
\textup{{[}Almost sure concentration for $n_{F}\left(r;\varphi\right)${]}}
There exists a set $E=E\left(\varphi\right)\subset\left[1,\infty\right)$
of finite logarithmic measure, such that almost surely we have
\begin{align*}
\left|n_{F}\left(r;\varphi\right)-\Ex\left\{ n_{F}\left(r;\varphi\right)\right\} \right| & \le\\
 & C\left(1+\norm[\varphi^{\dprime}]1\right)\left(\Ex\left\{ n_{F}\left(r;\varphi\right)\right\} \right)^{\gamma_{0}}+A_{F}\norm[\varphi^{\dprime}]{q_{0}}\log^{\gamma_{0}}r,
\end{align*}
for $r\ge r_{0}$, and $r\notin E$.\end{thm*}
\begin{rem*}
Here $r_{0}$ is random. It is worth to mention that
\[
\Ex\left\{ n_{F}\left(r;\varphi\right)\right\} =\left\langle \varphi\right\rangle s_{F}\left(r\right)+\intsy_{0}^{r}\left\langle \varphi\cdot\Ex\log\left|\widehat{F}_{t}\right|\right\rangle \,\frac{\dd t}{t}+\frac{\dd\left\langle \varphi\cdot\Ex\log\left|\widehat{F}_{t}\right|\right\rangle }{\dd\log r}.
\]

\end{rem*}
Put
\[
\varepsilon_{F}\left(r\right)=\sup_{t\ge r}\left(\frac{\log t}{s_{F}\left(t\right)}\right).
\]
Note that $\varepsilon_{F}\left(r\right)$ does not increase when
$F$ does not grow too slow; that is, when $\log r=\O\left(s_{F}\left(r\right)\right)$
(which is equivalent to $\log^{2}r=\O\left(\log\sigma_{F}\left(r\right)\right)$).
Otherwise $\varepsilon_{F}\left(r\right)\equiv+\infty$.

The second theorem is an uniform equidistribution result.
\begin{thm*}
\textup{{[}Uniform equidistribution for $n_{F}\left(r;\varphi\right)${]}}
There exists a set $E\subset\left[1,\infty\right)$ of finite logarithmic
measure, such that almost surely we have
\begin{align*}
\left|n_{F}\left(r;\varphi\right)-\left\langle \varphi\right\rangle s_{F}\left(r\right)\right| & \le\\
 & A_{F}\left(1+\norm[\varphi^{\dprime}]{q_{0}}\right)s_{F}^{\gamma_{0}}\left(r\right)+C\norm[\varphi^{\dprime}]{q_{0}}s_{F}\left(r\right)\varepsilon_{F}\left(r\right),\tag{\ensuremath{\star}}
\end{align*}
for any test function $\varphi$, and any $r\ge r_{0}$, such that
$r\notin E$. Furthermore, for any $r\notin E$
\[
\left|\Ex n_{F}\left(r;\varphi\right)-\left\langle \varphi\right\rangle s_{F}\left(r\right)\right|\le\mbox{the RHS of }\left(\star\right).
\]
\end{thm*}
\begin{rem*}
Here $r_{0}$ is random, but it does not depend on $\varphi$. In
particular, the last theorem implies the equidistribution of the zeros
of $F$, as long as $\log r=o\left(s_{F}\left(r\right)\right)$.
\end{rem*}

\subsubsection*{Functions with `many' real zeros}

Let $f$ be a Rademacher entire function of the form 
\[
f\left(z\right)=\sum_{k\ge0}\xi_{k}e^{-\alpha k^{2}}z^{k},
\]
and let us denote by $n_{\bbr^{\mathrm{c}}}\left(r\right),\, n_{\bbr}\left(r\right)$
the number of (complex non-real\textbackslash{}real) roots of $f$
inside $\disk r$. Recall that for $\alpha\ge\log3\approxeq1.099$
we proved that $n_{\bbr^{\mathrm{c}}}\left(r\right)=0$ always, for
any $r\ge0$. Some numerical experiments suggest that already for
$\alpha=0.7853$ we have that $n_{\bbr^{\mathrm{c}}}\left(r\right)>0$
with a positive probability. It seems that there is a critical `threshold'
value, the root of the equation
\[
\sum_{k\ge1}e^{-\alpha k^{2}}=\frac{1}{2},
\]
which is approximately $0.785409$. It is not clear if there are some
values of $\alpha$ such that $n_{\bbr^{\mathrm{c}}}\left(r\right)=0$
almost surely, but some instances (of measure $0$) for which $n_{\bbr^{\mathrm{c}}}\left(r\right)>0$.
Another possible situation is that there are values of $\alpha$ for
which the ratio $n_{\bbr^{\mathrm{c}}}\left(r\right)/n_{\bbr}\left(r\right)$
almost surely tends to $0$ as $r$ tends to infinity (that is, there
are `few' non-real zeros).

\subsubsection*{Distribution of zeros of random entire functions}

There are many possible generalization of Theorem \ref{thm:Rade_ent_funcs_ang_dist_of_zero},
and related questions. A natural problem is to generalize the results
to more general i.i.d. random variables, while trying to keep the
assumptions on the sequence of coefficients $\left\{ a_{k}\right\} $
minimal. A result for not necessarily independent coefficients also
seems very interesting.

Other related problems are to find bounds for the variance of the
number of zeros (at least for regular sequences $\left\{ a_{k}\right\} $,
or upper bounds for general coefficients), or to find accurate bounds
for the number of real zeros (in terms of $s_{F}\left(r\right)$).
Both problems are interesting even for (real) Gaussian coefficients.

\newpage{}

\part{The hole probability for Gaussian entire functions}

We begin this part with the background of the problem and a (short)
history of previous results. We mention relations to questions in
probability theory and in mathematical physics. We then give a proof
of the main result and discuss, briefly, its sharpness. We then give
an example that illustrates the significance of Gaussianity for the
result. We finish with a list of related open problems.

\section{Background and History}

Let us consider first the following special case of a Gaussian entire
function\index{Gaussian Entire Function, the|see{GEF}} \index{GEF}

\begin{equation}
F(z)=\sum_{n=0}^{\infty}\xi_{n}\frac{z^{n}}{\sqrt{n!}},\label{eq:GEF_def}
\end{equation}
which is sometimes referred to as the Gaussian Entire Function (or
GEF), for reasons that will soon become clear. The random zero set
of this function is known to be distribution invariant with respect
to isometries of the plane. Furthermore, this is the only Gaussian
entire function with distribution invariant zeros (see the book \cite[Section 2.5]{BKPV}
for details). It is also known that the GEF and GAFs in general have
simple zeros almost surely (unless the zero is non-random). Thus one
can consider the zero set of a GAF as a simple point process in the
plane, and compare it to other point processes.

We recall that if $\nu_{F}$ is the random zero counting measure of
$F$, then by the Edelman-Kostlan formula (\ref{eq:Edel-Kos_for})
we have \index{zero counting measure} \index{Edelman-Kostlan formula}
\index{GEF!statistics for the zeros}
\[
\Ex\left\{ \nu_{F}\left(z\right)\right\} =\frac{1}{2\pi}\Delta\log\sqrt{K_{F}\left(z,\overline{z}\right)}=\frac{1}{4\pi}\Delta\log\Ex\left\{ F\left(z\right)\overline{F\left(\overline{z}\right)}\right\} =\frac{\mathrm{d}m\left(z\right)}{\pi}
\]
in the sense of distributions, where $m$ is the Lebesgue measure
in the complex plane. We conclude that the GEF is similar to the Poisson
point process in the plane, in the sense that the expected number
of points in a region is proportional to the area of the region. We
mention that for the GEF there are many results concerning the linear
statistics of the zeros, such as the variance, aysmptotic normality
and deviations from normality (see \cite{NS1,NSV1,ST1} for explanations
and proofs). There are also various geometric properties that are
known about the zero set of the GEF (see for example \cite{NSV2}).

\subsection{Hole Probability}

We now show a particular way in which the Poisson point process and
the GEF zero set differ in a rather dramatic way. One of the interesting
characteristics of a random point process is the probability that
there are no points in some (say, simply connected) domain $D\subset\bbc$.
In the case of the Poisson point process, we know exactly what is
the distribution of the number of points, that is, the Poisson distribution
with parameter $m\left(D\right)$. We immediately conclude that the
probablity that there are no points in $D$ is $\exp\left(-m\left(D\right)\right)$.
To simplify matters, let us now consider the probability of the event
where there are no points of the process inside the disk $\disk r=\left\{ \left|z\right|\le r\right\} $,
as $r$ tends to infinity. We will call this event the `hole' event.
Before returning to the GEF, let us give some other examples.\index{hole probability}
\index{hole probability!examples}

\subsubsection{The Poisson point process}

As we have already mentioned, this is the simplest process to analyze.
The probability of the hole event is $\exp\left(-cr^{2}\right)$,
where the constant $c$ depends on the normalization of the processes
(i.e., the number of expected points of the process per unit area).
We do not expect that this process will be a good model for the zeros
of the GEF, as the points of this process tend to `clump' together,
while the zeros of the GEF have a more rigid (lattice like) structure.

\subsubsection{A (Gaussian) perturbed lattice}

Here we consider the perturbed lattice: 
\[
\left\{ \frac{\left(m,n\right)}{\sqrt{\pi}}+\xi_{m,n}\,\mid\, n,m\in\bbz\right\} ,
\]
where $\xi_{m,n}$ are i.i.d. standard complex Gaussians. The hole
probability is expected to be of order $\exp\left(-cr^{4}\right)$.
Heuristically, the reason is that in order to have a hole around $0$,
we have to ``push'' the points which are at distance $s$ from $0$,
a distance of at least $r-s$. The probability of the last event is
of order $\exp\left(-\left(r-s\right)^{2}\right)$. Since there are
approximately $2s$ points at distance $r-s$ from $0$, the hole
probability is approximately of order
\[
\exp\left(-\intsy_{0}^{r}\left(r-s\right)^{2}\cdot2s\, ds\right)=\exp\left(-\frac{r^{4}}{6}\right).
\]
We will see that this probability is much closer to the hole probability
of the GEF.

\subsubsection{Infinite Ginibre ensemble}

Let $M$ be an $N\times N$ matrix with entries which are i.i.d standard
complex Gaussians, and let $\lambda_{n}$ be its eigenvalues. Kostlan
showed that
\[
\left\{ \left|\lambda_{1}\right|,\ldots,\left|\lambda_{n}\right|\right\} \overset{d}{=}\left\{ Z_{1},\ldots,Z_{n}\right\} ,
\]
where the random variables $Z_{k}$ are independent and $Z_{k}^{2}$
has the $\Gamma\left(k,1\right)$ distribution. The infinite Ginibre
ensemble is the `limit' of this process (see \cite[p. 69]{BKPV}).
In this book there is a proof that the hole probability is of order
$\exp\left(-\frac{r^{4}}{4}\right)$. This model is a special case
of the one-component plasma model.

\subsubsection{One-component plasma}

This is a simple physical model, where particles with positive charge
are embedded in a background of negative charge. The simplest case
is the two-dimensional one, studied non-rigorously in \cite{JLM}.
They found the hole probability to be of order $\exp\left(-cr^{4}\right)$.
This model has natural generalizations to higher dimensions, which
are interesting to study.

\subsubsection{The GEF - The plane invariant GAF\index{GEF!hole probablity}}

We now consider the probability of the event where $F\left(z\right)\ne0$
in the disk $\disk r=\left\{ \left|z\right|\le r\right\} $, as $r$
tends to infinity, and denote it by $P_{H}\left(F;r\right)$. Since
the decay rate of this probability is known to be exponential in $r$,
we will use the notation\index{1pH@$p_{H}(F;r)$}
\[
p_{H}(F;r)=\log^{-}P_{H}(F;r)=\log^{-}\Pr\left(F(z)\ne0\mbox{ for }|z|\le r\right).
\]
In the paper \cite{ST3}, Sodin and Tsirelson showed that for $r\ge1$,
\begin{equation}
c_{1}r^{4}\le p_{H}(F;r)\le c_{2}r^{4}\label{eq:hole_ST_asymp}
\end{equation}
with some (unspecified) positive numerical constants $c_{1}$ and
$c_{2}$. This result was extended in different directions by Ben
Hough \cite{BH1}, Krishnapur \cite{Kri}, Zrebiec \cite{Zr1,Zr2}
and Shiffman, Zelditch and Zrebiec \cite{SZZ}.

In \cite{ST3}, Sodin and Tsirelson asked whether the limit
\[
\lim_{r\to\infty}\frac{p_{H}(F;r)}{r^{4}}
\]
exists and what is its value? We found an answer to this question
in the paper \cite{Ni1}.
\begin{thm}
For $r$ large enough\label{thm:Hole_Prob_spec}
\begin{equation}
p_{H}(F;r)=\frac{e^{2}}{4}\cdot r^{4}+\O\left(r^{18/5}\right).\label{eq:accurate_hole_asymp}
\end{equation}

\end{thm}
We mention that before this result the explicit constant for a GAF
was known only for the very special case $\sum_{n=0}^{\infty}\xi_{n}z^{n}$
(see \cite{PV}), using completely different methods.

\subsubsection{Gaussian entire functions}

We extended Theorem \ref{thm:Hole_Prob_spec} to a large class of
GAFs, whose sequence of coefficients $\left\{ a_{n}\right\} $ satisfy
some regularity condition (see \cite{Ni2}). In this part of the work
we prove Theorem \ref{thm:hole_prob_GAFs}, which generalizes the
latter result to arbitrary GAFs given in a Taylor series form. 

To recall the statement of the theorem, let\index{Gaussian entire functions}
\[
f\left(z\right)=\sum_{n\ge0}\xi_{n}a_{n}z^{n}
\]
be a GAF such that $\sum_{n\ge0}a_{n}z^{n}$ $f$ is a non-constant
entire functions and $a_{0}\ne0$ (in order to avoid trivial situations).
We write
\[
S\left(r\right)=2\cdot\sum_{n\ge0}\log^{+}\left(a_{n}r^{n}\right).
\]
A set $\excpset\subset\left[1,\infty\right)$ is of \textit{finite
logarithmic measure }if $\intsy_{\excpset}\,\frac{\mathrm{d}t}{t}<\infty$\index{finite logarithmic measure}.
Theorem \ref{thm:hole_prob_GAFs} states that for $\epsym\in\left(0,\frac{1}{2}\right)$
there exists an exceptional set $\excpset\subset\left[1,\infty\right)$,
of finite logarithmic measure, which depends only on $\epsym$ and
on the sequence $\left\{ a_{n}\right\} $ such that as $r\to\infty$
outside of the set $\excpset$,
\[
p_{H}(f;r)=S(r)+\O\left(S\left(r\right)^{1/2+\epsym}\right).
\]
We remark that the exceptional set $\excpset$ does not appear if
the coefficients of $f$ satisfy some regularity conditions, but in
general it is unavoidable (see Section \ref{sec:hole_prob_nec_exc_set}).
It turns out that Gaussianity is very important for Theorem \ref{thm:hole_prob_GAFs}
to hold. For the reader's convenience we reproduce in Section \ref{sec:No_Hole_for_Non_Gaus}
the proof of the next theorem (from \cite{Ni1}).\index{hole probability!examples}
\begin{thm}
\label{thm:No_Hole_Prob_for_Non_Gaus}Let $K\subset\bbc$ be a compact
set and $0\notin K$. Let $\left\{ \phi_{n}\right\} $ be any sequence,
such that $\mathbf{\phi}_{n}\in K$ for each $n$, and put
\begin{equation}
g\left(z\right)=\sum_{n\ge0}\phi_{n}\frac{z^{n}}{\sqrt{n!}}.\label{eq:non_Gaus_f_def}
\end{equation}
Then there exists $r_{0}=r_{0}(K)<\infty$ so that $g(z)$ must vanish
somewhere in the disk $\left\{ |z|\le r_{0}\right\} $.\end{thm}

\section{Proof of Theorem \ref{thm:hole_prob_GAFs} - Preliminaries \label{sec:Main-Thm-Prelim}}

In this section we introduce some special notation for the proof.
We also prove some estimates which we use later to control some error
terms. In addition, we give a short description of the strategy of
the proof.

\subsection{Notation}

For $r\ge1$, we denote by $\disk r$ the disk $\left\{ z\,\mid\,|z|\le r\right\} $
and by $\cir r$ its boundary $\left\{ z\,\mid\,|z|=r\right\} $.
The letters $c$ and $C$ denote positive absolute constants (which
can change from line to line). We also use the standard notation
\[
M\left(r\right)=\max_{z\in\disk r}\left|f\left(z\right)\right|.
\]

\subsection{A growth lemma by Hayman}

We recall that a set $E\subset\left[1,\infty\right)$ is of \textit{finite}
\textsl{logarithmic measure} if
\[
\intsy_{E}\,\frac{dt}{t}<\infty.
\]
The following lemma (taken from \cite{Ha2}) is a general result about
the growth of functions.
\begin{lem}
\label{lem:H_Growth_Lemma}Let $\eta>0$ and suppose that $D(r)$
is a positive increasing function of $r$ for $r\ge r_{0}.$ For all
$r$ outside a set of finite \textsl{logarithmic measure} and for
any $\delta$ such that $|\delta|<D(r)^{-\eta}$, we have
\[
\left|D\left(re^{\delta}\right)-D(r)\right|<\eta D(r).
\]

\end{lem}

\subsection{Further notation, definition of the exceptional set}

We use the following notations:
\[
b_{n}\left(r\right)=\begin{cases}
\frac{1}{n}\log a_{n}+\log r,\, & \mbox{if }a_{n}>0,\\
-\infty, & \mbox{if }a_{n}=0,
\end{cases}
\]
and 
\begin{eqnarray*}
N\left(r\right) & = & \left\{ n\,\mid\, b_{n}\left(r\right)\ge0\right\} ,\\
S(r) & = & 2\cdot\sum_{n\in N\left(r\right)}\log\left(a_{n}r^{n}\right)=2\cdot{\displaystyle \sum_{n\in N\left(r\right)}}nb_{n}\left(r\right).
\end{eqnarray*}
Since the coefficients $a_{n}$ satisfy $\frac{\log a_{n}}{n}\to-\infty$,
we have that $b_{n}\to-\infty$ as $n\to\infty$ and the set $N(r)$
is finite for every $r\ge1$.

We now write 
\begin{eqnarray*}
n\left(r\right) & = & \#N(r),\\
m\left(r\right) & = & 4\cdot\sum_{n\in N\left(r\right)}n,\\
N_{\delta}\left(r\right) & = & \left\{ n\,\mid\, b_{n}\left(r\right)\ge-\delta\right\} ,\\
n_{\delta}\left(r\right) & = & \#N_{\delta}\left(r\right).
\end{eqnarray*}
Note that $b_{n}\left(r\right)$ is increasing with $r$ and therefore
$n\left(r\right)$ and $m\left(r\right)$ are increasing functions
of $r$. Moreover, 
\begin{eqnarray}
N_{-\delta}\left(r\right) & \subset & N_{0}\left(r\right)=N\left(r\right),\label{eq:N_delta_to_N_r}\\
N_{-\delta}\left(r\right) & = & N\left(re^{-\delta}\right).\nonumber 
\end{eqnarray}
Let $\eta\in\left(0,\frac{1}{4}\right]$. We apply Lemma \ref{lem:H_Growth_Lemma}
to the function $m\left(r\right)$, taking 
\[
\delta=\delta\left(r\right)=m^{-\eta}\left(r\right).
\]
We conclude that outside an exceptional set $\excpset$ of finite
logarithmic measure, we have
\begin{eqnarray}
m\left(re^{-\delta}\right) & > & \left(1-\eta\right)m\left(r\right),\label{eq:m(r)_normal}\\
m\left(re^{\delta}\right) & < & \left(1+\eta\right)m\left(r\right).\nonumber 
\end{eqnarray}

From now on we will fix $r\notin\excpset$. We also assume that $r$
is large enough, depending on the coefficients $a_{n}$ and on the
value of $\eta$. In particular, we assume that $n\left(r\right)\ge2$. 

We also note that the choice of $\delta$ remains the same throughout
the paper. The limit in Theorem \ref{thm:hole_prob_GAFs} is taken
over $\excpset^{\mathrm{c}}$. If, for some set of coefficients $a_{n}$,
the inequalities (\ref{eq:m(r)_normal}) hold for all large values
of $r$, then there is no exceptional set in Theorem \ref{thm:hole_prob_GAFs}.

\subsection{Estimates for $S(r)$}

Here we find relations between $S\left(r\right)$ and the functions
$m\left(r\right),\, n\left(r\right)$, that will be used later in
the proof.
\begin{lem}
\label{lem:S(r)_low_bnd}We have
\[
S(r)\ge\frac{1}{8}\cdot\left(m\left(r\right)\right)^{1-\eta}\ge\frac{1}{8}\cdot n\left(r\right)^{2-2\eta}.
\]
\end{lem}
\begin{proof}
Recall that it is assumed that $n\left(r\right)\ge2$. Notice that
\begin{equation}
m\left(r\right)=4\cdot\sum_{n\in N\left(r\right)}n\ge n^{2}\left(r\right),\label{eq:m(r)_low_bnd}
\end{equation}
since $m\left(r\right)$ is minimal when $N\left(r\right)=\left\{ 0,1,\ldots,n\left(r\right)-1\right\} $.
Now, using the inclusion $N_{-\delta}\left(r\right)\subset N\left(r\right)$$ $,
we get:
\begin{eqnarray*}
\frac{S\left(r\right)}{2} & = & \sum_{n\in N\left(r\right)}nb_{n}\left(r\right)\ge\sum_{n\in N_{-\delta}\left(r\right)}nb_{n}\left(r\right)\\
 & \ge & \sum_{n\in N\left(re^{-\delta}\right)}n\delta\ge\frac{\delta}{4}\cdot m\left(re^{-\delta}\right)\\
 & \ge & \frac{\left(1-\eta\right)}{4}\cdot\frac{m\left(r\right)}{\left(m\left(r\right)\right)^{\eta}}\ge\frac{1}{8}\cdot n\left(r\right)^{2-2\eta}.
\end{eqnarray*}

\end{proof}
We now estimate the rate of growth of the function $S\left(r\right)$.
\begin{lem}
\label{lem:S(r)_growth}For $\gamma\in\left(0,\frac{1}{2}\right)$
we have,
\[
S\left(\left(1-\gamma\right)r\right)\ge S\left(r\right)-\gamma\cdot m\left(r\right).
\]
\end{lem}
\begin{proof}
Write $r^{\prime}=\left(1-\gamma\right)r$ and notice that for $\gamma<\frac{1}{2}$
we have
\[
\log\left(1-\gamma\right)\ge-\gamma-\gamma^{2}\ge-2\gamma.
\]
We will also use the inequality
\[
\log^{+}a-\log^{+}b\le\log\frac{a}{b},\quad a>b>0.
\]
It follows that (since $N(r^{\prime})\subset N(r)$)
\begin{eqnarray*}
S(r)-S\left(r^{\prime}\right) & = & 2\cdot\sum_{n\in N(r)}\left\{ \log^{+}\left(a_{n}r^{n}\right)-\log^{+}\left[a_{n}\left(r^{\prime}\right)^{n}\right]\right\} \\
 & \le & {\displaystyle 2\cdot\sum_{n\in N(r)}}n\cdot\log\frac{r}{r^{\prime}}=\log\left(\frac{1}{1-\gamma}\right)\cdot\frac{m\left(r\right)}{2}\\
 & \le & \gamma\cdot m\left(r\right).
\end{eqnarray*}

\end{proof}
In the next lemma we show that the value of $S\left(r\right)$ does
not depend on the scaling of the coefficients (up to some error term).
\begin{lem}
\label{lem:S(r)_scaling}Let $d>0$ and set $\Salt r=2\cdot\sum_{n\ge0}\log^{+}\left(d\cdot a_{n}r^{n}\right)$.
We have
\[
\left|S\left(r\right)-\Salt r\right|\le C\sqrt{m\left(r\right)},
\]
where $C$ is a positive constant that might depend on $d$.\end{lem}
\begin{proof}
If $d<1$, then $\Salt r$ is at least $S\left(r\right)-\log\frac{1}{d}\cdot n\left(r\right)\ge S\left(r\right)-\log\frac{1}{d}\sqrt{m\left(r\right)}$.
Now assume that $d>1$. Denote by $\widetilde{b}_{n}\left(r\right),\Nalt r,\widetilde{n}\left(r\right)$
the corresponding functions for the set of coefficients $\left\{ d\cdot a_{n}\right\} _{n\ge0}$.
By Lemma \ref{lem:S(r)_low_bnd}, we have $n\left(re^{\delta}\right)\le C\sqrt{m\left(r\right)}$,
so the contribution from terms in $N\left(re^{\delta}\right)$ to
$\Salt r$ is at most
\[
C\log d\cdot\sqrt{m\left(r\right)}.
\]
If $n\notin N\left(re^{\delta}\right)$ and $n\in\Nalt r$, then
\[
-\delta+\frac{\log d}{n}\ge b_{n}\left(r\right)+\frac{\log d}{n}=\frac{\log\left(da_{n}\right)}{n}+\log r=\widetilde{b}_{n}\left(r\right)\ge0,
\]
so we have $n\le\frac{\log d}{\delta}\le\log d\cdot m^{\eta}\left(r\right)$.
So the contribution from these terms is at most $\log^{2}d\cdot m^{\eta}\left(r\right)\le C\sqrt{m\left(r\right)}$.
\end{proof}

\subsection{Gaussian distributions\label{sub:Gauss-props}}

Many times we use the fact that if the random variable $a$ has standard
complex Gaussian distribution, then 
\begin{equation}
\pr{|a|\ge\lambda}=\exp(-\lambda^{2}),\label{eq:Gaus_prob_large}
\end{equation}
and for $\lambda\le1,$

\begin{equation}
\pr{|a|\le\lambda}\in\left[\frac{\lambda^{2}}{2},\lambda^{2}\right].\label{eq:Gaus_prob_small}
\end{equation}

\subsection{Strategy of the proof\label{sub:proof_strategy}}

The proof consists of two parts. We first show that
\[
p_{H}\left(r\right)\le S\left(r\right)+C\sqrt{m\left(r\right)}\log m\left(r\right)
\]
(Proposition \ref{prp:p_H_upp_bnd}), and then we prove that
\[
p_{H}\left(r\right)\ge S\left(r\right)-Cn\left(r\right)\log S\left(r\right)
\]
(Proposition \ref{prp:p_H_low_bnd}). Combining these bounds with
Lemma \ref{lem:S(r)_low_bnd}, we get Theorem \ref{thm:hole_prob_GAFs}.
If the coefficients $a_{n}$ are such that for $r\ge1$ there is no
exceptional set $\excpset$, then Theorem \ref{thm:hole_prob_GAFs}
holds for every $r$ large enough. In some cases, where the coefficients
$a_{n}$ have regular asymptotic behavior, it is possible to prove
that $m\left(r\right)\le CS\left(r\right)$ and obtain
\[
p_{H}\left(r\right)=S\left(r\right)+\O\left(\sqrt{S\left(r\right)}\log S\left(r\right)\right),\quad r\to\infty.
\]
For instance, such is the case for the coefficients of the GEF, where
$a_{n}=\frac{1}{\sqrt{n!}}$, or more generally where $a_{n}\sim\frac{1}{\Gamma\left(\alpha n+1\right)}$,
with some $\alpha>0$.
\begin{rem*}
By Lemma \ref{lem:S(r)_scaling}, we can scale $f\left(z\right)$
by a constant factor that will add a term of order $\sqrt{m\left(r\right)}$
to $S\left(r\right)$. By Lemma \ref{lem:S(r)_low_bnd}, this term
is of order at most $S\left(r\right)^{1/2+\epsym}$ and can be ignored.
From now on, in order to simplify some of the expressions in the paper,
we will assume that 
\[
a_{0}=1.
\]
\end{rem*}

\section{Proof of Theorem \ref{thm:hole_prob_GAFs} - Upper Bound for $p_{H}(f;r)$
\label{sec:Main-Thm-Upper-Bound}}

In this section we will prove the following
\begin{prop}
\label{prp:p_H_upp_bnd}We have 
\[
p_{H}(f;r)\le S(r)+C\cdot\sqrt{m\left(r\right)}\log m(r),
\]
with some positive absolute constant $C$.\end{prop}
\begin{rem*}
Recall that $r$ is assumed to be large.
\end{rem*}
The simplest case where $f(z)$ has no zeros inside $\disk r$ is
when the constant term dominates all the others. Consider the event
$\Omega_{r}$, that is the intersection of the events $\mbox{{\rm \mbox{(i)}}}$,$\mbox{(ii)}$,
and $\mbox{(iii)}$ ($C$ will be selected in an appropriate way)
\[
\begin{array}{ll}
\mbox{{\rm \mbox{(i)}}}: & |\xi_{0}|\ge C\left(m\left(r\right)\right)^{1/4},\\
\mbox{{\rm \mbox{(ii)}}}: & {\displaystyle \bigcap_{n\in N\left(r\right)\backslash\left\{ 0\right\} }}\mbox{{\rm \mbox{(ii)}}}_{n},\\
\mbox{{\rm \mbox{(iii)}}}: & {\displaystyle \bigcap_{n\in\Naltd r\backslash N\left(r\right)}}\mbox{{\rm \mbox{(iii)}}}_{n},\\
\mbox{{\rm \mbox{(iv)}}}: & {\displaystyle \bigcap_{n\in\left(\Naltd r\right)^{c}}}\mbox{{\rm \mbox{(iv)}}}_{n},
\end{array}
\]
where $\Naltd r=N_{\delta}\left(r\right)\cup\left\{ n\,\mid\, n<\sqrt{m\left(r\right)}\right\} $
and
\[
\begin{array}{ll}
\mbox{{\rm \mbox{(ii)}}}_{n}: & |\xi_{n}|\le\frac{(a_{n}r^{n})^{-1}}{\sqrt{m\left(r\right)}},\\
\mbox{{\rm \mbox{(iii)}}}_{n}: & |\xi_{n}|\le\frac{1}{\sqrt{m\left(r\right)}},\\
\mbox{{\rm \mbox{(iv)}}}_{n}: & |\xi_{n}|\le\exp\left(\frac{\delta n}{2}\right).
\end{array}
\]
We notice that by (\ref{eq:m(r)_normal}) and (\ref{eq:m(r)_low_bnd})
we have that $\#\Naltd r\le2\sqrt{m\left(r\right)}$.
\begin{lem}
\label{lem:Low-Bnd-Est-For-Func}If $\Omega_{r}$ holds, then $f$
has no zeros inside $\disk r$.\end{lem}
\begin{proof}
Recall that $\delta=m^{-\eta}\left(r\right)$. To see that $f(z)$
has no zeros inside $\disk r$ we note that 
\begin{equation}
|f(z)|\ge|\xi_{0}|-\sum_{n=1}^{\infty}|\xi_{n}|a_{n}r^{n}.\label{eq:f_low_bnd}
\end{equation}
First, we estimate the sum over the indices in $N(r)\backslash\left\{ 0\right\} $:
\[
\sum_{n\in N\left(r\right)\backslash\left\{ 0\right\} }|\xi_{n}|a_{n}r^{n}\le\sum_{n\in N\left(r\right)}\frac{1}{\sqrt{m\left(r\right)}}\le C_{1},
\]
by (\ref{eq:m(r)_low_bnd}). Second, we estimate the sum over the
indices in $\widetilde{N}_{\delta}\left(r\right)\backslash N\left(r\right)$.
Notice that here $b_{n}\left(r\right)\le0$ and $a_{n}r^{n}=e^{\log\left(a_{n}r^{n}\right)}=e^{nb_{n}}$,
so
\[
\sum_{n\in\widetilde{N}_{\delta}\left(r\right)\backslash N\left(r\right)}|\xi_{n}|a_{n}r^{n}=\sum_{n\in\widetilde{N}_{\delta}\left(r\right)\backslash N\left(r\right)}|\xi_{n}|e^{nb_{n}}\le\sum_{n\in\widetilde{N}_{\delta}\left(r\right)\backslash N\left(r\right)}\frac{1}{\sqrt{m\left(r\right)}}\le C_{2}.
\]
Now the rest of the tail is bounded by
\begin{eqnarray*}
\sum_{n\in\left(\widetilde{N}_{\delta}\left(r\right)\right)^{\mathrm{c}}}|\xi_{n}|a_{n}r^{n} & = & \sum_{n\in\left(\widetilde{N}_{\delta}\left(r\right)\right)^{\mathrm{c}}}|\xi_{n}|e^{nb_{n}}\le\sum_{n\in\left(\widetilde{N}_{\delta}\left(r\right)\right)^{\mathrm{c}}}e^{\frac{\delta n}{2}}\cdot e^{-\delta n}\\
 & \le & \sum_{n\ge0}\exp\left(-\frac{\delta n}{2}\right)\le\frac{3}{\delta}\le3\cdot\left(m\left(r\right)\right)^{\eta}.
\end{eqnarray*}
Here we used the fact that $b_{n}\left(r\right)\le-\delta$. Hence,
(\ref{eq:f_low_bnd}) yields
\[
|f(z)|>C\left(m\left(r\right)\right)^{1/4}-C_{1}-C_{2}-3\left(m\left(r\right)\right)^{\eta}>0,
\]
provided that $C>C_{1}+C_{2}+3$.\end{proof}
\begin{lem}
\label{lem:Prob-of-Low-Bnd-Event}The probability of the event $\Omega_{r}$
is bounded from below as follows:\textup{
\[
\log\pr{\Omega_{r}}\ge-S(r)-C\cdot\sqrt{m\left(r\right)}\log m(r).
\]
}\end{lem}
\begin{proof}
By the properties of $\xi_{n}$ (see Section \ref{sub:Gauss-props}),
\[
\pr{\mbox{{\rm \mbox{(i)}}}}=\exp\left(-C^{2}\cdot\sqrt{m\left(r\right)}\right)
\]
and
\[
\pr{\mbox{{\rm \mbox{(ii)}}}_{n}}\ge\frac{(a_{n}r^{n})^{-2}}{2m\left(r\right)}.
\]
Therefore,
\begin{eqnarray*}
\pr{\mbox{(ii)}} & \ge & \prod_{n\in N\left(r\right)}\frac{(a_{n}r^{n})^{-2}}{2m\left(r\right)}=\left(\prod_{n\in N\left(r\right)}e^{-2\cdot nb_{n}}\right)\cdot\exp\left(-n\left(r\right)\log\left(2m\left(r\right)\right)\right)\\
 & \ge & \exp\left(-2\cdot\sum_{n\in N\left(r\right)}nb_{n}\right)\cdot\exp\left(-Cn\left(r\right)\log m\left(r\right)\right)\\
 & \ge & \exp\left(-S\left(r\right)-C\sqrt{m\left(r\right)}\log m\left(r\right)\right).
\end{eqnarray*}
Similarly, we have
\[
\pr{\mbox{{\rm \mbox{(iii)}}}_{n}}\ge\frac{1}{2m\left(r\right)}
\]
and so (by (\ref{eq:m(r)_normal}) and (\ref{eq:m(r)_low_bnd}))
\[
\pr{\mbox{{\rm \mbox{(iii)}}}}\ge\exp\left(-C\sqrt{m\left(r\right)}\log m\left(r\right)\right).
\]
Finally,

\[
\pr{\mbox{{\rm \mbox{(iv)}}}_{n}^{c}}=\exp\left(-e^{\delta n}\right).
\]
Let $\left\{ A_{n}\right\} $ be some positive sequence. Using the
inequality
\[
\pr{\forall n\,:\,|\xi_{n}|\le A_{n}}=1-\pr{\exists n\,:\,|\xi_{n}|>A_{n}}\ge1-\sum\pr{|\xi_{n}|>A_{n}},
\]
we obtain
\begin{eqnarray*}
1-\pr{\mbox{{\rm \mbox{(iv)}}}} & \le & \sum_{n\in\left(\widetilde{N}_{\delta}\left(r\right)\right)^{c}}\exp\left(-e^{\delta n}\right)\\
 & \le & \sum_{n\ge\sqrt{m\left(r\right)}}\exp\left(-e^{\delta n}\right)\\
 & \le & \sum_{n\ge\sqrt{m\left(r\right)}}\exp\left(-\delta n\right)=\frac{\exp\left(-\delta\sqrt{m\left(r\right)}\right)}{1-e^{-\delta}}.
\end{eqnarray*}
Since $\sqrt{m\left(r\right)}\delta\ge\frac{1}{\delta}$, we get
\[
\frac{\exp\left(-\delta\sqrt{m\left(r\right)}\right)}{1-e^{-\delta}}\le\frac{\exp\left(-\frac{1}{\delta}\right)}{1-e^{-\delta}}\le\frac{3}{4},
\]
and so
\[
\pr{\mbox{{\rm \mbox{(iv)}}}}\ge\frac{1}{4}.
\]
Since the events ${\rm \mbox{(i)}}-{\rm \mbox{(iv)}}$ are independent,
we have
\begin{eqnarray*}
\pr{\Omega_{r}} & = & \pr{\mbox{{\rm \mbox{(i)}}}}\cdot\pr{\mbox{{\rm \mbox{(ii)}}}}\cdot\pr{\mbox{{\rm \mbox{(iii)}}}}\cdot\pr{\mbox{{\rm \mbox{(iv)}}}}\\
 & \ge & \exp\left(-S\left(r\right)-C\sqrt{m\left(r\right)}\log m\left(r\right)\right).
\end{eqnarray*}

\end{proof}
Proposition \ref{prp:p_H_upp_bnd} now follows from the previous lemmas.

\section{Proof of Theorem \ref{thm:hole_prob_GAFs} - Bounds for Gaussian
Entire Functions \label{sec:bnds_for_GAFs}}

In this section we get some bounds for the moduli and the logarithmic
derivatives of Gaussian entire functions, which hold with high probability
(We use the term `high probability' for events which occur with probability
greater than $1-\exp\left(-2\cdot S\left(r\right)\right)$). These
results will be used in the next section in the proof of the lower
bound.

\subsection{Bounds on the modulus of Gaussian entire functions\index{Gaussian entire functions!Bounds on the modulus}}

We first bound the probability of the events where $M\left(r\right)$
is relatively large or small.
\begin{lem}
\textup{\label{lem:upp_bnd_M(r)}}We have\textup{ 
\[
\pr{M\left(r\right)\ge e^{3S\left(r\right)}}\le C\cdot\exp\left(-\exp\left(S\left(r\right)\right)\right)\le e^{-S^{2}\left(r\right)}.
\]
}\end{lem}
\begin{proof}
We set $\Nalt r=N_{\delta}\left(r\right)\cup\left\{ n<S^{2}\left(r\right)\right\} $.
Notice that, by Lemma \ref{lem:S(r)_low_bnd}, we can assume that
$\#\Nalt r\le2S^{2}\left(r\right)$ if $r$ is large enough. The proof
is similar to that of Proposition \ref{prp:p_H_upp_bnd}. We define
the event $\Omega_{r}$ as the intersection of the events $\mbox{{\rm \mbox{(i)}}}$
and $\mbox{(ii)}$, where
\[
\begin{array}{ll}
\mbox{{\rm \mbox{(i)}}}: & {\displaystyle \bigcap_{n\in\Nalt r}}\mbox{{\rm \mbox{(i)}}}_{n},\\
\mbox{{\rm \mbox{(ii)}}}: & {\displaystyle \bigcap_{n\in\left(\Nalt r\right)^{\mathrm{c}}}}\mbox{{\rm \mbox{(ii)}}}_{n},
\end{array}
\]
and
\[
\begin{array}{ll}
\mbox{{\rm \mbox{(i)}}}_{n}: & |\xi_{n}|\le\left(a_{n}r^{n}\right)^{-1}e^{2S\left(r\right)},\\
\mbox{{\rm \mbox{(ii)}}}_{n}: & |\xi_{n}|\le\exp\left(\frac{1}{2}\delta n\right).
\end{array}
\]
We have the following estimate for $M\left(r\right)$:
\begin{eqnarray*}
|f(z)| & \le & \sum_{n\in\Nalt r}|\xi_{n}|a_{n}r^{n}+\sum_{n\in\left(\widetilde{N}\left(r\right)\right)^{\mathrm{c}}}|\xi_{n}|a_{n}r^{n}\\
 & \le & \#\Nalt r\cdot e^{2S\left(r\right)}+\sum_{n\ge S^{2}\left(r\right)}e^{\frac{\delta n}{2}}\cdot e^{-\delta n}\\
 & \le & 2S^{2}\left(r\right)\cdot e^{2S\left(r\right)}+\frac{C}{\delta}\cdot e^{-\frac{\delta S^{2}\left(r\right)}{2}}\\
 & \le & e^{3S\left(r\right)},
\end{eqnarray*}
provided that $r$ is sufficiently large. Here we used $\delta S^{2}\left(r\right)\ge c\delta\left(\delta m\left(r\right)\right)^{2}\ge\frac{c}{\delta}.$

Now we estimate the probability of the complement of $\Omega_{r}$.
We have:
\begin{eqnarray*}
\pr{|\xi_{n}|\ge\frac{e^{2S\left(r\right)}}{a_{n}r^{n}}} & = & \exp\left(-\frac{e^{4S\left(r\right)}}{\left(a_{n}r^{n}\right)^{2}}\right)\le\exp\left(-e^{2S\left(r\right)}\right),\\
\pr{|\xi_{n}|\ge e^{\frac{\delta n}{2}}} & = & \exp\left(-\exp\left(\delta n\right)\right).
\end{eqnarray*}
By the union bound,
\begin{eqnarray*}
\pr{\mbox{{\rm \mbox{(i)}}}^{\mathrm{c}}} & \le & CS^{2}\left(r\right)\cdot\exp\left(-\exp\left(2S\left(r\right)\right)\right),
\end{eqnarray*}
and
\[
\pr{\mbox{{\rm \mbox{(ii)}}}^{\mathrm{c}}}\le\sum_{n\ge S^{2}\left(r\right)}\exp\left(-\exp\left(\delta n\right)\right)\le C\cdot\exp\left(-\exp\left(S\left(r\right)\right)\right),
\]
since the first term in the sum above dominates the others. Here we
used
\[
\delta S^{2}\left(r\right)=S^{2}\left(r\right)/m^{\eta}\left(r\right)\ge S\left(r\right),
\]
for $r$ large enough. So overall we have
\[
\pr{M\left(r\right)\ge e^{3S\left(r\right)}}\le C\cdot\exp\left(-\exp\left(S\left(r\right)\right)\right)\le\exp\left(-S^{2}\left(r\right)\right).
\]

\end{proof}
In the other direction we have the following
\begin{lem}
\label{lem:low_bnd_M(r)}Let $\rho>0$ be sufficiently large. Then
\[
\pr{M\left(\rho\right)\le\exp\left(-S\left(\rho\right)\right)}\le\exp\left(-S\left(\rho\right)\cdot n\left(\rho\right)\right).
\]
\end{lem}
\begin{proof}
By Cauchy's estimate,
\[
|\xi_{n}|a_{n}\rho^{n}\le M\left(\rho\right)\le e^{-S\left(\rho\right)}.
\]
For $n\in N(\rho)$ we have
\[
\pr{|\xi_{n}|\le\left(a_{n}\rho^{n}\right)^{-1}e^{-S\left(\rho\right)}}\le e^{-2S\left(\rho\right)},
\]
and so we get
\[
\pr{M\left(\rho\right)\le\exp\left(-S\left(\rho\right)\right)}\le\prod_{n\in N(\rho)}e^{-2S\left(\rho\right)}=\exp\left(-2\cdot S(\rho)\cdot n\left(\rho\right)\right).
\]

\end{proof}
Notice that we do not assume that $\rho\notin\excpset$.

\subsection{Bounds for the logarithmic derivative\index{Gaussian entire functions!Bounds for the logarithmic derivative}}

We denote by $m$ the normalized angular measure on $\cir r$. In
this section we assume that $f\left(z\right)\ne0$ inside $\disk r$,
and therefore $\log\left|f\right|$ is harmonic there. Under this
condition we have the following bound for the average value of $\left|\log\left|f\right|\right|$:
\begin{lem}
\label{lem:approx_log_int}Let $0<\rho<r$. Outside a set of probability
at most
\[
2\cdot\exp\left(-S(\rho)\cdot n\left(\rho\right)\right),
\]
we have
\[
\intsy_{\cir r}\left|\log|f|\right|\d m\le C\left(1-\frac{\rho}{r}\right)^{-2}\cdot S\left(r\right).
\]
\end{lem}
\begin{proof}
Denote by $P(z,a)$ the Poisson kernel for the disk $\disk r,$ $|z|=r$,
$|a|=\rho$. By Lemma \ref{lem:low_bnd_M(r)}, we can assume that
there is a point $a\in\cir{\rho}$ such that $\log|f(a)|\ge-S\left(\rho\right)$
(discarding an event of probability $\le\exp\left(-S\left(\rho\right)\cdot n\left(\rho\right)\right)$).
Then we have
\[
-S\left(\rho\right)\le\intsy_{\cir r}P(z,a)\log|f(z)|\d{m\left(z\right)},
\]
and hence
\[
\intsy_{\cir r}P(z,a)\log^{-}|f(z)|\d{m\left(z\right)}\le\intsy_{\cir r}P(z,a)\log^{+}|f(z)|\d{m\left(z\right)}+S\left(\rho\right).
\]
For $|z|=r$ and $|a|=\rho$ we have,
\[
\frac{r-\rho}{2r}\le\frac{r-\rho}{r+\rho}\le P(z,a)\le\frac{r+\rho}{r-\rho}\le\frac{2r}{r-\rho}.
\]
By Lemma \ref{lem:upp_bnd_M(r)}, outside a set of very small probability
(of the order $e^{-S^{2}\left(r\right)}$), we have $\log M(r)\le3\cdot S(r)$.
Therefore
\[
\intsy_{\cir r}\log^{+}|f|\d{\mu}\le3\cdot S(r).
\]
Further,
\[
\intsy_{\cir r}\log^{-}|f|\d{\mu}\le\frac{2r}{r-\rho}\cdot S\left(\rho\right)+\frac{12r^{2}}{\left(r-\rho\right)^{2}}\cdot S\left(r\right).
\]
Finally, we get 
\begin{equation}
\intsy_{\cir r}\left|\log|f|\right|\d{\mu}\le\frac{Cr^{2}}{\left(r-\rho\right)^{2}}\cdot S(r)=C\left(1-\frac{\rho}{r}\right)^{-2}\cdot S(r)\label{eq:abs_log_int_upper_bnd}
\end{equation}

\end{proof}
Next we provide an upper bound for the (angular) logarithmic derivative
of $\log\left|f\right|$ inside $\disk r$. 
\begin{lem}
\label{lem:upp_bnd_log_deriv}Let $0<\rho<r$. Then
\[
\left|\frac{d\log\left|f\left(\rho e^{i\phi}\right)\right|}{d\phi}\right|\le C\left(1-\frac{\rho}{r}\right)^{-5}\cdot S\left(r\right)
\]
outside a set of probability at most
\[
2\cdot\exp\left(-S(\rho)\cdot n\left(\rho\right)\right).
\]
\end{lem}
\begin{proof}
We start with Poisson's formula
\[
\log\left|f\left(\rho e^{i\phi}\right)\right|=\intsy_{0}^{2\pi}\frac{r^{2}-\rho^{2}}{\left|re^{i\theta}-\rho e^{i\phi}\right|^{2}}\cdot\log\left|f\left(re^{i\theta}\right)\right|\,\frac{\dd\theta}{2\pi}.
\]
Differentiating under the integral we get
\[
\frac{d\log\left|f\left(\rho e^{i\phi}\right)\right|}{d\phi}=\intsy_{0}^{2\pi}\frac{\rho r\left(r^{2}-\rho^{2}\right)\sin\left(\theta-\phi\right)}{\left|re^{i\theta}-\rho e^{i\phi}\right|^{4}}\cdot\log\left|f\left(re^{i\theta}\right)\right|\,\frac{\dd\theta}{2\pi}.
\]
Taking the absolute value, we obtain
\[
\left|\frac{d\log\left|f\left(\rho e^{i\phi}\right)\right|}{d\phi}\right|\le\frac{C\left(r+\rho\right)}{\left(r-\rho\right)^{3}}\intsy_{0}^{2\pi}\left|\log\left|f\left(re^{i\theta}\right)\right|\right|\,\frac{\dd\theta}{2\pi}.
\]
Using the previous lemma we get the required result.\end{proof}

\section{Proof of Theorem \ref{thm:hole_prob_GAFs} - Lower Bound for $p_{H}(f;r)$
\label{sec:Main-Thm-Lower-Bound}}

The goal of this section is to prove
\begin{prop}
\label{prp:p_H_low_bnd}We have 
\[
p_{H}(f;r)\ge S(r)-Cn\left(r\right)\log S\left(r\right),
\]
with some positive numerical constant $C$.
\end{prop}
In order to find the lower bound for $p_{H}\left(f;r\right)$ we now
assume that $f(z)\ne0$ inside $\disk r$. We take a small $\gamma>0$
(that will depend on $r$), and write $\rho=r\left(1-\gamma\right)$.

The function $\log\left|f\left(z\right)\right|$ is harmonic in $\disk r$,
therefore
\[
\log\left|f\left(0\right)\right|=\intsy_{0}^{2\pi}\log\left|f\left(\rho e^{i\alpha}\right)\right|\,\frac{\dd\alpha}{2\pi}.
\]
Now if we select $n$ points $z_{j}=\rho e^{i\theta_{j}}$, we have
\begin{eqnarray*}
\intsy_{0}^{2\pi}\frac{1}{n}\sum_{j=1}^{n}\log\left|f\left(\rho e^{i\theta_{j}}\cdot e^{i\alpha}\right)\right|\,\frac{\dd\alpha}{2\pi} & = & \frac{1}{n}\sum_{j=1}^{n}\intsy_{0}^{2\pi}\log\left|f\left(\rho e^{i\theta_{j}}\cdot e^{i\alpha}\right)\right|\,\frac{\dd\alpha}{2\pi}\\
 & = & \frac{1}{n}\sum_{j=1}^{n}\log\left|f\left(0\right)\right|\\
 & = & \log\left|f\left(0\right)\right|.
\end{eqnarray*}
Since $\log\left|f\left(z\right)\right|$ is continuous inside $\disk r$,
we conclude that there exists some $\alpha^{*}$ such that
\[
\frac{1}{n}\sum_{j=1}^{n}\log\left|f\left(\rho e^{i\theta_{j}}\cdot e^{i\alpha^{*}}\right)\right|=\log\left|f\left(0\right)\right|.
\]
Let $\Delta\alpha=\frac{c\gamma^{5}}{S\left(r\right)}$. By Lemma
\ref{lem:upp_bnd_log_deriv}, if $\alpha$ satisfies 
\[
\left|\alpha-\alpha^{*}\right|\le\Delta\alpha,
\]
then the logarithmic derivative of $f$ is not too large with high
probability. Discarding this event, we get 
\begin{equation}
\frac{1}{n}\sum_{j=1}^{n}\log\left|f\left(\rho e^{i\theta_{j}}\cdot e^{i\alpha}\right)\right|\le\log\left|f\left(0\right)\right|+1.\label{eq:log_sum_is_small}
\end{equation}
In this section we will show that, if we select the points $\left\{ z_{j}\right\} $
is an appropriate way, then the probability of (\ref{eq:log_sum_is_small})
is of order $\exp\left(-S\left(r\right)\right)$. This will allow
us to prove Proposition \ref{prp:p_H_low_bnd}.

\subsection{Reduction to an estimate of a multivariate Gaussian event}

In this section, we reduce the problem to an estimate of the probability
of an event in some finite-dimensional complex Gaussian space. We
first note that we work in the product space $\left\{ \left(\alpha,\omega\right)\in\left[0,2\pi\right]\times\Omega\right\} $,
where $\alpha$ is chosen uniformly in $\left[0,2\pi\right]$ and
$\Omega$ is the probability space for our Gaussian entire function
$f$ (we denote the probability measures by $m$ and $\mu$, respectively).
We define the following events (all depend on $r$):
\begin{itemize}
\item $H=\left\{ \left(\alpha,\omega\right)\,\mid\, f\left(z\right)\ne0\mbox{ in }\disk r\right\} $
-- the hole event.

\begin{itemize}
\item \medskip{}

\end{itemize}
\item $L=\left\{ \left(\alpha,\omega\right)\,\mid\,\left|\frac{d\log\left|f\left(\rho e^{i\phi}\right)\right|}{d\phi}\right|\le C\cdot\left(1-\frac{\rho}{r}\right)^{-5}\cdot S\left(r\right),\,\forall\phi\in\left[0,2\pi\right]\right\} $ 

\begin{itemize}
\item The non-exceptional event of Lemma \ref{lem:upp_bnd_log_deriv}, w.r.t
$r$ and $\rho$.\medskip{}

\end{itemize}
\item $C=\left\{ \left(\alpha,\omega\right)\,\mid\,\frac{1}{n}\sum\log\left|f\left(z_{j}e^{i\alpha}\right)\right|\le\log\left|f\left(0\right)\right|+1\right\} $.
\medskip{}

\item $D=\left\{ \left(\alpha,\omega\right)\,\mid\,\left|\alpha-\alpha^{\star}\left(\omega\right)\right|<\Delta\alpha\right\} .$\end{itemize}
\begin{rem*}
We defined $\alpha^{\star}$ only when $f\left(z\right)\ne0$ in $\disk r$,
in other cases we can arbitrarily select $\alpha^{\star}=0$.
\end{rem*}
We note that $\alpha^{\star}$ can be chosen to be measurable with
respect to $\Omega$ and that the events $H$ and $L$ do not depend
on the choice of $\alpha$.

Notice that the event $D$ is independent from the event $H\cap L$.
Indeed, since
\[
\ind{H\cap L\cap D}{\alpha,\omega}=\ind{H\cap L}{\alpha,\omega}\cdot\ind D{\alpha,\omega}=\ind{H\cap L}{0,\omega}\cdot\ind D{\alpha,\omega}
\]
we get that
\begin{eqnarray*}
\pr{H\cap L\cap D} & = & \int\left[\intsy_{\left|\alpha-\alpha^{\star}\left(\omega\right)\right|<\Delta\alpha}\d{m\left(\alpha\right)}\right]\cdot\ind{H\cap L}{0,\omega}\d{\mu\left(\omega\right)}\\
 & = & 2\,\Delta\alpha\cdot\int\ind{H\cap L}{0,\omega}\d{\mu\left(\omega\right)}=\pr D\cdot\pr{H\cap L}.
\end{eqnarray*}
Notice also that by the discussion in the previous section we have
$H\cap L\cap D\subset C$. Therefore,
\begin{eqnarray*}
\pr C & \ge & \pr{H\cap L\cap D}=\pr D\cdot\pr{H\cap L}\ge\\
 & \ge & \Delta\alpha\cdot\left(\pr H-\pr{L^{\mathrm{c}}}\right),
\end{eqnarray*}
and so
\[
\pr H\le\frac{1}{\Delta\alpha}\cdot\pr C+\pr{L^{\mathrm{c}}}.
\]
We now introduce the following events:
\begin{itemize}
\item $A=\left\{ \left(\alpha,\omega\right)\,\mid\,\log\left|f\left(0\right)\right|\le\log S\left(r\right)\right\} $.\medskip{}

\item $B=\left\{ \left(\alpha,\omega\right)\,\mid\,{\displaystyle M\left(r\right)=\max_{z\in\disk r}\left|f\left(z\right)\right|}\le e^{3S\left(r\right)}\right\} .$
\end{itemize}
To estimate $\pr C$ from above we use the inequality
\[
\pr C\le\pr{A\cap B\cap C}+\pr{A^{\mathrm{c}}}+\pr{B^{\mathrm{c}}}.
\]
By (\ref{eq:Gaus_prob_large}) from Subsection \ref{sub:Gauss-props}
and Lemma \ref{lem:upp_bnd_M(r)},
\begin{eqnarray*}
\pr{A^{\mathrm{c}}} & \le & e^{-S^{2}\left(r\right)},\\
\pr{B^{\mathrm{c}}} & \le & e^{-cS^{2}\left(r\right)}.
\end{eqnarray*}
Note that these probabilities are very small compared to $\exp\left(-2S\left(r\right)\right)$.

In the next section, we will show that for a good selection of the
set of points $\left\{ z_{j}\right\} $ we have
\begin{equation}
\pr{A\cap B\cap C}\le\exp\left(-S\left(\rho\right)+Cn\left(r\right)\log S\left(r\right)\right).\label{eq:low_bnd_main_est}
\end{equation}
Therefore, we get (using Lemma \ref{lem:upp_bnd_log_deriv}),
\begin{eqnarray*}
\pr H & \le & \pr{L^{\mathrm{c}}}+\frac{1}{\Delta\alpha}\cdot\left(\pr{A\cap B\cap C}+\pr{A^{\mathrm{c}}}+\pr{B^{\mathrm{c}}}\right)\\
 & \le & 2e^{-S\left(\rho\right)\cdot n\left(\rho\right)}+\exp\left(-S\left(\rho\right)+\log\frac{1}{\Delta\alpha}+Cn\left(r\right)\log S\left(r\right)+\O\left(1\right)\right)\\
 & \le & \exp\left(-S\left(\rho\right)+c\log\frac{1}{\gamma}+Cn\left(r\right)\log S\left(r\right)+\O\left(1\right)\right).
\end{eqnarray*}
Finally, by Lemma \ref{lem:S(r)_growth}, if we select $\gamma=\frac{1}{m\left(r\right)}$,
we have
\[
\pr H\le\exp\left(-S\left(r\right)+Cn\left(r\right)\log S\left(r\right)\right),
\]
thus proving Proposition \ref{prp:p_H_low_bnd}.

\subsection{Estimates for the probabilities}

We now turn to finding an upper bound for the probability of the event
$A\cap B\cap C$, that is the event when simultaneously
\begin{eqnarray*}
\log\left|f\left(0\right)\right| & \le & \log S\left(r\right),\\
M\left(r\right)=\max_{z\in\disk r}\left|f\left(z\right)\right| & \le & e^{3S\left(r\right)},\\
\frac{1}{n}\sum\log\left|f\left(z_{j}e^{i\alpha}\right)\right| & \le & \log\left|f\left(0\right)\right|+1\le\log S\left(r\right)+1.
\end{eqnarray*}
Recall that the points $\left\{ z_{j}\right\} $ were until now some
$n$ arbitrary points on $\cir{\rho}$.

For every $\alpha\in\bbr$, the vector $\left(f(e^{i\alpha}z_{1}),\ldots,f(e^{i\alpha}z_{n})\right)$
has the multivariate complex Gaussian distribution with the covariance
matrix\index{covariance matrix}
\[
\Sigma_{ij}=\Cov{f(e^{i\alpha}z_{i}),f(e^{i\alpha}z_{j})}=\Ex(f(e^{i\alpha}z_{i})\overline{f(e^{i\alpha}z_{j})})=\sum a_{k}^{2}\left(z_{i}\bar{z_{j}}\right)^{k}.
\]
We see that the covariance matrix does not depend on rotation by $\alpha$.
By Fubini,
\begin{eqnarray*}
\pr{A\cap B\cap C} & = & \iint\ind{A\cap B\cap C}{\alpha,\omega}\d{m\left(\alpha\right)}\dd\mu\left(\omega\right)=\\
 & = & \intsy\left[\intsy\ind{A\cap B\cap C}{\alpha,\omega}\d{\mu\left(\omega\right)}\right]\d{m\left(\alpha\right)}\\
 & \le & \intsy_{\Omega}\frac{1}{\pi^{n}\det\Sigma}\exp\left(-\zeta^{*}\Sigma^{-1}\zeta\right)\d{\zeta},
\end{eqnarray*}
where $\Omega=\Omega_{r}$ is the following set\textbackslash{}event:
\[
\Omega=\left\{ \left(\zeta_{1},\ldots,\zeta_{n}\right)\,\mid\,\frac{1}{n}\sum_{j=1}^{n}\log\left|\zeta_{j}\right|\le\log S\left(r\right)+1,\,\left|\zeta_{j}\right|\le e^{3S(r)},\,1\le j\le n\right\} .
\]
We have the following upper bound for the probability of the event
$\Omega$:
\[
\pr{\Omega}=\intsy_{\Omega}\frac{1}{\pi^{n}\det\Sigma}\exp\left(-\zeta^{*}\Sigma^{-1}\zeta\right)\d{\zeta}\le\intsy_{\Omega}\frac{1}{\pi^{n}\det\Sigma}\d{\zeta}=\frac{\volcn{\Omega}}{\pi^{n}\det\Sigma}.
\]
We start by finding a good lower bound for $\det\Sigma$. This depends
on a special selection of the points $\left\{ z_{j}\right\} $. 
\begin{lem}
Let $n\in\bbn$ and $j_{1},\ldots,j_{n-1}\in\bbn$. There exist $n$
points $\left\{ z_{j}\right\} $ on $\cir{\rho}$ such that the determinant
of the generalized Vandermonde matrix 
\[
A=\left(\begin{array}{cccc}
1 & z_{1}^{j_{1}} & \ldots & z_{1}^{j_{n-1}}\\
\vdots & \vdots & \ddots & \vdots\\
1 & z_{n}^{j_{1}} & \ldots & z_{n}^{j_{n-1}}
\end{array}\right)
\]
satisfies $\left|\det A\right|\ge\rho^{\sum_{k=1}^{n-1}j_{k}}$.\end{lem}
\begin{proof}
Start with the formal expression for the determinant:
\[
\det A=\sum_{\sigma}\mathrm{sgn}\left(\sigma\right)\prod_{m=1}^{n}z_{m}^{j_{\sigma\left(m\right)-1}},
\]
where the sum is over all the permutations of the set $\left\{ 1,\ldots,n\right\} $
and we set $j_{0}=0$. If we write $z_{j}=\rho e^{i\theta_{j}}$ then
we have
\begin{align*}
 & \intsy_{\mathbb{T}^{n}}\left|\det A\right|^{2}\,\dd\theta_{1}\ldots\dd\theta_{n}\\
 & =\,\intsy_{\mathbb{T}^{n}}\left(\sum_{\sigma}\mathrm{sgn}\left(\sigma\right)\prod_{m=1}^{n}z_{m}^{j_{\sigma\left(m\right)-1}}\right)\cdot\left(\sum_{\tau}\mathrm{sgn}\left(\tau\right)\prod_{m=1}^{n}\bar{z}_{m}^{j_{\tau\left(m\right)-1}}\right)\,\dd\theta_{1}\ldots\dd\theta_{n}\\
 & =\,\rho^{2\cdot\sum_{m=1}^{n-1}j_{k}}\left({\textstyle \sum}_{1}+{\textstyle \sum}_{2}\right),
\end{align*}
where
\begin{eqnarray*}
{\textstyle \sum}_{1} & = & \sum_{\sigma}\intsy_{\mathbb{T}^{n}}1\,\dd\theta_{1}\ldots\dd\theta_{n}=\left(2\pi\right)^{n}\cdot n!,\\
{\textstyle \sum}_{2} & = & \sum_{\sigma\ne\tau}\mathrm{sgn}\left(\sigma\right)\cdot\mathrm{sgn}\left(\tau\right)\left[\intsy_{\mathbb{T}^{n}}\prod_{m=1}^{n}\exp\left(i\theta_{m}\left(j_{\sigma\left(m\right)-1}-j_{\tau\left(m\right)-1}\right)\right)\,\dd\theta_{1}\ldots\dd\theta_{n}\right].
\end{eqnarray*}
Notice that $j_{k}\ne j_{l}$ if $k\ne l$. Since $\sigma\ne\tau$
in the sum $\sum_{2}$, for at least one $m\in\left\{ 1,\ldots,n\right\} $
we have that $j_{\sigma\left(m\right)-1}\ne j_{\tau\left(m\right)-1}$.
The numbers $j_{k}$ are all integers and therefore we get that $\sum_{2}=0$.
Thus we conclude that there exist some $n$ points $\left\{ z_{j}\right\} $
on $\cir{\rho}$ such that
\[
\left|\det A\right|\ge\rho^{\sum_{m=1}^{n-1}j_{k}}.
\]

\end{proof}
We now have the following\index{covariance matrix!determinant estimates}
\begin{cor}
\label{cor:low_bnd_det_sigma}Using the configuration of $n=n\left(r\right)$
points $\left\{ z_{j}\right\} $ given by the previous lemma, we have
\[
\log\left(\det\Sigma\right)\ge S(r).
\]
\end{cor}
\begin{proof}
Notice that we can represent $\Sigma$ in the form
\[
\Sigma=V\cdot V^{*},
\]
where 
\[
V=\left(\begin{matrix}a_{0} & a_{1}\cdot z_{1} & \dots & a_{j}\cdot z_{1}^{j} & \dots\\
\vdots & \vdots & \vdots & \vdots & \dots\\
a_{0} & a_{1}\cdot z_{n} & \dots & a_{j}\cdot z_{n}^{j} & \dots
\end{matrix}\right).
\]
We can estimate $\det\Sigma$ by projecting $V$ on $N(r)=\left\{ a_{0},a_{j_{1}},\ldots,a_{j_{n-1}}\right\} $
coordinates (let's denote this projection by $P$). Since $\det\Sigma$
is the square of the product of the singular values of $V$, and these
values are only reduced by the projection, we have
\[
\det\Sigma\ge\left(\det PV\right)^{2}=\left|\begin{array}{cccc}
a_{0} & a_{j_{1}}z_{1}^{j_{1}} & \ldots & a_{j_{n-1}}z_{1}^{j_{n-1}}\\
\vdots & \vdots & \ddots & \vdots\\
a_{0} & a_{j_{1}}z_{n}^{j_{1}} & \ldots & a_{j_{n-1}}z_{n}^{j_{n-1}}
\end{array}\right|^{2},
\]
and so, by the previous lemma,
\begin{eqnarray*}
\det\Sigma & \ge & \prod_{j\in N\left(r\right)}a_{j}^{2}\cdot\left|\begin{array}{cccc}
1 & z_{1}^{j_{1}} & \ldots & z_{1}^{j_{n-1}}\\
\vdots & \vdots & \ddots & \vdots\\
1 & z_{n}^{j_{1}} & \ldots & z_{n}^{j_{n-1}}
\end{array}\right|^{2}\\
 & \ge & \prod_{j\in N\left(r\right)}a_{j}^{2}\cdot r^{2\cdot\sum_{j\in N\left(r\right)}j}\\
 & = & \prod_{j\in N\left(r\right)}a_{j}^{2}r^{2j}=\exp\left(S\left(r\right)\right)
\end{eqnarray*}

\end{proof}
We now want to estimate 
\[
I=\frac{1}{\pi^{n}}\cdot\volcn{\Omega},
\]
with respect to the Lebesgue measure on $\bbc^{n}$, where
\[
\Omega=\left\{ \zeta\in\bbc^{n}\,\mid\,\frac{1}{n}\sum_{j=1}^{n}\log\left|\zeta_{j}\right|\le\log S\left(r\right)+1\mbox{ and}\left|\zeta_{j}\right|\le e^{3S\left(r\right)},\,1\le j\le n\right\} .
\]

\begin{lem}
\label{lem:prob_integral_upper_bnd}We have
\[
I\le\exp\left(Cn\left(r\right)\log S\left(r\right)\right).
\]
\end{lem}
\begin{proof}
We use the following abbreviated notation: $n=n\left(r\right),S=S\left(r\right)$.
Let $\zeta=\left(\zeta_{1},\ldots,\zeta_{n}\right)\in\Omega$. For
every $\zeta_{j}$ we find the minimal integer $k_{j}$ such that
$\left|\zeta_{j}\right|\le e^{k_{j}}$. Note that $k_{j}\le4S$. We
also have
\[
\sum_{j=1}^{n}\left(k_{j}-1\right)\le\sum_{j=1}^{n}\log\left|\zeta_{j}\right|\le n\left(\log S+1\right),
\]
and so
\[
\sum_{j=1}^{n}k_{j}\le n\log S+2n.
\]
If it happens that $\sum_{j}k_{j}<0$, then we just increase some
of the negative $k_{j}$'s until the sum is $0$. We now observe that
each $\zeta\in\Omega$ is contained in some polydisk of the form $\left\{ \xi\,\mid\,\left|\xi_{j}\right|\le e^{k_{j}},\,1\le j\le n\right\} $.
What is left is to estimate the total volume of all polydisks for
which
\[
0\le\sum_{j=1}^{n}k_{j}\le n\log S+2n
\]
and
\[
k_{j}\le4S.
\]
It is clear that the volume of each polydisk is bounded by 
\[
\pi^{n}\exp\left(\sum_{j}k_{j}\right)\le\pi^{n}\exp\left(n\log S+2n\right).
\]
Also, the value of each $k_{j}$ is between $-4S\left(n-1\right)$
and $4S$, so the total number of polydisks is at most $\left(4Sn\right)^{n}.$
Overall, the total volume of all the polydisks is at most
\begin{eqnarray}
\pi^{n}\exp\left(n\log S+2n\right)\cdot\left(4Sn\right)^{n} & \le & \pi^{n}\exp\left(2n\log S+Cn\log n\right)\label{eq:precise_upp_bnd_integral}\\
 & \le & \pi^{n}\exp\left(Cn\log S\right).\nonumber 
\end{eqnarray}
Therefore we get the required bound for $I$.
\end{proof}
Now, the estimate (\ref{eq:low_bnd_main_est}) follows, since the
original points $\left\{ z_{j}\right\} $ satisfy $\left|z_{j}\right|=\rho$,
and by Corollary \ref{cor:low_bnd_det_sigma} and Lemma \ref{lem:prob_integral_upper_bnd}:
\[
\pr{\Omega}\le\exp\left(-S\left(\rho\right)+Cn\left(r\right)\log S\left(r\right)\right).
\]
This completes the proof of Proposition \ref{prp:p_H_low_bnd}.

\section{On The Exceptional Set in Theorem \ref{thm:hole_prob_GAFs}\label{sec:hole_prob_nec_exc_set}}

We construct a GAF $f\left(z\right)=\sum_{n\ge0}\xi_{n}a_{n}z^{n}$
and a set $E$ of infinite Lebesgue measure such that
\[
p_{H}(f;r)\ge2\cdot S\left(r\right)-C\sqrt{S\left(r\right)},\quad r\in E.
\]
This shows that in general, an exceptional set in the statement of
Theorem \ref{thm:hole_prob_GAFs} is necessary. The idea of the proof
is that indices $n$ such that $a_{n}r^{n}=1$ do not increase the
value of $S\left(r\right)$, while they reduce the probability of
a hole event (since they reduce the probability that the free term
will dominate the others). Therefore, if there are sufficiently many
indices with this property, then the hole probability can be reduced.
It is intuitively clear that this phenomenon cannot occur for too
many values of $r$. In fact, by Theorem \ref{thm:hole_prob_GAFs}
(and specifically Lemma \ref{lem:upp_bnd_log_deriv_spec_case}), the
set of such `bad' $r$'s is of a finite logarithmic measure.

\subsection{Parameters of the construction}

The construction is relatively simple, we start by fixing some increasing
sequence of radii $\left\{ r_{m}\right\} $. At each radius $r_{m}$
we set a `block' of terms $a_{j}r^{j}$ to be equal to $1$, the length
of the $m$th block is equal to $l_{m}$ (which is an increasing sequence
of integers). All the blocks are consecutive, so that this construction
determines all the coefficients of the function $f$. We note, that
at the radius $r_{m}$ only the terms in the blocks $1$ to $m-1$
contribute to $S\left(r_{m}\right)$ (since all the other terms will
be of modulus equal or smaller than $1$).

We then choose a sequence $\delta_{m}\in\left(0,1\right)$, such that
the terms in the $m$th block are still sufficiently large at the
radius $r_{m}e^{-\delta_{m}}$. Now we repeat the proof of the lower
bound for the hole probability, with some changes. The most important
of these changes, is that instead of using Corollary \ref{cor:low_bnd_det_sigma}
to get a lower bound for the determinant of the covariance matrix,
we use a more precise estimate. We note that the corollary has advantage
over the lemma in the case where the coefficients $a_{n}$ are not
regular (for example when the function $f$ is lacunary).

We will fix the following values of the parameters, for $m\ge1$:

\begin{eqnarray*}
r_{0}=1, &  & r_{m}=\exp\left(a^{m}\right),\\
k_{0}=0, &  & k_{m}=\left[\exp\left(b^{m}\right)\right],\\
l_{0}=1, &  & l_{m}=k_{m}-k_{m-1}\\
\delta_{0}=\tfrac{1}{2}, &  & \delta_{m}=\frac{1}{2k_{m}},
\end{eqnarray*}
where $\left[x\right]$ denotes the integral part of $x$. We will
choose the values of $a,b\in\left(1,\infty\right)$ later in the proof.

For every $m\in\bbn$ we choose $a_{j}$ for all the indices $j\in\left\{ k_{m-1}+1,\ldots,k_{m}\right\} $
such that they will satisfy $a_{j}r_{m}^{j}=1$. In addition, we set
$a_{0}=1$.

We now turn to look at the properties of this function. First it is
worth to mention that this choice of coefficients guarantees that
$f$ is an entire function, since $a_{j}^{1/j}=\frac{1}{r_{m}}$ for
$j$'s in the $m$th block and $r_{m}\to\infty$ as $m\to\infty$.

\subsection{Some properties of the function $f$}

We start by analyzing the size of $S\left(r\right)$ at and near the
radius $r_{m}$.
\begin{claim}
\label{clm:bnds_for_S(r)}Let $\eta\in\left(0,1\right)$ and $\mu\in\bbn$
sufficiently large. We have the following bounds
\begin{eqnarray*}
S\left(r_{\mu}e^{-\eta}\right) & \ge & S\left(r_{\mu}\right)-2\eta\exp\left(2b^{\mu-1}\right),\\
S\left(r_{\mu}e^{\eta}\right) & \le & S\left(r_{\mu}\right)+2\eta\exp\left(2b^{\mu}\right)
\end{eqnarray*}
and
\begin{eqnarray*}
S\left(r_{\mu}e^{-\eta}\right) & \ge & \frac{\left(a-1\right)}{2}\cdot a^{\mu-1}\exp\left(2b^{\mu-1}\right)-2\eta\exp\left(2b^{\mu-1}\right),\\
S\left(r_{\mu}e^{\eta}\right) & \le & 2\left(a-1\right)a^{\mu-1}\exp\left(2b^{\mu-1}\right)+2\eta\exp\left(2b^{\mu}\right).
\end{eqnarray*}
\end{claim}
\begin{proof}
We recall that
\[
S\left(r\right)=2\cdot\sum_{n\ge0}\log^{+}\left(a_{n}r^{n}\right),
\]
which, using the special properties of our construction, we can rewrite
in the following way
\begin{eqnarray}
S\left(r\right) & = & 2\cdot\sum_{m=1}^{\infty}\sum_{j=k_{m-1}+1}^{k_{m}}\log^{+}\left(a_{j}r_{m}^{j}\left(\frac{r}{r_{m}}\right)^{j}\right)=2\cdot\sum_{m=1}^{\infty}\left[\log^{+}\left(\frac{r}{r_{m}}\right)\sum_{j=k_{m-1}+1}^{k_{m}}j\right]\nonumber \\
 & = & \sum_{m=1}^{\infty}\left[\log^{+}\left(\frac{r}{r_{m}}\right)\left(k_{m}^{2}+k_{m}-k_{m-1}^{2}-k_{m-1}\right)\right].\label{eq:expression_for_S(r)}
\end{eqnarray}
By our choice of the sequence $r_{m}$ we note that $r_{\mu-1}<r_{\mu}e^{-\eta}$,
for sufficiently large $m$. We conclude that if $r=r_{\mu}e^{-\eta}$
only the first $\mu-1$ terms in (\ref{eq:expression_for_S(r)}) do
not vanish, while if $r=r_{\mu}e^{\eta}$ the first $\mu$ contribute
to $S\left(r\right)$. This leads to the following estimates
\begin{eqnarray*}
S\left(r_{\mu}e^{-\eta}\right) & = & \sum_{m=1}^{\mu-1}\left[\left(\log r_{\mu}-\log r_{m}-\eta\right)\left(k_{m}^{2}+k_{m}-k_{m-1}^{2}-k_{m-1}\right)\right]\\
 & = & S\left(r_{\mu}\right)-\eta\sum_{m=1}^{\mu-1}\left(k_{m}^{2}+k_{m}-k_{m-1}^{2}-k_{m-1}\right)\\
 & \ge & S\left(r_{\mu}\right)-2\eta k_{\mu-1}^{2}\ge S\left(r_{\mu}\right)-2\eta\exp\left(2b^{\mu-1}\right).
\end{eqnarray*}
and 
\begin{eqnarray*}
S\left(r_{\mu}e^{\eta}\right) & = & \sum_{m=1}^{\mu}\left[\left(\log r_{\mu}-\log r_{m}+\eta\right)\left(k_{m}^{2}+k_{m}-k_{m-1}^{2}-k_{m-1}\right)\right]\\
 & \le & S\left(r_{\mu}\right)+\eta\sum_{m=1}^{\mu}\left(k_{m}^{2}+k_{m}-k_{m-1}^{2}-k_{m-1}\right)\\
 & \le & S\left(r_{\mu}\right)+2\eta k_{\mu}^{2}\le S\left(r_{\mu}\right)+2\eta\exp\left(2b^{\mu}\right).
\end{eqnarray*}
The other estimates follow
\begin{eqnarray*}
S\left(r_{\mu}e^{-\eta}\right) & \ge & \left(a^{\mu}-a^{\mu-1}\right)\left(\exp\left(2b^{\mu-1}\right)-\exp\left(2b^{\mu-2}\right)\right)-2\eta\exp\left(2b^{\mu-1}\right)\\
 & \ge & \frac{\left(a-1\right)}{2}\cdot a^{\mu-1}\exp\left(2b^{\mu-1}\right)-2\eta\exp\left(2b^{\mu-1}\right),
\end{eqnarray*}
and
\begin{eqnarray*}
S\left(r_{\mu}e^{\eta}\right) & \le & 2\left(a^{\mu}-a^{\mu-1}\right)\exp\left(2b^{\mu-1}\right)+2\eta\exp\left(2b^{\mu}\right)\\
 & = & 2\left(a-1\right)a^{\mu-1}\exp\left(2b^{\mu-1}\right)+2\eta\exp\left(2b^{\mu}\right).
\end{eqnarray*}

\end{proof}
We recall the following definitions
\begin{eqnarray*}
N_{\delta}\left(r\right) & = & \left\{ n\,:a_{n}r^{n}\ge\exp\left(-\delta n\right)\right\} ,\\
N\left(r\right) & = & N_{0}\left(r\right),\\
n\left(r\right) & = & \#N\left(r\right).
\end{eqnarray*}
It is clear from the construction that $n\left(r_{m}\right)=k_{m}$.
Moreover, for any $\delta\in\left[0,1\right]$ we have that $\#N_{\delta}\left(r\right)=n\left(r\right)$
for $r$ sufficiently large. Indeed if $r\in\left[r_{m-1},r_{m}\right],$
then it is enough to consider indices $j\in\left\{ k_{m}+1,k_{m+1}\right\} $,
and for them
\[
a_{j}r^{j}=a_{j}r_{m+1}^{j}\left(\frac{r}{r_{m+1}}\right)^{j}\le\left(\frac{r_{m}}{r_{m+1}}\right)^{j}<e^{-\delta j},
\]
where the last inequality holds (for $m$ sufficiently large) by our
choice of the sequence $r_{m}$. A similar argument shows that $\#N_{\delta}\left(re^{\eta}\right)=n\left(r\right)$
for any $\eta\in\left(0,1\right)$ and any $\delta\in\left[0,1\right]$.

In order to improve the lower bound estimate we need a `corrected'
version of Lemma \ref{lem:upp_bnd_M(r)}, which will hold for every
sufficiently large $r\in\left[r_{m},r_{m}e^{\eta}\right]$ (for our
specific function $f$). For the convenience of the reader we reprove
the lemma with the necessary changes. We recall our notation
\[
M\left(r\right)=\max_{z\in\disk r}\left|f\left(z\right)\right|.
\]

\begin{lem}
\textup{\label{lem:upp_bnd_M(r)_spec_case}}Let $\eta\in\left(0,1\right)$
and let $m\in\bbn$ be sufficiently large. For any $r\in\left[r_{m},r_{m}e^{\eta}\right]$\textup{
\[
\pr{M\left(r\right)\ge e^{3S\left(r\right)}}\le C\cdot\exp\left(-\exp\left(S\left(r\right)\right)\right)\le e^{-S^{2}\left(r\right)}.
\]
}\end{lem}
\begin{proof}
We set $\Nalt r=\left\{ n<S^{b/2}\left(r\right)\right\} $. Notice
that, by Claim \ref{clm:bnds_for_S(r)} and by the previous discussion
we have that $S^{b/2}\left(r\right)\ge S^{b/2}\left(r_{m}\right)\ge N_{1}\left(r_{m}e^{\eta}\right)=n\left(r_{m}\right)=k_{m}$
for every $\eta\in\left[0,1\right)$. The rest of the proof is very
similar to that of Lemma \ref{lem:upp_bnd_M(r)}. We define the event
$\Omega_{r}$ as the intersection of the events $\mbox{{\rm \mbox{(i)}}}$
and $\mbox{(ii)}$, where
\[
\begin{array}{ll}
\mbox{{\rm \mbox{(i)}}}: & {\displaystyle \bigcap_{n\in\Nalt r}}\mbox{{\rm \mbox{(i)}}}_{n},\\
\mbox{{\rm \mbox{(ii)}}}: & {\displaystyle \bigcap_{n\in\left(\Nalt r\right)^{\mathrm{c}}}}\mbox{{\rm \mbox{(ii)}}}_{n},
\end{array}
\]
and
\[
\begin{array}{ll}
\mbox{{\rm \mbox{(i)}}}_{n}: & |\xi_{n}|\le\left(a_{n}r^{n}\right)^{-1}e^{2S\left(r\right)},\\
\mbox{{\rm \mbox{(ii)}}}_{n}: & |\xi_{n}|\le\exp\left(\frac{n}{2}\right).
\end{array}
\]
We have the following estimate for $M\left(r\right)$:
\begin{eqnarray*}
|f(z)| & \le & \sum_{n\in\Nalt r}|\xi_{n}|a_{n}r^{n}+\sum_{n\in\left(\widetilde{N}\left(r\right)\right)^{\mathrm{c}}}|\xi_{n}|a_{n}r^{n}\\
 & \le & S^{b/2}\left(r\right)\cdot e^{2S\left(r\right)}+\sum_{n\ge S^{b/2}\left(r\right)}e^{\frac{n}{2}}\cdot e^{-n}\\
 & \le & S^{b/2}\left(r\right)\cdot e^{2S\left(r\right)}+\O\left(1\right)\\
 & \le & e^{3S\left(r\right)},
\end{eqnarray*}
provided that $r$ is sufficiently large.

Now we estimate the probability of the complement of $\Omega_{r}$.
We have:
\begin{eqnarray*}
\pr{|\xi_{n}|\ge\frac{e^{2S\left(r\right)}}{a_{n}r^{n}}} & = & \exp\left(-\frac{e^{4S\left(r\right)}}{\left(a_{n}r^{n}\right)^{2}}\right)\le\exp\left(-e^{2S\left(r\right)}\right),\\
\pr{|\xi_{n}|\ge e^{\frac{n}{2}}} & = & \exp\left(-\exp\left(n\right)\right).
\end{eqnarray*}
By the union bound,
\begin{eqnarray*}
\pr{\mbox{{\rm \mbox{(i)}}}^{\mathrm{c}}} & \le & S^{b/2}\left(r\right)\cdot\exp\left(-\exp\left(2S\left(r\right)\right)\right),
\end{eqnarray*}
and
\[
\pr{\mbox{{\rm \mbox{(ii)}}}^{\mathrm{c}}}\le\sum_{n\ge S^{b/2}\left(r\right)}\exp\left(-\exp\left(n\right)\right)\le C\cdot\exp\left(-\exp\left(S^{b/2}\left(r\right)\right)\right),
\]
since the first term in the sum above dominates the others. So overall
we have
\[
\pr{M\left(r\right)\ge e^{3S\left(r\right)}}\le C\cdot\exp\left(-\exp\left(S^{b/2}\left(r\right)\right)\right)\le\exp\left(-S^{2}\left(r\right)\right),
\]
again for $r$ large enough.
\end{proof}
To make the reading of this proof easier, we quote here the statements
of some of the lemmas that are used in the proof of the lower bound.
The proofs can be found in Section \ref{sec:bnds_for_GAFs}.
\begin{lem}
\label{lem:low_bnd_M(r)_spec_case}Let $\rho>0$ be sufficiently large.
Then
\[
\pr{M\left(\rho\right)\le\exp\left(-S\left(\rho\right)\right)}\le\exp\left(-S\left(\rho\right)\cdot n\left(\rho\right)\right).
\]

\end{lem}
Note the proof of this lemma do not require any modification. The
next two lemmas requires the use of Lemma \ref{lem:upp_bnd_M(r)_spec_case}
and Lemma \ref{lem:low_bnd_M(r)_spec_case} instead of the original
lemmas used. The rest of the proofs is the same.
\begin{lem}
\label{lem:approx_log_int_spec_case}Let $r\in\left[r_{m},r_{m}e^{\eta}\right]$
and $0<\rho<r$. Outside a set of probability at most
\[
2\cdot\exp\left(-S(\rho)\cdot n\left(\rho\right)\right),
\]
we have
\[
\intsy_{\cir r}\left|\log|f|\right|\d m\le C\left(1-\frac{\rho}{r}\right)^{-2}\cdot S\left(r\right).
\]

\end{lem}
\begin{lem}
\label{lem:upp_bnd_log_deriv_spec_case}Let $r\in\left[r_{m},r_{m}e^{\eta}\right]$
and $0<\rho<r$. Then
\[
\left|\frac{d\log\left|f\left(\rho e^{i\phi}\right)\right|}{d\phi}\right|\le C\left(1-\frac{\rho}{r}\right)^{-5}\cdot S\left(r\right)
\]
outside a set of probability at most
\[
2\cdot\exp\left(-S(\rho)\cdot n\left(\rho\right)\right).
\]

\end{lem}

\subsection{Proof of the lower bound for $f$}

We repeat step by step the general proof of the lower bound. The main
difference being the better lower bound estimate for the determinant
of the covariance matrix. Let $r\in\left[r_{m},r_{m}e^{\eta}\right]$,
where $\eta=\eta_{m}$. In the original proof, we got the following
upper bound for the probability of the hole event

\begin{eqnarray*}
\pr H & \le & \pr{L^{\mathrm{c}}}+\frac{CS\left(r\right)}{\gamma^{5}}\cdot\left(\pr{A\cap B\cap C}+\pr{A^{\mathrm{c}}}+\pr{B^{\mathrm{c}}}\right).
\end{eqnarray*}
Here $\rho=r\left(1-\gamma\right)=re^{-\delta}$, and the following
events are used
\begin{itemize}
\item $A$ - $\log\left|f\left(0\right)\right|\le\log S\left(r\right)$,
\item $B$ - ${\displaystyle M\left(r\right)=\max_{z\in\disk r}\left|f\left(z\right)\right|}\le e^{3S\left(r\right)}$,
\item $C$ - $\frac{1}{n}\sum\log\left|f\left(z_{j}e^{i\alpha}\right)\right|\le\log\left|f\left(0\right)\right|+1$, 
\item $L$ - $\left|\frac{d\log\left|f\left(\rho e^{i\phi}\right)\right|}{d\phi}\right|\le C\cdot\left(1-\frac{\rho}{r}\right)^{-5}\cdot S\left(r\right)$.
\end{itemize}
By (\ref{eq:Gaus_prob_large}), Lemma \ref{lem:upp_bnd_M(r)_spec_case}
and Lemma \ref{lem:upp_bnd_log_deriv_spec_case} we have the following
bounds
\begin{eqnarray*}
\pr{A^{\mathrm{c}}} & \le & e^{-S^{2}\left(r\right)},\\
\pr{B^{\mathrm{c}}} & \le & e^{-cS^{2}\left(r\right)},\\
\pr{L^{\mathrm{c}}} & \le & 2e^{-S\left(\rho\right)\cdot n\left(\rho\right)}.
\end{eqnarray*}
In addition, like in the original proof we have
\[
\pr{A\cap B\cap C}\le\frac{\volcn{\Omega}}{\pi^{n}\det\Sigma}.
\]
Here $\Omega=\Omega_{r}$ is the following set (or event)
\[
\Omega=\left\{ \left(\zeta_{1},\ldots,\zeta_{n}\right)\,\mid\,\frac{1}{n}\sum_{j=1}^{n}\log\left|\zeta_{j}\right|\le\log S\left(r\right)+1,\,\left|\zeta_{j}\right|\le e^{3S(r)},\,1\le j\le n\right\} ,
\]
and $\Sigma$ is the covariance matrix of the Gaussian vector $\left(f(e^{i\alpha}z_{1}),\ldots,f(e^{i\alpha}z_{n})\right)$.
Unlike the original proof, we choose the points $z_{j}$ to be equally
distributed on $\cir{\rho}$, we also choose $n$ to be smaller than
$n\left(r\right)=k_{m}$. We set $n=\alpha k_{m}$ with some $\alpha=\alpha_{m}\in\left(0,1\right)$
to be chosen later.

The upper bound for the volume of $\Omega$ is unchanged, but we need
to use the precise estimate (\ref{eq:precise_upp_bnd_integral}) that
was proved in Lemma \ref{eq:precise_upp_bnd_integral}, that is
\[
\frac{\volcn{\Omega}}{\pi^{n}}\le\exp\left(2n\log S\left(r\right)+n\log n+Cn\right).
\]
It remains to show that the lower bound for the determinant can be
significantly improved beyond the original estimate $\log\left(\det\Sigma\right)\ge S(\rho)$.

\subsubsection{Some facts about circulant matrices}

A matrix $C$ of the following form
\[
C=\left(\begin{array}{ccccc}
c_{0} & c_{1} & \ldots & c_{n-2} & c_{n-1}\\
c_{n-1} & c_{0} &  & c_{n-3} & c_{n-2}\\
\vdots &  & \ddots &  & \vdots\\
c_{2} & c_{3} &  & c_{0} & c_{1}\\
c_{1} & c_{2} & \cdots & c_{n-1} & c_{0}
\end{array}\right)
\]
is called \textit{circulant}\index{circulant matrix} (it is a special
kind of a Toeplitz matrix). It is known that the eigenvalues of $C$
are given by
\[
\lambda_{j}\left(C\right)=c_{0}+c_{1}\omega_{j}+\ldots+c_{n-1}\omega_{j}^{n-1},
\]
where $\omega_{j}=\exp\left(2\pi ij/n\right)$, $j\in\left\{ 0,\ldots,n-1\right\} $
are the $n$th roots of unity.

Let $f\left(z\right)=\sum_{n\ge0}\xi_{n}a_{n}z^{n}$ be a GAF, with
$\left\{ a_{n}\right\} \in\bbr^{+}$ and let $\left\{ z_{j}\right\} $
be $N$ equidistributed points on the circle $\cir r$, say $z_{j}=re^{2\pi i\cdot j/N}$,
$j\in\left\{ 0,\ldots,N-1\right\} $. We denote by $\Sigma$ the covariance
matrix of the complex Gaussian vector $\left(f\left(z_{0}\right),\ldots,f\left(z_{N-1}\right)\right)$,
thus 
\[
\Sigma_{jk}=\Cov{f\left(z_{j}\right),f\left(z_{k}\right)}=\Ex\left\{ f\left(z_{j}\right)\overline{f\left(z_{k}\right)}\right\} =\sum_{n\ge0}a_{n}^{2}r^{2n}e^{2\pi i\left(j-k\right)n/N}.
\]
Then we have
\begin{lem}
\label{lem:eigenvals_of_cov_matrix}The eigenvalues of the covariance
matrix $\Sigma$ are
\[
\lambda_{k}\left(\Sigma\right)=N\cdot\sum_{l\ge0}a_{k+lN}^{2}r^{2\left(k+lN\right)},\quad k\in\left\{ 0,\ldots,N-1\right\} .
\]
\end{lem}
\begin{proof}
We note that since the $z_{j}$s are equidistributed on $\cir r$
it follows that the matrix $\Sigma$ is circulant. Indeed $\Sigma_{jk}=C_{k-j\,\mbox{mod }N}$,
where
\[
C_{m}=\sum_{n\ge0}a_{n}^{2}r^{2n}e^{-2\pi imn/N},\quad m\in\left\{ 0,\ldots,N-1\right\} .
\]
Therefore for $l\in\left\{ 0,\ldots,N-1\right\} $
\begin{eqnarray*}
\lambda_{l}\left(\Sigma\right) & = & \sum_{m=0}^{N-1}C_{m}\omega_{l}^{m}=\sum_{m=0}^{N-1}\left[\sum_{n\ge0}a_{n}^{2}r^{2n}e^{-2\pi imn/N}\right]e^{2\pi ilm/N}\\
 & = & \sum_{n\ge0}a_{n}^{2}r^{2n}\sum_{m=0}^{N-1}\exp\left(\frac{2\pi i\left(l-n\right)}{N}\cdot m\right).
\end{eqnarray*}
Since 
\[
\sum_{m=0}^{N-1}\exp\left(\frac{2\pi i\left(l-n\right)}{N}\cdot m\right)=\begin{cases}
N & ,\,\mbox{if }l=n\,\mbox{mod }N\\
0 & ,\,\mbox{otherwise,}
\end{cases}
\]
we get the result.
\end{proof}

\subsubsection{Lower bound for the determinant of $\Sigma$ and conditions on $n$}

We set $\rho=re^{-\delta_{m}}$. Now we can use Lemma \ref{lem:eigenvals_of_cov_matrix}
to improve our estimate for the determinant of $\Sigma$.\index{covariance matrix!determinant estimates}
\begin{claim}
We have
\[
\det\Sigma\ge\exp\left(S\left(\rho\right)\right)\exp\left(\alpha k_{m}\left(\log k_{m}-2\right)\right).
\]
\end{claim}
\begin{proof}
By Lemma \ref{lem:eigenvals_of_cov_matrix} the eigenvalues of $\Sigma$
are given by
\[
\lambda_{k}\left(\Sigma\right)=n\cdot\sum_{l=0}^{\infty}a_{k+ln}^{2}\rho^{2\left(k+ln\right)},\quad k\in\left\{ 0,\ldots,n-1\right\} .
\]
We trivially have for $k\in\left\{ 0,\ldots,k_{m-1}\right\} $
\[
\lambda_{k}\left(\Sigma\right)\ge n\cdot a_{k}^{2}\rho^{2k}.
\]
For $j\in\left\{ k_{m-1}+1,\ldots,k_{m}\right\} $ we have
\[
a_{j}^{2}\rho^{2j}=a_{j}^{2}\left(re^{-\delta_{m}}\right)^{2j}\ge a_{j}^{2}r_{m}^{2j}\cdot e^{-2\delta_{m}j}=e^{-2\delta_{m}j},
\]
thus by our choice $\delta_{m}=\frac{1}{2k_{m}}$, we have $a_{j}^{2}\rho^{2j}\ge\frac{1}{e}$.
Thus, for $k\in\left\{ k_{m-1}+1,\ldots,n\right\} $ we have 
\[
\lambda_{k}\left(\Sigma\right)\ge\frac{n}{e}\sum_{l=0}^{\infty}\indf_{\left\{ k+ln\le k_{m}\right\} }\ge\frac{n}{e}\cdot\left(\frac{k_{m}}{n}-1\right)\ge\frac{k_{m}}{e^{2}},
\]
therefore,
\begin{eqnarray*}
\det\Sigma & \ge & \prod_{k\in\left\{ 0,\ldots,k_{m-1}\right\} }\left(n\cdot a_{k}^{2}\rho^{2k}\right)\cdot\prod_{k=k_{m-1}+1}^{n}\frac{k_{m}}{e^{2}}\\
 & = & \exp\left(S\left(\rho\right)\right)\cdot\exp\left(k_{m-1}\log n+\left(\log k_{m}-2\right)\left(n-k_{m-1}\right)\right),
\end{eqnarray*}
After setting $n=\alpha k_{m}$ we get
\begin{eqnarray*}
\det\Sigma & \ge & \exp\left(S\left(\rho\right)\right)\cdot\exp\left(k_{m-1}\log\left(\alpha k_{m}\right)+\left(\log k_{m}-2\right)\left(\alpha k_{m}-k_{m-1}\right)\right)\\
 & \ge & \exp\left(S\left(\rho\right)\right)\exp\left(\alpha k_{m}\left(\log k_{m}-2\right)\right).
\end{eqnarray*}

\end{proof}
In order to improve on the original estimate we want to choose $k_{m}$
such that
\[
\alpha k_{m}\log k_{m}\ge\epsym_{1}S\left(r\right)\ge\epsym_{1}S\left(\rho\right),
\]
for some numerical constant $\epsym_{1}>0$. This gives the condition
\[
\alpha\ge\frac{\epsym_{1}S\left(r\right)}{k_{m}\log k_{m}}.
\]
A non-trivial bound for the hole probability, we imply an upper bound
for $\alpha$. We have 
\begin{eqnarray*}
\frac{\volcn{\Omega}}{\pi^{n}\det\Sigma} & \le & \frac{\exp\left(2n\log S\left(r\right)+n\log n+Cn\right)}{\exp\left(S\left(\rho\right)+\alpha k_{m}\left(\log k_{m}-2\right)\right)}\\
 & \le & \exp\left(\alpha k_{m}\left(2\log S\left(r\right)+\log\left(\alpha\right)+C\right)-S\left(\rho\right)\right).
\end{eqnarray*}
The natural condition on $\alpha$ is therefore
\[
2\log S\left(r\right)+\log\left(\alpha\right)+C\le-\epsym_{2}\log k_{m},
\]
with some numerical constant $\epsym_{2}>0$. Decreasing $\epsym_{2}$
by a little, we require
\[
\log\left(\frac{1}{\alpha}\right)\ge2\log S\left(r\right)+\epsym_{2}\log k_{m},
\]
which we can write as
\[
\alpha\le\frac{1}{S^{2}\left(r\right)k_{m}^{\epsym_{2}}}\le1.
\]
Combining the two conditions on $\alpha$, we get
\[
\frac{\epsym_{1}S\left(r\right)}{k_{m}\log k_{m}}\le\frac{1}{S^{2}\left(r\right)k_{m}^{\epsym_{2}}}\Rightarrow\epsym_{1}S^{3}\left(r\right)\le k_{m}^{1-\epsym_{2}}\log k_{m}.
\]

\subsubsection{Choosing the values of the parameters and finishing the proof}

We now choose the parameters $a,b,\epsym_{1},\epsym_{2},\eta$ in
an appropriate way. We have
\[
k_{m}^{1-\epsym_{2}}\log k_{m}\ge\frac{1}{2}\exp\left(\left(1-\epsym_{2}\right)b^{m}\right)b^{m},
\]
and
\[
S\left(r\right)\le2\left(a-1\right)a^{m-1}\exp\left(2b^{m-1}\right)+2\eta\exp\left(2b^{m}\right).
\]
We first choose $\eta=\exp\left(-\left(2-\tfrac{2}{b}\right)b^{m}\right)$
by balancing, which leads to the bound
\[
S\left(r\right)\le2a^{m}\exp\left(2b^{m-1}\right),\quad r\in\left[r_{m},r_{m}e^{\eta}\right].
\]
Choosing $\epsym_{1}=2$ and $\epsym_{2}=\frac{1}{2}$ we are left
to select $a,b$ such that
\[
16a^{3m}\exp\left(\frac{6}{b}\cdot b^{m}\right)\le\frac{b^{m}}{2}\exp\left(\frac{1}{2}b^{m}\right).
\]
It is clear that any $b>12$ will be sufficient (for $m$ sufficiently
large), so let us choose $b=20$. We are left with some freedom to
chose the value of $a$, in order to get an exceptional set of infinite
Lebesgue measure we will require
\[
\eta=\exp\left(-\tfrac{19}{10}\cdot20^{m}\right)\ge\frac{1}{r_{m}}=\exp\left(-a^{m}\right),
\]
so $a=21$ will do. Notice that by this choice $\eta=o\left(\delta_{m}\right)$
and thus $\rho<r_{m}$.
\begin{rem*}
We can check the validity of these choices, by noticing that
\begin{eqnarray*}
n & = & \alpha k_{m}\ge\frac{\epsym_{1}S\left(r\right)}{\log k_{m}}\ge20^{-m}\cdot\left(20^{m}\cdot\exp\left(2\cdot20^{m-1}\right)\right)\\
 & = & \exp\left(2\cdot20^{m-1}\right)\ge\exp\left(20^{m-1}\right)\ge k_{m-1},
\end{eqnarray*}
(here we used Claim \ref{clm:bnds_for_S(r)}).
\end{rem*}
Now, lets us write $r=r_{m}e^{\rho}$, with $\rho\in\left[0,\eta\right]$.
By the first two bounds in Claim \ref{clm:bnds_for_S(r)} we have
\begin{eqnarray*}
S\left(\rho\right) & = & S\left(re^{-\delta_{m}}\right)=S\left(r_{m}e^{\rho-\delta_{m}}\right)\ge S\left(r_{m}\right)-2\left(\rho-\delta_{m}\right)\exp\left(2\cdot20^{m-1}\right)\\
 & \ge & S\left(r\right)-2\left(\rho-\delta_{m}\right)\exp\left(2\cdot20^{m-1}\right)-2\eta\exp\left(2\cdot20^{m}\right)\\
 & \ge & S\left(r\right)-4\eta\exp\left(2\cdot20^{m}\right)\ge S\left(r\right)-4\exp\left(20^{m-1}\right)\\
 & \ge & S\left(r\right)-\sqrt{S\left(r\right)}.
\end{eqnarray*}
Wrapping things up
\begin{eqnarray*}
\frac{\volcn{\Omega}}{\pi^{n}\det\Sigma} & \le & \exp\left(-\epsym_{2}\alpha k_{m}\log k_{m}-S\left(\rho\right)\right)\le\exp\left(-\epsym_{1}\epsym_{2}S\left(r\right)-S\left(\rho\right)\right)\\
 & \le & \exp\left(-2S\left(r\right)+\sqrt{S\left(r\right)}\right).
\end{eqnarray*}
Now, we use the fact that $\gamma\ge2\delta_{m}=\frac{1}{k_{m}}\ge\exp\left(-20^{m}\right)\ge\frac{1}{S\left(r\right)^{10}}$,
to conclude that
\[
\frac{CS\left(r\right)}{\gamma^{5}}\le S\left(r\right)^{100},
\]
and so finally 
\begin{eqnarray*}
\pr H & \le & \exp\left(-S\left(\rho\right)\cdot n\left(\rho\right)\right)+S\left(r\right)^{100}\cdot\exp\left(-2S\left(r\right)+C\sqrt{S\left(r\right)}\right)\\
 & \le & \exp\left(-2S\left(r\right)+C\sqrt{S\left(r\right)}\right).
\end{eqnarray*}
We remind that this upper bound holds for any $r\in\left[r_{m},r_{m}e^{\eta}\right]$. 
\begin{rem*}
The exceptional set is of infinite Lebesgue measure. This follows
from the fact that the following series
\[
\sum_{m=1}^{\infty}r_{m}\left(e^{\eta}-1\right)\ge\sum_{m=1}^{\infty}\exp\left(21^{m}-\tfrac{19}{10}\cdot20^{m}\right)
\]
are divergent.\end{rem*}

\section{(Counter-) Example for non-Gaussian random variables - Proof of Theorem
\ref{thm:No_Hole_Prob_for_Non_Gaus} \label{sec:No_Hole_for_Non_Gaus}}

Let $K\subset\bbc$ be a compact set such that $0\notin K$. Let $g\left(z\right)=\sum_{n\ge0}\phi_{n}\frac{z^{n}}{\sqrt{n!}}$
be an entire function, with a sequence of coefficients $\phi_{n}\in K$.
We want to prove that there exists an $r_{0}=r_{0}(K)<\infty$ so
that $g(z)$ must vanish somewhere in the disk $\left\{ |z|\le r_{0}\right\} $.

Suppose that the theorem is false, that is, there exists a sequence
of entire functions
\[
g_{k}(z)=\sum_{n=0}^{\infty}\phi_{n,k}\frac{z^{n}}{\sqrt{n!}},\qquad\phi_{n,k}\in K,
\]
and a sequence $r_{k}\to\infty$, such that $g_{k}$ does not vanish
in $\disk{r_{k}}$. Since $K$ is a compact set, we can find a subsequence,
also denoted by $\left\{ g_{k}\right\} $, such that $\phi_{n,k}\to\phi_{n}$
for each $n\in\bbn$. It is easy to see that the sequence $\left\{ g_{k}\right\} $
converges locally uniformly to a limiting function $f$. Since $0\notin K$,
the limiting function $f$ is not identically zero. Now, using Hurwitz's
theorem (see \cite[pg. 178]{Ahl}), $f$ does not vanish in any disk
$\disk{r_{k}}$; i.e. is does not vanish in the whole complex plane.

By known formulas expressing the order and type of entire functions
in terms of its Taylor coefficients (see, for instance \cite[pg. 6]{Le2}),
$f$ has order $2$ and type $\frac{1}{2}$. Since it does not vanish
on $\bbc$, by Hadamard's theorem, $f(z)=\exp\left(\alpha z^{2}+\beta z+\gamma\right),$
with complex constants $\alpha,\beta,\gamma$, $|\alpha|=\frac{1}{2}$.

We want to prove that we cannot get a function $f$ of this form,
using coefficients from the set $K$. We will use the asymptotics
of the coefficients of $f$ to prove this. Denoting the Taylor coefficients
of $f(z)$ by $b_{n}$, it is sufficient to show that the product
\[
|b_{n}|\cdot\sqrt{n!}
\]
is not bounded between any two positive constants. 

We first study the asymptotics of function of the form \eqref{eq:non_Gaus_f_def}.
Using Stirling's approximation we get
\begin{multline}
\sqrt{n!}=\left(\sqrt{2\pi n}\left(\frac{n}{e}\right)^{n}\left(1+\O\left(\frac{1}{n}\right)\right)\right)^{1/2}\\
=\left(2\pi n\right)^{1/4}\left(\frac{n}{e}\right)^{n/2}\left(1+\O\left(\frac{1}{n}\right)\right).\label{eq:sqrt_n_fact_asymp}
\end{multline}
The asymptotics of $f$ are not as simple. Using rotation and scaling,
we can assume $\alpha=\frac{1}{2}$ and $\gamma=1$; moreover, it
is easy to see that $\beta$ should not be zero. Therefore the problem
is reduced to the study of the asymptotics of
\begin{equation}
f(z)=\exp\left(\frac{1}{2}\cdot z^{2}+\beta\cdot z\right)=\sum_{n=0}^{\infty}b_{n}\left(\beta\right)\cdot z^{n},\label{eq:limit_func_form}
\end{equation}
with $\beta\ne0$, possibly a complex number. A standard application
of the saddle point method%
\footnote{For the details see the version of the paper \cite{Ni1} found in
the arXiv.%
} shows that
\begin{equation}
b_{n-1}(\beta)=C_{\beta}\cdot\left(\frac{e}{n}\right)^{\frac{n}{2}}\cdot\left(e^{\beta\sqrt{n}}+\left(-1\right)^{n}e^{-\beta\sqrt{n}}\right)\cdot\left(1+\O\left(\frac{1}{\sqrt{n}}\right)\right),\label{eq:limit_func_coeff_asymp}
\end{equation}
where $C_{\beta}$ is some constant. We can see that this is not the
same rate of decay as in \eqref{eq:sqrt_n_fact_asymp} for $n\to\infty$.
We arrive at a contradiction, which completes the proof of Theorem
\ref{thm:No_Hole_Prob_for_Non_Gaus}.

\section{Some open problems\index{open problems}}

One can ask if the error term in Theorem \ref{thm:hole_prob_GAFs}
is optimal for a regular sequence of coefficients $\left\{ a_{n}\right\} $.
This could already be interesting for the GEF. It is not clear if
the lower bound for $p_{H,K}\left(f;r\right)$ requires an exceptional
set. Another possibility is to try to find a `smoothed' version of
the function $S\left(r\right)$, that will give rise to a result with
no exceptional set $\excpset$.

A natural question in the case of the GEF is to study the hole probability
in more general domains. Let $K\subset\bbc$ be some connected compact
set, with non-empty interior, and denote by $rK$ the homothety by
$r$ of $K$. We propose the following conjecture%
\footnote{This problem is well-defined since the distribution of the zero set
of the GEF is invariant with respect to the plane isometries.%
}\index{hole probability!general domains}
\begin{conjecture}
For large values of $r$,
\[
p_{H,K}\left(f;r\right)=\log^{-}\pr{f\left(z\right)\ne0\mbox{ in }rK}=S\left(c\left(K\right)\cdot r\right)\left(1+o\left(1\right)\right),
\]
where $c\left(K\right)$ is the capacity (transfinite diameter) of
$K$.
\end{conjecture}
The reason for proposing this conjecture is that in the case of the
plane invariant model, the evaluation of the determinant in Corollary
\ref{cor:low_bnd_det_sigma} leads to the Vandermonde determinant
\[
\left|\begin{array}{cccc}
1 & z_{1} & \ldots & z_{1}^{n-1}\\
\vdots & \vdots & \vdots & \vdots\\
1 & z_{n} & \ldots & z_{n}^{n-1}
\end{array}\right|^{2}=\prod_{1\le i\ne j\le n}\left|z_{i}-z_{j}\right|^{2}.
\]
Taking the maximum of this expression over the possible positions
of the $z_{j}$'s we can approximate the (appropriate power of) the
capacity $c\left(K\right)$. However, as the current method heavily
uses the Taylor expansion, which works primarily in the case of the
disk, some new ideas will be required.

Another very natural question is the asymptotics of the hole probability
for disconnected domains. The most basic case is when our domain is
the union of two disjoint disks. Notice that there are three parameters
for this problem, the scale, the radius of one of the disk, and the
distance between the center of the disks. In particular, it is not
clear if the event that $f\left(z\right)\ne0$ in one disk is `almost
independent' from the event that $f\left(z\right)\ne0$ in the other
disk, if they are sufficiently far from each other.

The third type of questions for the GEF is the question of hole probability
for `thin domains', that is, domains which are scaled only in one
direction. For example, it is not clear what are the asymptotics of
the hole probability in a rectangular domain with sides $r$ and $1$,
as $r$ tends to infinity. For an upper bound estimate that might
be sharp see \cite[Theorem 4.1]{GKP}. Another interesting question
is the asymptotics of the conditional hole probability, that is\index{hole probability!conditional}
\begin{problem}
What is the probability that there are no zeros in the disk $\disk R$,
conditioned on the event that there are no zeros in $\disk r$, for
$R>r$?
\end{problem}
Still another question is to consider non-Gaussian random variables.
We are interested in characterizing the asymptotics of the hole probability
in terms of properties of the distribution function of the random
variables. This can be viewed as a more accurate version of some results
of Offord \cite{Of2} (see also \cite[sect. 7.1]{BKPV}).

It should be mentioned that there are some partial results related
to the asymptotics of the hole probability for GAFs in the unit disk
(\cite{BNPS}). However, it is still an open problem to find the analog
of Theorem \ref{thm:hole_prob_GAFs} for general GAFs in the unit
disk.

It would be interesting to generalize Theorem \ref{thm:hole_prob_GAFs}
to large deviations of the number of zeros (for the GEF, see the papers
\cite{NSV1,ST3}).
\newpage{}\thispagestyle{empty}$\,$\markright{References}\newpage{}

\newpage{}\thispagestyle{empty}$\,$$\,$\newpage{}\markright{Index}\printindex{}
\end{document}